\documentclass[11pt]{article}
\usepackage[top=23.4mm, bottom=23.4mm, left=23.4mm, right=23.4mm]{geometry}
\usepackage{graphicx,dcolumn,bm}
\usepackage{amssymb,amstext,amsmath}

\usepackage{graphicx}
\usepackage{dcolumn}
\usepackage{bm}
\usepackage{amssymb}
\usepackage{multirow}
\usepackage{bigstrut}
\usepackage{makecell,rotating}
\usepackage{mathrsfs}
\usepackage{booktabs}
\usepackage{threeparttable}
\usepackage{multirow}
\usepackage{subfigure}
\usepackage{colortbl}
\usepackage{epsfig}
\usepackage{threeparttable}
\usepackage{chngpage}
\usepackage{float}
\usepackage{amstext}
\usepackage{amsmath}
\usepackage{amsthm}
\usepackage{pdflscape}
\usepackage{authblk}
\usepackage{hyperref}
\usepackage{rotating}
\usepackage[graphicx]{realboxes}
\usepackage{adjustbox}
\usepackage{setspace}
\usepackage{caption}
\usepackage{xcolor}
\usepackage{xr}
\usepackage{caption}
\usepackage{lineno}
\usepackage{algorithmic}
\usepackage{algorithm}
\usepackage{textcomp}
\usepackage{appendix}
\allowdisplaybreaks
\newtheorem{lemma}{Lemma}
\newtheorem{corollary}{Corollary}
\newtheorem{theorem}{Theorem}
\newtheorem {proposition}{Proposition}
\newtheorem{remark}{Remark}
\theoremstyle{definition}

\newtheorem{definition}
{Definition}[section]

\DeclareMathOperator{\rank}{rank}
\DeclareMathOperator{\nulll}{null}

\DeclareMathOperator{\argmin}{argmin}
\DeclareMathOperator{\spann}{span}
\DeclareMathOperator{\diag}{diag}

\usepackage{enumitem}
\setitemize[1]{itemsep=1pt,partopsep=0pt,parsep=\parskip,topsep=4pt}
\title{Node bipartition for rigidity and localization of networks with heterogeneous sensing 
}
\author{Yongjie Liu$^1$, Gangshan Jing$^{2*}$, and Long Wang$^{1*}$}

\date{
    $^1$ Center for Systems and Control, College of Engineering, Peking University, Beijing 100871, China\\
    $^2$ School of Automation, Chongqing University, Chongqing 400044, China\\
    $^*$ \footnotesize Corresponding authors. E-mail: jinggangshan@cqu.edu.cn, longwang@pku.edu.cn\\[2ex]
}

\begin{document}
\maketitle
\begin{abstract}
Graph rigidity theory is an important tool for examining the solvability of sensor network localization (SNL) problems, and ensuring global convergence of localization algorithms. Along this direction, diverse measurements such as signed angle (SA) and ratio of distance (RoD) have been considered. However, little is known about how the bipartition of nodes based on perceptual abilities affects the rigidity property of the network. In this paper, we study the rigidity and localization of networks with heterogeneous nodes, namely, two types of sensors measuring SA and RoD, respectively. Interestingly, the rigidity property is shown to be strongly dependent on the bipartition of nodes, and exhibits a duality. Moreover, an SA-RoD constrained network can be uniquely determined up to uniform rotations, translations, and scalings (global SA-RoD rigidity) even if it is neither SA rigid nor RoD rigid. A scalable approach to construction of globally SA-RoD rigid frameworks is proposed. Localizability analysis and localization algorithm synthesis are both conducted based on  weaker network topology conditions, compared with SA- or RoD-based SNL approaches. Numerical simulations are worked out to validate the theoretical results.
\end{abstract}

\section{Introduction}
Graph rigidity theory investigates whether the shape of a framework can be uniquely determined by certain geometric constraints.  It has been widely applied in computer-aided design (CAD)\cite{servatius1999constraining},  formation control \cite{suttner2018formation,jing2018weak,chen2019controlling}, and  network localization \cite{anderson2010formal,lyjb05,peters2015sensor}. Recent years have witnessed substantial progress in this field, including the extension to embeddings into  surfaces \cite{nixon2014characterization}, the characterization of globally rigid graphs \cite{jackson2005connected,jordan2021note}, and globally linked pairs \cite{garamvolgyi2024partial}. To adapt to practical scenarios with diverse geometric constraints, traditional distance rigidity theory has been extended to bearing rigidity \cite{lyjb4}, angle rigidity \cite{lyjbb8}, signed angle (SA) rigidity \cite{lyjb9,lyjbb10}, and ratio-of-distance (RoD) rigidity \cite{lyjbb8.1}, etc. However, these theories are merely based on one type of geometric constraints. 

For sensor networks with mixed measurements, various rigidity theories mixing different types of constraints have been developed. \cite{TE} mixes constraints of distances and bearings, and establishes a combinatorial characterization of rigidity. \cite{CAOr} mixes constraints of bearings and RoDs to propose a bearing-ratio-of-distance (B-RoD) rigidity theory, and shows the equivalence between rigidity and graph connectivity.  \cite{lyjbb17} mainly mixes distances (or angles) and signed normalized areas (or volumes) to address the flip and flex ambiguity. In \cite{lyjb014,lyjb013}, the authors mix angles and displacements to study the shape uniqueness problem. In \cite{clinch2023global}, the author provides a sufficient and necessary condition for global rigidity of 2-dimensional direction-length frameworks with connected rigidity matroids. In summary, mixed geometric constraints (distance, bearing, angle, SA, RoD, displacement) in most existing works are defined through edges, with little consideration given to vertex-based attributes. However, different types of edge constraints often stem from different attributes of vertices, a situation frequently seen in reality.

Building on aforementioned developments, these mixed rigidity theories have been extensively applied to sensor network localization (SNL), which aims to localize a network when the positions of some nodes (anchors) and relative measurements are provided.  In \cite{lyjb0012}, the author provides graphical criteria for localizability and designs SNL algorithms under hybrid distance-bearing information. In \cite{lyjb012,lyjb014,lyjb013, TIE2025}, the authors study SNL by constructing linear displacement constraints for locations of nodes using mixed types of local relative measurements.  Despite these interesting works on SNL with mixed measurements, the advantages of heterogeneous nodes have yet to be fully explored. More specifically, how to assign the sensing capabilities for the nodes to ensure network localizability under weaker topology conditions has rarely been discussed. 

This paper explores the influence of node bipartition on rigidity and localization of sensor networks with heterogeneous SA and RoD sensing, where each node typically accesses only one type of measurements. 
Notably, in contrast to the well-established bipartite rigidity \cite{KNN}, the nodes are partitioned according to their sensing capabilities, and links may exist within each subset of the bipartition in our setting. Our analysis reveals that node bipartition is crucial for ensuring framework rigidity, network localizability, and convergence of localization algorithms. An additional benefit is that heterogeneous sensing with appropriate bipartitions can reduce the number of edges required for rigidity compared with pure SA or RoD sensing \cite{lyjbb8.1,lyjb9}. These aspects represent the most distinctive features of our work.

The main contributions of this work can be summarized as follows.

1) The SA-RoD rigidity theory in $\mathbb{R}^{2}$ is established. It is shown that the shape of a framework can be uniquely determined by SA-RoD constraints up to uniform rotations, translations, and scalings if it is globally SA-RoD rigid (Theorem \ref{shp}). Furthermore, we find an upper bound on the rank of the infinitesimal SA-RoD rigidity matrix to constrain the minimal number of edges required for rigidity (Lemma \ref{infi}).

2) To examine how bipartition affects the rigidity property of a heterogeneously constrained framework, we introduce the notion of connectivity for bipartition-induced SA and RoD index sets (Definition $\ref{def:con}$). It is proved that an SA-RoD constrained framework inherits the global rigidity property of the pure SA or RoD constrained framework when the corresponding SA or RoD index set is connected (Theorems \ref{thmm} and \ref{thmm1}).  An algorithm of finding a suitable bipartition for SA (RoD) connectivity is proposed (Algorithm \ref{alg:buildpar}). In addition, a novel result of bipartition-induced duality is established for infinitesimal SA-RoD rigidity (Corollary $\ref{thm10}$). Furthermore, we discover that global SA-RoD rigidity is not a generic property of the underlying graph with a specific bipartition, while (infinitesimal ) SA-RoD rigidity is a generic property (Theorem $\ref{gen1}$).

3) Based on the SA-RoD rigidity theory, a vertex addition approach, including the design of node bipartition, is proposed to construct globally SA-RoD rigid frameworks. We show that  frameworks generated by SA-RoD orderings are globally SA-RoD rigid (Theorem $\ref{order1}$). Moreover, we find that under suitable bipartitions, the number of edges required to guarantee rigidity of frameworks can be less than those in pure SA/RoD rigidity theory \cite{lyjbb8.1,lyjb9}. Specifically, we construct a class of minimally globally SA-RoD rigid frameworks with $n$ vertices and $\lceil{\frac{3n-4}{2}\rceil}$ edges. 

4) Necessary and sufficient conditions for network localizability are established. The SNL problem with SA and RoD measurements is reformulated as an edge-based optimization problem by utilizing the cycle basis matrix (Theorem $\ref{equv1}$). The connectivity of the SA (RoD) index sets is shown to be essential for decomposing the edge-based problem into two subproblems of finding distances and bearings, respectively. In particular, if the SA index set is connected over the underlying graph, the centralized SNL problem reduces to solving linear equations (Theorem $\ref{snllp}$). Furthermore, a two-stage distributed SNL protocol with global convergence is proposed for sensor networks with connected SA index sets (Theorem \ref{con_dis_SNL}). 

The remainder of this paper is organized as follows. Section 2 introduces SA-RoD rigidity theory, including the effects of node bipartition, shape uniqueness, generic property, and construction of globally SA-RoD rigid frameworks. Section 3 formulates the SA-RoD-based SNL problem. An edge-based analysis via a graph cycle basis is performed. Furthermore, localization algorithms are discussed based on the connectivity of constraints. Section 4 provides some simulating examples. Section 5 concludes this article with some suggestions on future research.

Notations: In this paper, $\mathbb{R}$ represents the set of real numbers; $\mathbb{C}$ is the set of complex numbers; $\mathbb{Z}$ is the set of integers; $I_{N}\in\mathbb{R}^{N\times N}$ is the unit matrix; $\mathbb{R}^{N}$ denotes the vector space of dimension $N$; $\mathbb{S}^1$ is the unit circle in $\mathbb{R}^2$; For a set $T$, $|T|$ denotes the number of elements in $T$; For $a\in\mathbb{R}$, a ceiling function of $a$ is denoted by $\lceil{a\rceil}$; The norm of $s=(s_1,...,s_n)^\top\in\mathbb{R}^{n}$ is denoted by $||s||\triangleq[\sum\limits_{k=1}^{n}{s_{k}^{2}}]^{\frac{1}{2}}$; For a matrix $X\in\mathbb{R}^{M\times N}$, the null space and rank of $X$ are denoted by $\nulll(X)$ and $\rank(X)$, respectively; $X^{\top}$ denotes the transpose of $X$; $\mathbf{1}_{n}\triangleq (1,\ldots,1)^{\top}\in\mathbb{R}^{n}$; For a square matrix $X\in\mathbb{R}^{N\times N}$, its determinant is denoted by $\det(X)$; $\otimes$ is the Kroneck product; $\odot$ is the Khatri-Rao product; $\mathcal{R}(\theta)\in\mathbb{R}^{2\times 2}$ is the rotation matrix associated with $\theta\in[0,2\pi)$.
 We assume that $\mathcal{G}=(\mathcal{V},\mathcal{E})$ is an undirected graph, where $\mathcal{V}=\left\{ {1,2,\ldots ,n} \right\}$ is the vertex set, $\mathcal{E}\subseteq \mathcal{V}\times\mathcal{V}$ is the edge set with $m$ edges. Let $\deg(i)$ be the degree of vertex $i$, and $\mathcal{N}_{i}\triangleq\{j\in\mathcal{V}:(j,i)\in\mathcal{E}\}$ be the set of neighbors of vertex $i$ in graph $\mathcal{G}$. A complete graph with $n$ vertices is denoted by $\mathcal{K}_{n}$. If the number of vertices is irrelevant, we simply denote it by $\mathcal{K}$. Let $\Delta ijk$ denote the triangle with vertices $i,j,k$, while $\Diamond{ijkl}$ denotes the quadrilateral with vertices $i,j,k,l$.

\section{Bipartition-based SA-RoD rigidity theory}
In this section, we will define some concepts about SA-RoD rigidity. Based on these definitions, the effects of node bipartition, the shape determination problem, and generic properties of the defined rigidity concepts will be studied, respectively. Furthermore, construction approaches of globally SA-RoD rigid frameworks will be proposed.
\subsection{Formulation of SA-RoD rigidity}
Rigidity theories based on pure constraints such as distances, bearings, RoDs, and SAs are briefly introduced in Appendix $\ref{pre}$.  Throughout this paper,  $(\mathcal{G}(A,D),p)$ is used to represent the framework with heterogeneous vertices, where $\mathcal{G}(A,D)=(\mathcal{V}_{A}\cup\mathcal{V}_{D},\mathcal{E})$ is the underlying graph with the {\it node bipartition} $\mathcal{V}=\mathcal{V}_{A}\cup\mathcal{V}_{D}$, $\mathcal{V}_A$ and $\mathcal{V}_D$ correspond to the sets of nodes sensing SA and RoD measurements, respectively, $|\mathcal{V}|=n\geq 3$, $|\mathcal{E}|=m$, and $p={{(p_{1}^{\top},p_{2}^{\top},\ldots p_{n}^{\top})}^{\top}}\in\mathbb{R}^{2n}$ is a configuration. We define four triplet index sets $\mathcal{T}_{A}$, $\mathcal{T}_{D}$, $\mathcal{T}_{\mathcal{G}}$, and $\mathcal{T}_{\mathcal{K}}$ as
$\mathcal{T}_{A}=\{(u,v,w)\in\mathcal{V}^{3}:u\in\mathcal{V}_{A},(u,v),(u,w)\in\mathcal{E},v<w\}$, 
$\mathcal{T}_{D}=\{(u,v,w)\in\mathcal{V}^{3}:u\in\mathcal{V}_{D},(u,v),(u,w)\in\mathcal{E},v<w\}$,
$\mathcal{T}_{\mathcal{G}}=\mathcal{T}_{A}\cup \mathcal{T}_{D}$,  and $\mathcal{T}_{\mathcal{K}}=\{(u,v,w)\in\mathcal{V}^{3}:(u,v),(u,w)\in\mathcal{E}_{\mathcal{K}}, v<w\}$, respectively. For $(r,s,t)\in\mathcal{T}_{A}$ and $(i,j,k)\in\mathcal{T}_{D}$, $\alpha_{rst}$ is the SA defined by $(\ref{def:SA})$ , and $\kappa_{ijk}$ is the RoD defined by $(\ref{def:RoD})$ (with $l=1$).  
Throughout this work, the following assumptions are made:

\textit{Assumption 1:} Different vertices in the framework are not collocated.
 
\textit{Assumption 2:} $\mathcal{V}_{A}\cap\mathcal{V}_{D}=\emptyset, \mathcal{V}_{A}\neq\emptyset, \mathcal{V}_{D}\neq\emptyset$. 

Assumption 1 is commonly used in the literature concerning rigidity theory and localization problems\cite{lyjb014,lyjb013,lyjbb10,TIE2025}. From Assumption 2, we know that $\mathcal{V}_{D}$ and $\mathcal{V}_{A}$ are non-empty and constitute a non-trivial partition of the vertex set $\mathcal{V}$. Furthermore, $\mathcal{T}_{A}\neq\emptyset$, $\mathcal{T}_{D}\neq\emptyset$, and $\mathcal{T}_{A}\cap \mathcal{T}_{D}=\emptyset$. 
To demonstrate the nontrivial properties of networks with heterogeneous SA and RoD nodes, we present a motivating example. 

\textbf{A motivating example:} Consider non-degenerate quadrilateral frameworks as shown in Figure $\ref{lf13}$. Note that the quadrilateral framework is neither RoD rigid nor SA rigid since there exist nontrivial deformations preserving RoD or SA constraints as plotted in Figures $\ref{lf13}(a)$ and $\ref{lf13}(b)$. When $\mathcal{V}_{A}=\{1\}$ as shown in Figure \ref{lf13}(c), 
there exists a flipping ambiguity for vertex 3. 
Similar results hold for frameworks in Figures $\ref{lf13}(d)$ and $\ref{lf13}(e)$. 
\begin{figure}[!htbp]
\centerline{\includegraphics[width=\columnwidth]{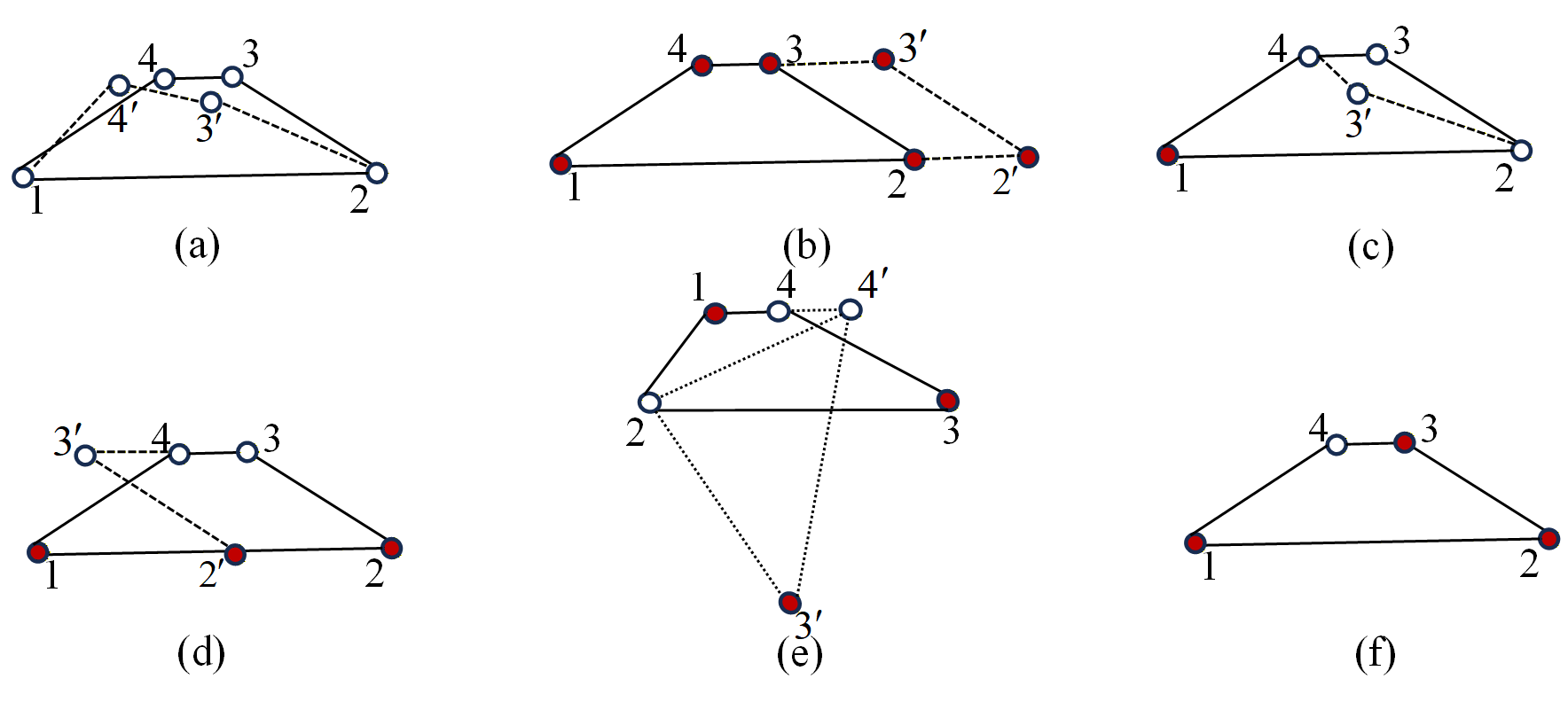}}
\caption{Examples of quadrilateral frameworks. The dots in red represent vertices in $\mathcal{V}_{A}$, while hollow dots represent vertices in $\mathcal{V}_{D}$. (a) A pure RoD constrained framework that is not infinitesimally RoD rigid. (b) A pure SA constrained framework that is not infinitesimally SA rigid. (c) A framework with $\mathcal{V}_{A}=\{1\}$ whose shape can not be uniquely determined under the given bipartition. (d) A framework with $\mathcal{V}_{A}=\{1,2\}$  whose shape can not be uniquely determined under the given bipartition. The coordinates of vertices $1,2,3,4,2^{\prime},3^{\prime}$ are $(0,0),(4,0),(3,1),(2,1),(2,0),(1,1)$, respectively. (e) A framework with $\mathcal{V}_{A}=\{1,3\}$  whose shape can not be uniquely determined under the given bipartition. The coordinates of vertices $1,2,3,4,3^{\prime},4^{\prime}$ are $(1,\sqrt{3}),(0,0),(4,0),(2,\sqrt{3}),(2,-2\sqrt{3}),(3,\sqrt{3})$, respectively. (f) A framework with $\mathcal{V}_{A}=\{1,2,3\}$  whose shape can be uniquely determined under the given bipartition.}
\label{lf13}
\end{figure}
However, the shape of the quadrilateral framework constrained by SAs and RoDs in Figure $\ref{lf13}(f)$ can be uniquely determined up to uniform rotations, translations, and scalings as long as the three vertices in $\mathcal{V}_{A}$ are non-collinear. 
This simple example indicates that mixing SA and RoD perceptions appropriately can guarantee the shape uniqueness even if it is neither RoD rigid nor SA rigid. 

We will introduce definitions in SA-RoD rigidity theory that are analogous to those in distance rigidity theory. The  {\it SA-RoD rigidity function} is defined as 
\begin{equation}\label{def:rigf}
f_{\mathcal{G}(A,D)}(p)\triangleq(\ldots,\alpha_{rst},\ldots,\kappa_{ijk},\ldots) ^{\top}\in\mathbb{R}^{|\mathcal{T}_{\mathcal{G}}|},\ (r,s,t)\in\mathcal{T}_{A},\ (i,j,k)\in\mathcal{T}_{D}.
\end{equation}

\begin{definition}
A framework $(\mathcal{G}(A,D),p)$ is said to be SA-RoD rigid if there exists an open neighborhood $U_{p}$ of $p$ such that $$f^{-1}_{\mathcal{G}(A,D)}(f_{\mathcal{G}(A,D)}(p))\cap U_{p}=f^{-1}_{\mathcal{K}(A,D)}(f_{\mathcal{K}(A,D)}(p))\cap U_{p}.$$ Otherwise, it is said to be flexible. $(\mathcal{G}(A,D),p)$ is called globally SA-RoD rigid if $$f^{-1}_{\mathcal{G}(A,D)}(f_{\mathcal{G}(A,D)}(p))=f^{-1}_{\mathcal{K}(A,D)}(f_{\mathcal{K}(A,D)}(p)).$$
\end{definition}

The  {\it SA-RoD rigidity matrix} is defined as
 \begin{equation}
R_{\mathcal{G}(A,D)}(p)\triangleq\frac{\partial f_{\mathcal{G}(A,D)}(p)}{\partial p}\in\mathbb{R}^{|\mathcal{T_{\mathcal{G}}}|\times 2n}.
\end{equation}
The  {\it infinitesimal SA-RoD motion} refers to the motion $\delta p$ maintaining the rigidity function $f_{\mathcal{G}(A,D)}(p)$, i.e.,
\begin{equation}
 \dot{f}_{\mathcal{G}(A,D)}(p)=\frac{\partial f_{\mathcal{G}(A,D)}}{\partial p}\delta p=R_{\mathcal{G}(A,D)}(p)\delta p=0,
\end{equation} where $\delta p=(\delta p_{1}^{\top},\delta p_{2}^{\top},\ldots,\delta p_{n}^{\top})^{\top},\ \delta p_{i}=\dot{p}_{i}$ is the velocity of $p_{i}$. Note that uniform rotations, translations, and scalings always preserve the SA-RoD constraints of a framework. Therefore, an infinitesimal SA-RoD motion is said to be trivial if it corresponds to a combination of uniform translations, rotations and scalings of the framework. According to \cite[Lemma 3.1]{lyjbb8}, the dimension of the trivial SA-RoD infinitesimal motion space in $\mathbb{R}^{2}$ is 4. 

\begin{definition}\label{def:inr}
A framework $(\mathcal{G}(A,D),p)$ is said to be infinitesimally SA-RoD rigid if all infinitesimal
SA-RoD motions of the framework are trivial.
\end{definition}

To obtain a concise form of the rigidity matrix, let $e_{ij}=p_{j}-p_{i},\ b_{ij}=\frac{e_{ij}}{||e_{ij}||}$. For derivatives of SA, we have 
\begin{multline}\label{eq:sar}
\dot{\alpha}_{rst}=-[\frac{b_{rs}}{||e_{rs}||}-\frac{b_{rt}}{||e_{rt}||}]^{\top}\mathcal{R}(\pi/2)\dot{p}_{r}+\frac{b_{rs}^{\top}}{||e_{rs}||}\mathcal{R}(\pi/2)\dot{p}_{s}
-\frac{b_{rt}^{\top}}{||e_{rt}||}\mathcal{R}(\pi/2)\dot{p}_{t},\ (r,s,t)\in\mathcal{T}_{A}.
\end{multline}
A direct calculation on derivatives of RoD provides that
\begin{align}\label{eq:rodr}
\dot{\kappa}_{ijk}&=\kappa_{ijk}\{[\frac{b_{ij}}{||e_{ij}||}-\frac{b_{ik}}{||e_{ik}||}]^{\top}\dot{p}_{i}-\frac{b_{ij}^{\top}}{||e_{ij}||}\dot{p}_{j}+\frac{b_{ik}^{\top}}{||e_{ik}||}\dot{p}_{k}\},\ (i,j,k)\in\mathcal{T}_{D}.
\end{align} 


The following proposition can be established by similar arguments in \cite[Lemma 1]{lyjbb10}.
\begin{proposition}
The SA-RoD rigidity matrix $R_{\mathcal{G}(A,D)}(p)$ can be expressed as
\begin{equation}
R_{\mathcal{G}({A,D})}(p)=[\bar{R}_{A}^{\top}(p),\bar{R}_{D}^{\top}(p)]^{\top}\bar{H},
\end{equation}
where $\bar{H}=H\otimes I_{2}$, $\bar{R}_{A}(p)\in\mathbb{R}^{|\mathcal{T}_{A}|\times 2|\mathcal{E}|}$ and $\bar{R}_{D}(p)\in\mathbb{R}^{|\mathcal{T}_{D}|\times 2|\mathcal{E}|}$ can be written as
\begin{equation}\label{IRM2}
    \bordermatrix{%
&\ldots	& b_{rs}\ & \ldots     & b_{rt} & \ldots\cr
 \ldots &\ldots & \ldots & \ldots & \ldots & \ldots \cr
	\alpha_{rst}&\mathbf{0}  & \frac{b^{\top}_{rs}\mathcal{R}{(\frac{\pi}{2})}}{||e_{rs}||}& \mathbf{0}& -\frac{b^{\top}_{rt}\mathcal{R}{(\frac{\pi}{2})}}{||e_{rt}||} & \mathbf{0} \cr
 \ldots &\ldots & \ldots & \ldots & \ldots & \ldots 
    },
\end{equation}
\begin{equation}\label{IRM1}
    \bordermatrix{%
&\ldots	& b_{ij}\ & \ldots     & b_{ik} & \ldots\cr
 \ldots &\ldots & \ldots & \ldots & \ldots & \ldots \cr
	\kappa_{ijk}&\mathbf{0}  & -\kappa_{ijk}\frac{b^{\top}_{ij}}{||e_{ij}||}& \mathbf{0}& \kappa_{ijk}\frac{b^{\top}_{ik}}{||e_{ik}||} & \mathbf{0} \cr
 \ldots &\ldots & \ldots & \ldots & \ldots & \ldots 
    }.
\end{equation}
\end{proposition}

All triplets corresponding to vertex $i$ are subject to constraints induced by $\mathcal{T}_{A}$ or $\mathcal{T}_{D}$
. Nevertheless, we stress that there is a considerable degree of redundancy in these quantities. This is due to a fundamental observation: for any $j,k,l\in\mathcal{N}_{i}$ satisfying $j<k<l$, it always holds that $\alpha_{ijk}+\alpha_{ikl}=\alpha_{ijl}, \kappa_{ijk}\kappa_{ikl}=\kappa_{ijl}$.  
\begin{lemma}\label{lemm01}
$\rank(R_{\mathcal{G}(A,D)}(p))\leq 2m-n$.
\end{lemma}

The proof is presented in Appendix $\ref{pol1}$. 
The following characterization of infinitesimal SA-RoD rigidity can be established with a proof similar to that in \cite{lyjb8,lyjb9,lyjbb8}. 

\begin{theorem}\label{infi}
Given a framework $(\mathcal{G}(A,D),p)$, then the following statements are equivalent:\\
1) $(\mathcal{G}(A,D),p)$ is infinitesimally SA-RoD rigid;\\
2) $\rank (R_{\mathcal{G}(A,D)}(p))=2n-4$;\\
3) $\nulll(R_{\mathcal{G}(A,D)}(p))=\spann\{\mathbf{1}_{n}\otimes (1,0)^\top,\mathbf{1}_{n}\otimes (0,1)^\top,({I}_{n}\otimes \mathcal{R}(\pi/2))p, p\}$.
\end{theorem}

Here we revisit frameworks in Figure $\ref{lf13}$. Frameworks in Figures $\ref{lf13}(c)$, $\ref{lf13}(d)$, and $\ref{lf13}(e)$ are (infinitesimally) SA-RoD rigid but not globally SA-RoD rigid. In Figure $\ref{lf13}(d)$, the SA-RoD rigidity functions for quadrilaterals $\Diamond 1234$ and $\Diamond 12^{\prime}3^{\prime}4$ yield identical values. Similar results hold for quadrilaterals $\Diamond 1234$ and $\Diamond 123^{\prime}4^{\prime}$ in Figure $\ref{lf13}(e)$. The framework in Figure $\ref{lf13}(f)$ is both (infinitesimally) SA-RoD rigid and globally SA-RoD rigid.
\begin{figure}[!h]
\centerline{\includegraphics[width=\columnwidth]{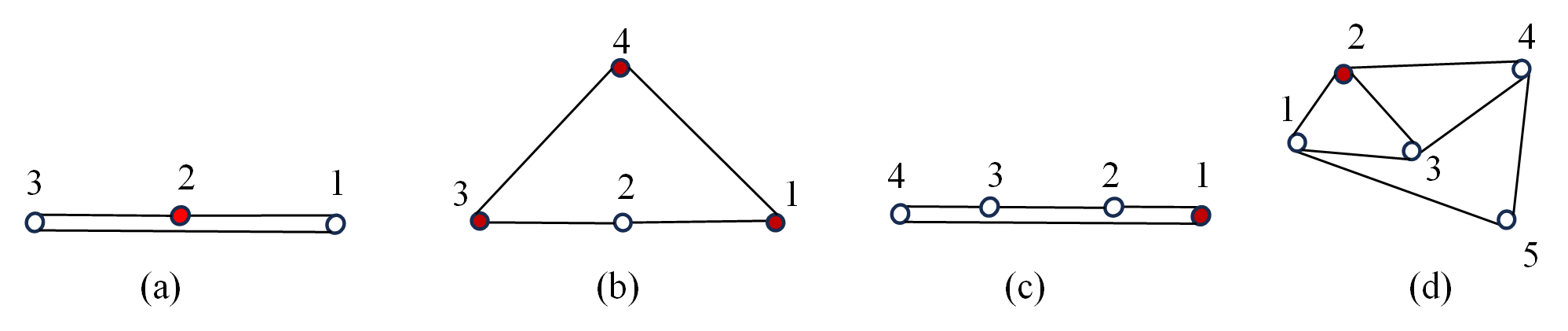}}
\caption{(a) A globally and infinitesimally SA-RoD rigid framework corresponding to a degenerate triangle with $\mathcal{V}_{A}=\{2\}$. (b) A globally and infinitesimally SA-RoD rigid framework corresponding to a quadrilateral with $\mathcal{V}_{A}=\{1,3,4\}$. (c) A globally SA-RoD rigid framework that is not infinitesimally SA-RoD rigid corresponding to a degenerate quadrilateral with $\mathcal{V}_{A}=\{1\}$. (d) An infinitesimally SA-RoD rigid framework that is not globally SA-RoD rigid corresponding to a Laman graph with $\mathcal{V}_{A}=\{2\}$. The dots in red represent vertices in $\mathcal{V}_{A}$.} 
\label{lf9}
\end{figure}
Figure $\ref{lf9}$ also presents some SA-RoD constrained frameworks. We can see that even a degenerate triangle can become globally SA-RoD rigid. A quadrilateral is also globally SA-RoD rigid if $|\mathcal{V}_{D}|=1$ and three vertices in $\mathcal{V}_{A}$ are non-collinear. Figure $\ref{lf9}(c)$ provides a quadrilateral with collinear four vertices. Due to the collinearity, global rigidity can be derived from RoD constraints $\kappa_{213},\kappa_{324},\kappa_{413}$ and SA constraint $\alpha_{124}$. The framework in Figure $\ref{lf9}(d)$ is infinitesimally SA-RoD rigid but not globally SA-RoD rigid. There exists a reflection of vertex 5 with respect to the line connecting vertices 1 and 4, which can preserve the SA-RoD constraints. 

To investigate the relationships among various concepts concerning SA-RoD rigidity, our analysis starts with $(\mathcal{K}(A,D),p)$. For a given configuration $p\in\mathbb{R}^{2n}$ satisfying Assumption 1, the set $S_{p}$ of configurations with the same shape as $p$ up to uniform rotations, translations, and scalings is defined as
\begin{align}\label{defeq}
S_{p}\triangleq\{q\in\mathbb{R}^{2n}:q=\mathbf{1}_{n}\otimes\xi+c(I_{n}\otimes\mathcal{R}(\theta))p,
\xi\in\mathbb{R}^{2},\theta\in [0,2\pi),c\in\mathbb{R}^{+}\}.
\end{align}
It is noteworthy that an equivalent definition of $S_{p}$ permits 
$c\in\mathbb{R}\setminus\{0\}$ in $(\ref{defeq})$. The restriction of 
$c$ to positive values is imposed to ensure the uniqueness of parameters $c$ and $\theta$.
 Specifically, given any $q\in S_{p}$, parameters $c>0$ and $ \theta\in [0,2\pi)$ satisfying $q=\mathbf{1}_{n}\otimes\xi+c(I_{n}\otimes\mathcal{R}(\theta))p$ are unique.
\begin{lemma}\label{lem1o}
Given a framework consisting of a complete graph $\mathcal{K}$ and a configuration $p\in\mathbb{R}^{2n}$, it holds that
$
f^{-1}_{\mathcal{K}(A,D)}(f_{\mathcal{K}(A,D)}(p))=S_{p}
$ for any bipartition.
\end{lemma}

The proof of Lemma $\ref{lem1o}$ is provided in in Appendix $\ref{pK}$.  Then we show that  under Assumption 1, $S_{p}$ is a 4-dimensional smooth manifold even if $p$ is degenerate. The detailed proof can be found in Appendix \ref{diffe}. 
\begin{lemma}\label{lem2nn}
$S_{p}$ is a 4-dimensional smooth manifold and is diffeomorphic to $\mathbb{S}^{1}\times\mathbb{R}^{2}\times \mathbb{R}^{+}$.
\end{lemma}

Based on Lemma $\ref{lem1o}$ and Lemma $\ref{lem2nn}$, we derive the following relationship
between infinitesimal SA-RoD rigidity and SA-RoD rigidity. 
\begin{theorem}\label{the1}
If a framework $(\mathcal{G}(A,D),p)$ is infinitesimally SA-RoD rigid, it is SA-RoD rigid.
\end{theorem}
 \begin{proof}
  Based on the infinitesimal SA-RoD rigidity of $(\mathcal{G}(A,D),p)$, it holds that $\rank(R_{\mathcal{G}(A,D)}(p)))=2n-4$. According to \cite[Proposition 2]{lyjb8}, there exists a neighborhood $U$ of $p$ such that $f^{-1}_{\mathcal{G}(A,D)}(f_{\mathcal{G}(A,D)}(p))\cap U$ is a 4-dimensional manifold. Combining Lemma $\ref{lem1o}$ and Lemma $\ref{lem2nn}$,  $f^{-1}_{\mathcal{K}(A,D)}(f_{\mathcal{K}(A,D)}(p))\cap U=S_{p}\cap U$ is a 4-dimensional manifold. Since $f^{-1}_{\mathcal{K}(A,D)}(f_{\mathcal{K}(A,D)}(p))\subseteq f^{-1}_{\mathcal{G}(A,D)}(f_{\mathcal{G}(A,D)}(p))$, we can conclude that  $f^{-1}_{\mathcal{K}(A,D)}(f_{\mathcal{K}(A,D)}(p))\cap U=f^{-1}_{\mathcal{G}(A,D)}(f_{\mathcal{G}(A,D)}(p))\cap U$. Therefore, $(\mathcal{G}(A,D),p)$ is SA-RoD rigid.
 \end{proof}

Note that the converse of this theorem is not true. Figure $\ref{lf9}(c)$ presents a framework that is globally SA-RoD rigid, but not infinitesimally SA-RoD rigid. By definition, global SA-RoD rigidity of a framework implies SA-RoD rigidity. Furthermore, the example in Figure $\ref{lf9}(d)$ indicates that there exist frameworks that are SA-RoD rigid, but not globally SA-RoD rigid.

Based on the established results, we investigate the conditions under which the shape of a framework can be uniquely determined by SA and RoD constraints. Lemma \ref{lem1o} shows that a framework with a complete underlying graph can be uniquely determined by such constraints. In fact, such a framework is globally SA-RoD rigid.  Here we present the following theorem to demonstrate the equivalence between global SA-RoD rigidity and shape uniqueness under SA and RoD constraints. 
\begin{theorem}\label{shp}
For a given framework $(\mathcal{G}(A,D),p)$ in $\mathbb{R}^{2}$, the following statements are equivalent: \\
1) $(\mathcal{G}(A,D),p)$ is globally SA-RoD rigid.\\
2) The shape of $(\mathcal{G}(A,D),p)$ can be uniquely determined by $\{\alpha_{ijk}:(i,j,k)\in\mathcal{T}_{A}\}$ and $\{\kappa_{ijk}:(i,j,k)\in\mathcal{T}_{D}\}$ up to uniform rotations, translations, and scalings.\\
3) All SAs in $\{\alpha_{ijk}:(i,j,k)\in\mathcal{T}_{\mathcal{G}}\}$ and RoDs in $\{\kappa_{ijk}:(i,j,k)\in\mathcal{T}_{\mathcal{G}}\}$ can be uniquely determined by $\{\alpha_{ijk}:(i,j,k)\in\mathcal{T}_{A}\}$ and $\{\kappa_{ijk}:(i,j,k)\in\mathcal{T}_{D}\}$.\\
4) All SAs in $\{\alpha_{ijk}:(i,j,k)\in\mathcal{T}_{\mathcal{K}}\}$ and RoDs in $\{\kappa_{ijk}:(i,j,k)\in\mathcal{T}_{\mathcal{K}}\}$ can be uniquely determined by $\{\alpha_{ijk}:(i,j,k)\in\mathcal{T}_{A}\}$ and $\{\kappa_{ijk}:(i,j,k)\in\mathcal{T}_{D}\}$.
\end{theorem}
 \begin{proof}
 \textit{Equivalence of 1) and 2):} From Lemma $\ref{lem1o}$, $f^{-1}_{\mathcal{K}(A,D)}(f_{\mathcal{K}(A,D)}(p))=S_{p}$. 
According to the definition of $S_p$ in (\ref{defeq}), the shape of a framework $(\mathcal{G}(A,D),p)$ in $\mathbb{R}^{2}$ can be uniquely determined by SA and RoD constraints up to uniform rotations, translations, and scalings if and only if $f^{-1}_{\mathcal{G}(A,D)}(f_{\mathcal{G}(A,D)}(p))=S_{p}$. Thus, this is equivalent to $f^{-1}_{\mathcal{G}(A,D)}(f_{\mathcal{G}(A,D)}(p))=f^{-1}_{\mathcal{K}(A,D)}(f_{\mathcal{K}(A,D)}(p))$, i.e., global SA-RoD rigidity of $(\mathcal{G}(A,D),p)$.\\
 \textit{Equivalence of 2) and 3):} The argument is similar to that of Lemma $\ref{lem1o}$.\\
 \textit{Equivalence of 2) and 4):} Observe that 2) implies 4), and 4) implies 3). Since 2) and 3) are equivalent, we obtain the desired equivalence of 2) and 4).
 \end{proof}

Now we deal with the shape determination problem of quadrilateral frameworks. According to different bipartitions of the vertex set, the global rigidity property of a quadrilateral framework can be categorized into six cases as illustrated in Figure $\ref{lf13}$. Since purely SA or RoD constrained quadrilateral frameworks are well understood, the remaining four cases are analyzed in Appendix $\ref{p2}$. To state the following result, denote $d_{ij}=\|p_{i}-p_{j}\|>0$, $b_{ij}=\frac{p_{j}-p_{i}}{\|p_{j}-p_{i}\|}=(\cos\theta_{ij},\sin\theta_{ij})^{\top}$, where $\theta_{ij}\in[0,2\pi)$ for distinct $i, j\in\{1,2,3,4\}$.
 \begin{proposition}\label{pro1}
 A quadrilateral framework $(\mathcal{G}(A,D),p)$ with the vertex set $\mathcal{V}=\{1,2,3,4\}$ and the edge set $\mathcal{E}=\{(1,2),(2,3),(3,4),(4,1)\}$ is globally SA-RoD rigid if and only if it satisfies one of the following conditions:\\
 1) $|\mathcal{V}_{A}|=3$, $|\mathcal{V}_{D}|=1$, and vertices in $\mathcal{V}_{A}$ are non-collinear.\\
 2) $|\mathcal{V}_{A}|=1$, $|\mathcal{V}_{D}|=3$, and either vertices $2,3,$ and $4$ are collinear or $d_{14}=d_{34}, d_{12}=d_{32}$, where without loss of generality, it is assumed that $\mathcal{V}_{A}=\{1\}, \mathcal{V}_{D}=\{2,3,4\}$.\\
 3) $|\mathcal{V}_{A}|=2$, $|\mathcal{V}_{D}|=2$, vertices in $\mathcal{V}_{A}$ are adjacent and it holds that 
 \begin{equation}\label{qua3n}
 d_{12}+2d_{34}\cos(\theta_{34}-\theta_{12})\leq 0,
 \end{equation}
 where without loss of generality, it is assumed that $\mathcal{V}_{A}=\{1,2\},\mathcal{V}_{D}=\{3,4\}$. \\
 4) $|\mathcal{V}_{A}|=2$, $|\mathcal{V}_{D}|=2$, vertices in $\mathcal{V}_{A}$ are not adjacent and it holds that  
 \begin{equation}\label{qua4}
 d_{12}^{2}d_{34}^{2}+d_{14}^{2}d_{23}^{2}-d_{14}^{2}d_{21}^{2}\sin^{2}\theta_{124}-d_{34}^{2}d_{23}^{2}\sin^{2}\theta_{324}-2d_{12}d_{23}d_{34}d_{14}\cos\theta_{324}\cos\theta_{124}=0,
 \end{equation}
 or 
 \begin{equation}\label{qua5}
 (d_{23}-d_{12})(d_{34}-d_{14})\leq 0,
 \end{equation}
 where without loss of generality, it is assumed that $\mathcal{V}_{A}=\{1,3\},\mathcal{V}_{D}=\{2,4\}$, $\theta_{324}=\theta_{34}-\theta_{32}$, and $\theta_{124}=\theta_{14}-\theta_{12}$. 
 \end{proposition}

  This proposition provides useful criteria to check whether an SA-RoD constrained quadrilateral framework is globally rigid. It can serve as an essential unit to construct globally rigid frameworks and explore generic properties of underlying graphs. We can verify that the frameworks in Figures $\ref{lf13}(d)$ and $\ref{lf13}(e)$, which are not globally SA-RoD rigid, do not satisfy these criteria.

\subsection{Bipartition for SA-RoD rigidity}

As shown in the motivating example, bipartitions of the vertex set have significant effects on the rigidity property of corresponding frameworks. Given a general graph 
$\mathcal{G}=(\mathcal{V},\mathcal{E})$ and a configuration $p\in\mathbb{R}^{2n}$, there are some natural questions related to possible bipartitions:

1) Is there any bipartition such that the associated framework $(\mathcal{G}(A,D),p)$ is infinitesimally/globally SA-RoD rigid? If yes, how to construct one?  

2) Are there any intrinsic relationships among (global) SA-RoD rigidities under different bipartitions of the vertex set?

3) How does node bipartition affect the generic property of SA-RoD rigidity?

\subsubsection{Necessary condition for existence of bipartitions for SA-RoD rigidity}
 In this subsection, we will present a necessary condition for the existence of bipartitions such that the framework $(\mathcal{G}(A,D), p)$ is (infinitesimally) SA-RoD rigid. This provides a partial answer to Question 1). More specifically, we can establish a lower bound on the number of edges required to guarantee SA-RoD rigidity as follows. 
\begin{lemma}\label{ll13}
Given a framework $(\mathcal{G}(A,D),p)$ in $\mathbb{R}^{2}$, if $(\mathcal{G}(A,D),p)$ is infinitesimally SA-RoD rigid, then $m\geq \frac{3n-4}{2}$.
\end{lemma}
\begin{proof}
Based on Lemma $\ref{lemm01}$ and Theorem $\ref{infi}$, we have $\rank (R_{\mathcal{G}(A,D)}(p))=2n-4\leq 2m-n.$ It follows that $m\geq \frac{3n-4}{2}$.
\end{proof}

According to Lemma $\ref{lem:eqge}$, Lemma \ref{ll13} is also valid for (globally) SA-RoD rigid frameworks with generic configurations. 
It should be noted that the converse of Lemma $\ref{ll13}$ is not true. For example, let $\mathcal{G}=(\mathcal{V},\mathcal{E})$ with $\mathcal{V}=\{1,2,3,4\}$ and  $\mathcal{E}=\{(1,2),(1,3),(2,3),(3,4)\}$, and $p\in\mathbb{R}^{8}$ be a generic configuration. Then $m=\frac{3n-4}{2}=4$ but $(\mathcal{G}(A,D),p)$ is not infinitesimally SA-RoD rigid under any bipartition as there always exist non-trivial infinitesimal SA-RoD motions.

\begin{remark}
    The lower bound in Lemma $\ref{ll13}$ is tight. In Subsection \ref{cons_gr}, we will construct globally SA-RoD rigid frameworks satisfying $m= \lceil{\frac{3n-4}{2}\rceil}$. Combining with Lemma $\ref{ll13}$, these frameworks with generic configurations are shown to be {\it minimally globally SA-RoD rigid}. That is, they are globally SA-RoD rigid and the deletion of an arbitrary edge in the underlying graphs destroys global SA-RoD rigidity. Lemma $\ref{lem:eqge}$ implies that globally SA-RoD rigid frameworks with generic configurations are also infinitesimally SA-RoD rigid. 
\end{remark}

If the underlying graph $\mathcal{G}$ is a cycle with $n\geq 5$, it holds that $m<\frac{3n-4}{2}$. Therefore, we obtain the following proposition.
\begin{proposition}\label{prop:5cyc}
    For a polygon  \footnote{In this paper, polygons may be non-convex and their edges may intersect at non-vertex points.}with more than 5 vertices in $\mathbb{R}^{2}$, the corresponding framework is not infinitesimally SA-RoD rigid under any bipartition. 
\end{proposition}

\subsubsection{Bipartition-induced triplet index graph for global SA-RoD rigidity}

First, we introduce the following definition of the {\it triplet index graph}.
\begin{definition}[Triplet index graph]\label{def:tri}
For a graph $\mathcal{G}$ and a
triplet index set $\hat{\mathcal{T}}_{\mathcal{G}}\subseteq\mathcal{T}_{\mathcal{G}}$, the undirected triplet index graph is denoted by $\mathcal{G}_{T}(\hat{\mathcal{T}}_{\mathcal{G}})= (\mathcal{V}_{T}(\mathcal{G}), \mathcal{E}_{T}(\hat{\mathcal{T}}_{\mathcal{G}}))$, where the vertex set is $\mathcal{V}_{T}(\mathcal{G})=\{a_{ij}=H(i, j, |\mathcal{V}|):(i, j)\in \mathcal{E}\}$, $H(i,j,|\mathcal{V}|)=(\min\{i,j\}-1)\times |\mathcal{V}|+\max\{i,j\}$, and the edge set is $\mathcal{E}_{T}(\hat{\mathcal{T}}_{\mathcal{G}})=\{(a_{ij},a_{ik}):a_{ij},a_{ik}\in\mathcal{V}_{T}(\mathcal{G}), (i,j,k)\in\hat{\mathcal{T}}_{\mathcal{G}}\  or\ (i,k,j)\in\hat{\mathcal{T}}_{\mathcal{G}}\}$.
\end{definition}

The expression of $H$ in the above definition indicates that $a_{ij}=a_{ji},\ (i,j)\in\mathcal{E}$. 
Given graph $\mathcal{G}(A,D)$ with a specific bipartition, the SA (RoD) index set $\mathcal{T}_A\ (\mathcal{T}_D)$ can be derived naturally. Accordingly, we can obtain the bipartition-induced SA (RoD) index graph $\mathcal{G}_{T}(\mathcal{T}_{A})$ $(\mathcal{G}_{T}(\mathcal{T}_{D}))$ based on Definition $\ref{def:tri}$. Hence we introduce the definition of the {\it SA (RoD) connectivity}.
\begin{definition}[SA (RoD) connectivity]\label{def:con}
Given graph $\mathcal{G}$, a
triplet SA (RoD) index set $\mathcal{T}_{A}\subseteq\mathcal{T}_{\mathcal{G}}$  $(\mathcal{T}_{D}\subseteq\mathcal{T}_{\mathcal{G}})$ is said to be connected over $\mathcal{G}$ if the undirected triplet index graph $\mathcal{G}_{T}(\mathcal{T}_{A})$ $(\mathcal{G}_{T}(\mathcal{T}_{D}))$ is connected. Moreover, $\mathcal{T}_{A}$ $(\mathcal{T}_{D})$ is said to have $c$ connected components over $\mathcal{G}$ if $\mathcal{G}_{T}(\mathcal{T}_{A})$ $(\mathcal{G}_{T}(\mathcal{T}_{D}))$ has $c$ connected components.
\end{definition}

For a framework $(\mathcal{G}(A,D),p)$, if $\mathcal{T}_{A}$ ($\mathcal{T}_{D}$) is connected over $\mathcal{G}$, then $(\mathcal{G}(A,D),p)$ is said to be {\it  SA (RoD) connected} over $\mathcal{G}$. Figure $\ref{lf11}$ presents a graph with a given bipartition, the SA index graph, and the RoD index graph. According to two definitions above, $\mathcal{T}_{A}$ has five connected components, and $\mathcal{T}_{D}$ is connected over $\mathcal{G}$. 


\begin{figure}[!h]
\centerline{\includegraphics[width=0.8\columnwidth]{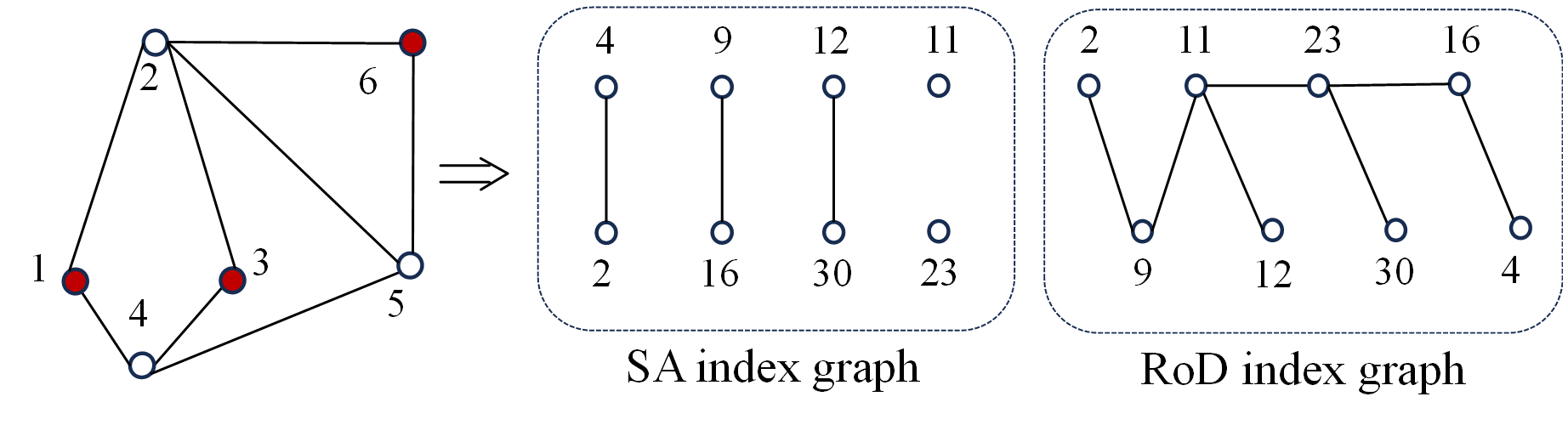}}
\caption{An example of a graph associated with a given bipartition, the SA index graph, and the RoD index graph. There exist five connected components in the SA index graph, while the RoD index graph is connected. The vertices in SA/RoD index graphs satisfy $2=a_{12}, 4=a_{14}, 9=a_{23}, 11=a_{25}, 12=a_{26}, 16=a_{34}, 23=a_{45}$, and $30=a_{56}$. The dots in red represent vertices in $\mathcal{V}_{A}$.}
\label{lf11}
\end{figure}

 Next, we discuss how to construct a suitable bipartition for the SA (RoD) connectivity. Without loss of generality, we assume $\mathcal{G}$ is connected and consider the SA connectivity as an example. A direct observation yields that if $|\mathcal{V}_A|\geq n-1$, the SA index set $\mathcal{T}_A$ is always connected over $\mathcal{G}$. Generally,  $|\mathcal{V}_A|$ may be further reduced under SA connectivity. Here we provide Algorithm \ref{alg:buildpar} to implement this procedure. 
The output $\mathcal{V}_A$ in Algorithm \ref{alg:buildpar} is minimal, in the sense that for any 
 $\mathcal{V}=\mathcal{V}^\prime_A\cup\mathcal{V}^\prime_D$ satisfying that $\mathcal{V}^\prime_A\subseteq\mathcal{V}_A$, and the SA index set $\mathcal{T}^\prime_A$ is connected over $\mathcal{G}$, it holds that 
 $\mathcal{V}^\prime_A=\mathcal{V}_A$. An example to illustrate the construction of the bipartition for the SA connectivity by Algorithm \ref{alg:buildpar} is shown in Figure \ref{cb11}.

 \begin{algorithm}
\caption{Bipartition design for SA connectivity over graph $\mathcal{G}=(\mathcal{V},\mathcal{E})$ }
\label{alg:buildpar}
\begin{algorithmic}[1]
\STATE{Initialize $\mathcal{V}_A=\mathcal{V}$, $\mathcal{V}_D=\emptyset$, and $\mathcal{T}_A=\mathcal{T}_{\mathcal{G}}$}
\FOR{$j\in\mathcal{V}\setminus\bigcup_{i\in\mathcal{V}_D}\mathcal{N}_i$}
\STATE{Update $\mathcal{V}_A=\mathcal{V}_A\setminus\{j\}$, and $\mathcal{T}_A=\mathcal{T}_A\setminus\{(j,k,l)\in\mathcal{V}^3:k,l\in\mathcal{N}_j, k<l\}$}
\STATE{Determine the SA index graph $\mathcal{G}_T(\mathcal{T}_A)$ according to Definition \ref{def:tri}} 
\STATE{Check the connectivity of $\mathcal{G}_T(\mathcal{T}_A)$ by breadth-first search \cite[Section 22.2]{Alg}}
\IF{$\mathcal{G}_T(\mathcal{T}_A)$ is disconnected}
\STATE{$\mathcal{V}_A=\mathcal{V}_A\bigcup\{j\}$, $\mathcal{T}_A=\mathcal{T}_A\bigcup\{(j,k,l)\in\mathcal{V}^3:k,l\in\mathcal{N}_j, k<l\}$}
\ENDIF
\STATE{Update $\mathcal{V}_D = \mathcal{V}\setminus\mathcal{V}_A$}
\ENDFOR
\RETURN $\mathcal{V}_A$ and $\mathcal{V}_D$ 
\end{algorithmic}
\end{algorithm}
 \begin{remark}
For graph $\mathcal{G}$ with a specific bipartition, it holds that $|\mathcal{T}_A|\leq |\mathcal{T}_\mathcal{G}|\leq \frac{n(n-1)(n-2)}{2}$. Thus there exist at most $m$ vertices and $\frac{n(n-1)(n-2)}{2}$ edges in $\mathcal{G}_T(\mathcal{T}_A)$. Note that $m\leq\frac{n(n-1)}{2}$. Hence the worst-case time complexity of BFS is 
 $\mathcal{O}(n^4)$ in Line 5. 
By examining elementary operations in Lines $3-4$ and $6-9$, we can verify the worst-case time complexity of 
 Algorithm \ref{alg:buildpar} is $\mathcal{O}(n^5)$.
  \end{remark}

\begin{figure}[!h]
\centerline{\includegraphics[width=0.8\columnwidth]{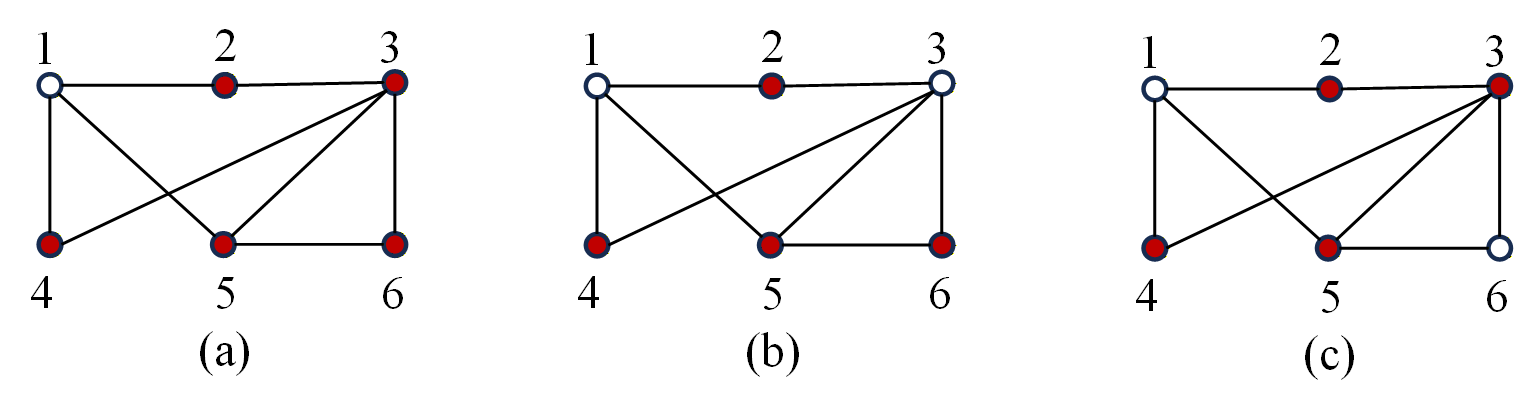}}
\caption{An example of finding bipartition $\mathcal{V}=\mathcal{V}_A\cup\mathcal{V}_D$ by Algorithm \ref{alg:buildpar}. (a) In Step 1, $\mathcal{V}_{A,1}=\mathcal{V}\setminus\{1\}=\{2,3,4,5,6\}$, and the SA index set is connected over $\mathcal{G}$. (b) In Step 2, $\mathcal{V}_{A,2}=\mathcal{V}_{A,1}\setminus\{3\}=\{2,4,5,6\}$, and the SA index set is disconnected over $\mathcal{G}$. After performing the operation in Lines 
$6-8$, we have $\mathcal{V}_{A,2}=\{2,3,4,5,6\}$.  (c) In Step 3, $\mathcal{V}_{A,3}=\mathcal{V}_{A,2}\setminus\{6\}=\{2,3,4,5\}$, and the SA index set is connected over $\mathcal{G}$. The dots in red represent vertices in $\mathcal{V}_{A}$.}
\label{cb11}
\end{figure}


If the SA index set $\mathcal{T}_{A}$ is connected over $\mathcal{G}$, we can establish the following theorem.
\begin{theorem}\label{thmm}
For an infinitesimally SA rigid framework $(\mathcal{G}, p)$ in $\mathbb{R}^{2}$, and a given bipartition $\mathcal{V}=\mathcal{V}_{A}\cup\mathcal{V}_{D}$, if the SA index set $\mathcal{T}_{A}$ is connected over $\mathcal{G}$, then $(\mathcal{G}(A,D),p)$ is globally SA-RoD rigid.
\end{theorem}
\begin{proof}
Suppose $q\in f^{-1}_{\mathcal{G}(A,D)}(f_{\mathcal{G}(A,D)}(p))$, then $\alpha_{ijk}(q)=\alpha_{ijk}(p)$ for $(i,j,k)\in\mathcal{T}_{A}$. As in the proof of \cite[Lemma 6]{lyjbb10},  we have  $\alpha_{ijk}(q)=\alpha_{ijk}(p)$  for all $(i,j,k)\in\mathcal{T}_{\mathcal{G}}$ since $\mathcal{T}_{A}$ is connected over $\mathcal{G}$. 
Based on Lemma $\ref{eqSA1}$, infinitesimal SA rigidity of $(\mathcal{G}, p)$ implies that SA constraints induced by $\mathcal{T}_{\mathcal{G}}$ are sufficient to guarantee the shape uniqueness. Thus, the shape of $(\mathcal{G}(A,D),p)$ can be uniquely determined by SA constraints in $\mathcal{T}_{A}$.   From Theorem $\ref{shp}$, $(\mathcal{G}(A,D),p)$ is globally SA-RoD rigid. 
\end{proof}

Based on Lemma $\ref{eqSA1}$, we know that Theorem $\ref{thmm}$ is also valid for non-degenerate and globally SA rigid frameworks. 
Similarly, if $(\mathcal{G}, p)$ is globally RoD rigid or globally distance rigid, we can show that $(\mathcal{G}(A,D),p)$ can inherit the rigidity property when the RoD index set $\mathcal{T}_{D}$ is connected over $\mathcal{G}$.

\begin{theorem}\label{thmm1}
For a globally RoD rigid framework $(\mathcal{G}, p)$ in $\mathbb{R}^{2}$, and a given bipartition $\mathcal{V}=\mathcal{V}_{A}\cup\mathcal{V}_{D}$, if the RoD index set $\mathcal{T}_{D}$ is connected over $\mathcal{G}$, then $(\mathcal{G}(A,D),p)$ is globally SA-RoD rigid as well.
\end{theorem}


Here we summarize the results in the above two subsections to answer Question 1). First, there exists a necessary condition for the existence of bipartitions to ensure rigidity. Next, if framework $(\mathcal{G},p)$ is infinitesimally SA rigid (globally RoD rigid), there exists non-trivial bipartition such that $(\mathcal{G}(A,D),p)$ is globally SA-RoD rigid. In this setting, we can construct any bipartition, which guarantees the connectivity of the SA (RoD) constraints, to ensure global SA-RoD rigidity of $(\mathcal{G}(A,D),p)$. Furthermore, a polynomial-time algorithm is provided to construct a suitable bipartition for the SA (RoD) connectivity.
\subsubsection{Bipartition-induced duality for infinitesimal SA-RoD rigidity}
To answer Question 2), we present a novel duality of infinitesimal SA-RoD rigidity with respect to two bipartitions in this subsection.
\begin{theorem}[Dual rank property]\label{thm0}
 Given two frameworks $(\mathcal{G}(A,D),p)=(\mathcal{V}_{A}\cup\mathcal{V}_{D},\mathcal{E},p)$\ and $(\mathcal{G}^{\prime}(A,D),p)=(\mathcal{V}^{\prime}_{A}\cup\mathcal{V}^{\prime}_{D},\mathcal{E},p)$. If $\mathcal{V}^{\prime}_{D}=\mathcal{V}_{A}$ and $\mathcal{V}^{\prime}_{A}=\mathcal{V}_{D}$, then $\rank(R_{\mathcal{G}(A,D)}(p))=\rank(R_{\mathcal{G}^{\prime}(A,D)}(p)).$
\end{theorem}

The proof can be found in Appendix $\ref{pdual}$. According to Theorem \ref{infi} , we can derive the following corollary.
\begin{corollary}[Duality of infinitesimal SA-RoD rigidity]\label{thm10}
 Under the same assumptions in Theorem \ref{thm0}, $(\mathcal{G}(A,D),p)$ is infinitesimally SA-RoD rigid if and only if $(\mathcal{G}^{\prime}(A,D),p)$ is infinitesimally SA-RoD rigid.
\end{corollary}

\begin{remark}
In fact, our proof still holds for $\mathcal{V}_{A}=\emptyset$ or $\mathcal{V}_{D}=\emptyset$. As a byproduct, we can show that the SA rigidity matrix and RoD rigidity matrix of a given framework have the same rank. Therefore, in $\mathbb{R}^{2}$, infinitesimal SA rigidity is equivalent to infinitesimal RoD rigidity. This can also be obtained by the results in \cite{lyjbb10,lyjbb8.1}. However, as illustrated in Figure $\ref{lf9}(a)$, infinitesimal SA-RoD rigidity of a framework is not equivalent to infinitesimal SA rigidity or infinitesimal RoD rigidity. 
\end{remark}

It is worth noting that this duality is lacking for SA-RoD rigidity and global SA-RoD rigidity. A typical example is provided by the degenerate quadrilateral in Figure $\ref{lf9}(c)$. If $\mathcal{V}_{A}=\{1\}$ and $\mathcal{V}_{D}=\{2,3,4\}$, then $(\mathcal{G}(A,D),p)$ is globally SA-RoD rigid. If $\mathcal{V}_{A}$ and $\mathcal{V}_{D}$ are swapped, then the induced framework $(\mathcal{G}(A,D),p)$ is not SA-RoD rigid. The reason is that vertex 3 can move freely on the segment connecting vertex $2$ and vertex $4$. However, we can show that for a generic configuration $p$ (see the definition in Appendix $\ref{deg:grt}$), this duality holds for SA-RoD rigidity as a consequence of Lemma $\ref{lem:eqge}$. A detailed discussion of generic properties is given in Subsection $\ref{def:gep}$.
\begin{corollary}[Duality of SA-RoD rigidity]\label{thm:dualg}
    For frameworks $(\mathcal{G}(A,D),p)=(\mathcal{V}_{A}\cup\mathcal{V}_{D},\mathcal{E},p)$ and $(\mathcal{G}^{\prime}(A,D),p)=(\mathcal{V}^{\prime}_{A}\cup\mathcal{V}^{\prime}_{D},\ \mathcal{E},p)$. If the configuration $p$ is generic in $\mathbb{R}^{2n}$, $\mathcal{V}^{\prime}_{D}=\mathcal{V}_{A}$ and $\mathcal{V}^{\prime}_{A}=\mathcal{V}_{D}$, then $(\mathcal{G}(A,D),p)$ is SA-RoD rigid if and only if $(\mathcal{G}^{\prime}(A,D),p)$ is SA-RoD rigid.
\end{corollary}

\subsubsection{Bipartition-dependent generic property}\label{def:gep}
This subsection focuses on the generic property of (infinitesimal/global) SA-RoD rigidity. 
To answer Question 3), we first give a definition of {\it generic SA-RoD rigidity} for a graph with a specific bipartition.

\begin{definition}
A graph $\mathcal{G}(A,D)$ with the specific bipartition $\mathcal{V}_A\cup\mathcal{V}_D$ is called generically (infinitesimally/globally) SA-RoD rigid in $\mathbb{R}^{2}$ if for any generic configuration $p\in\mathbb{R}^{2n}$, $(\mathcal{G}(A,D),p)$ is (infinitesimally/globally) SA-RoD rigid.
\end{definition}

The following theorem shows that (infinitesimal) SA-RoD rigidity is a generic property of the underlying graph with a specific bipartition. The proof is provided in Appendix $\ref{pg1}$. 
\begin{theorem}\label{gen1}
Given a generic configuration $p\in\mathbb{R}^{2n}$, if $(\mathcal{G}(A,D),p)$ is (infinitesimally) SA-RoD rigid, then graph $\mathcal{G}(A,D)=(\mathcal{V}_{A}\cup\mathcal{V}_{D},\mathcal{E})$ is generically (infinitesimally) SA-RoD rigid.
\end{theorem}

However, global SA-RoD rigidity is not a generic property of the underlying graph with the specific bipartition. That is, Theorem \ref{gen1} is not valid for global SA-RoD rigidity. Next, we give a counterexample to show this. 
 
 \textbf{A counterexample:} For graph $\mathcal{G}(A,D)$ with the specific bipartition $\mathcal{V}_{A}=\{1,2\}$, $\mathcal{V}_{D}=\{3,4\}$, and $\mathcal{E}=\{(1,2),(2,3),(3,4),(4,1)\}$, let $(\mathcal{G}(A,D),p)$ be a quadrilateral framework with configuration $p=(0,0,4,0,3,1,2,1)^{\top}\in\mathbb{R}^{8}$, which is not globally SA-RoD rigid as shown in Figure $\ref{lf13}(d)$. According to Proposition \ref{pro1}, the configuration set $\mathcal{C}_{1}$ of  frameworks that are not globally SA-RoD rigid is characterized by 
$d_{12}+2d_{34}\cos(\theta_{34}-\theta_{12})>0$.   
Since any configuration in a small neighborhood of $p$ must satisfy the above inequality, it holds that $\mathcal{C}_{1}$ is of positive Lebesgue measure in $\mathbb{R}^{8}$. Conversely, let $q=(0,0,4,0,4,1,0,1)^{\top}\in\mathbb{R}^{8}$, $(\mathcal{G}(A,D),q)$ with the same underlying graph is a globally SA-RoD rigid framework. Furthermore, configuration $q$ satisfies $d_{12}+2d_{34}\cos(\theta_{34}-\theta_{12})<0$. Similarly, the configuration set $\mathcal{C}_{2}$ of globally SA-RoD rigid frameworks is also of positive Lebesgue measure in $\mathbb{R}^{8}$. However, if global SA-RoD rigidity were a generic property of $\mathcal{G}(A,D)$, either the configuration set $\mathcal{C}_{1}$ or $\mathcal{C}_{2}$ must have zero measure. Therefore, under this setting, global SA-RoD rigidity is not a generic property for the underlying graph $\mathcal{G}(A,D)$ with the given bipartition. 

The above results yield the following answer to Question 3). (Infinitesimal) SA-RoD rigidity is a generic property of the underlying graph with a specific bipartition, while global SA-RoD rigidity is not a generic property. This reveals an important aspect that distinguishes the SA-RoD rigidity theory from previous rigidity theories. As a consequence, besides underlying graphs, the effects of configurations and bipartitions can not be neglected when constructing global SA-RoD rigid frameworks. 

\subsection{Construction of globally SA-RoD rigid frameworks}
\label{cons_gr}
\par This subsection presents a scalable approach to construction of globally SA-RoD rigid frameworks. Since global SA-RoD rigidity is highly relevant to the bipartition, our construction of globally SA-RoD rigid frameworks involves the design of bipartition, which is significantly different from Henneberg construction for distance (bearing) rigidity theory. 


First, we propose four types of {\it 1-vertex addition} methods to maintain global SA-RoD rigidity as follows.

\begin{definition}\label{1-ver}(1-Vertex Addition) Given a framework $(\mathcal{G}(A,D),p)$ with $|\mathcal{V}|=n$, add a vertex $n+1$ and two edges connecting vertex $n+1$ with two existing vertices $i,j$.  Let $p_{n+1}$ be the coordinate of vertex $n+1$ and denote $\mathcal{V}^{\prime}=\mathcal{V}\cup\{n+1\}$. Consider a bipartition $\mathcal{V}^{\prime}=\mathcal{V}^{\prime}_{A}\cup\mathcal{V}^{\prime}_{D}$ such that $\mathcal{V}_{A}\subseteq\mathcal{V}^{\prime}_{A}$ and $\mathcal{V}_{D}\subseteq\mathcal{V}^{\prime}_{D}$, the operation is called\\
1) a Type $A_{1}$ 1-vertex addition if $n+1\in\mathcal{V}^{\prime}_{A}$, and $i\in\mathcal{V}_{D}$ or $j\in\mathcal{V}_{D}$.\\
2) a Type $D_{1}$ 1-vertex addition if $n+1\in\mathcal{V}^{\prime}_{D}$, and $i\in\mathcal{V}_{A}$ or $j\in\mathcal{V}_{A}$.\\
3) a Type $A_{2}$ 1-vertex addition if $n+1\in\mathcal{V}^{\prime}_{A}$, $i,j\in\mathcal{V}_{A}$, and $p_{i},p_{j},p_{n+1}$ are non-collinear.\\
4) a Type $D_{2}$ 1-vertex addition if $n+1\in\mathcal{V}^{\prime}_{D}$, $i,j\in\mathcal{V}_{D}$, and for a third existing vertex $k\in\mathcal{V}_{D}$ such that $p_{i},\ p_{j}$, and $p_{k}$ are non-collinear, the third edge $(n+1,k)$ is added.
\end{definition}

These four types defined above are illustrated in Figure $\ref{lf117v}$. 
Then we show that all types of 1-vertex addition can preserve the global SA-RoD rigidity property. The proof is provided in Appendix $\ref{padd}$.
\begin{figure}[!h]
\centerline{\includegraphics[width=\columnwidth]{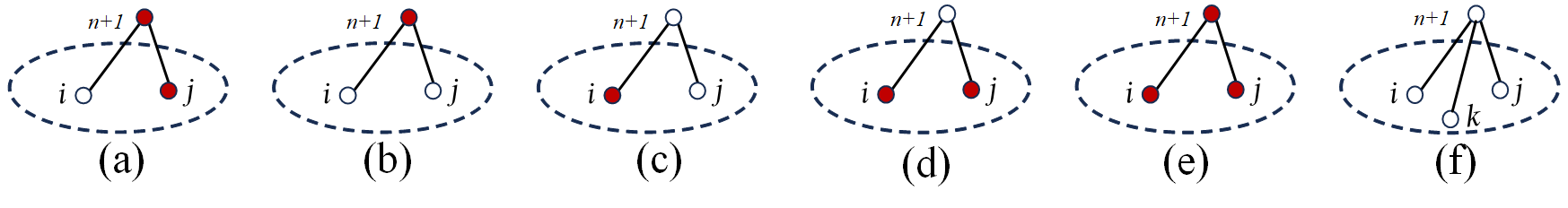}}
\caption{1-vertex addition. (a)(b) Type $A_{1}$ 1-vertex addition. (c)(d)  Type $D_{1}$ 1-vertex addition. (e)Type $A_{2}$ 1-vertex addition. (f) Type $D_{2}$ 1-vertex addition. The dots in red represent vertices in $\mathcal{V}_{A}$.}
\label{lf117v}
\end{figure}
\begin{lemma}\label{1ver}
Given a globally SA-RoD rigid framework $(\mathcal{G}(A,D),p)$ in $\mathbb{R}^{2}$, after performing a 1-vertex addition (including Type $A_{1}$, Type $A_{2}$, Type $D_{1}$, and Type $D_{2}$), the induced framework $(\mathcal{G}^{\prime}(A,D),p^{\prime})$ is globally SA-RoD rigid as well.
\end{lemma}

 To reduce the number of edges added during the construction, we further propose the following {\it 2-vertex addition} to preserve global rigidity.
\begin{definition}\label{2-ver}(2-Vertex Addition) Given a framework $(\mathcal{G}(A,D),p)$ with $|\mathcal{V}|=n$, add two vertices $n+1,n+2$ and three edges $(j,n+1), (n+1,n+2), (n+2,i)$, denote $\mathcal{V}^{\prime}=\mathcal{V}\cup\{n+1,n+2\}$. Consider a bipartition $\mathcal{V}^{\prime}=\mathcal{V}^{\prime}_{A}\cup\mathcal{V}^{\prime}_{D}$ such that $\mathcal{V}_{A}\subseteq\mathcal{V}^{\prime}_{A}$ and $\mathcal{V}_{D}\subseteq\mathcal{V}^{\prime}_{D}$, then the operation is called a 2-vertex addition if exactly three non-collinear vertices among $i,j,n+1,n+2$ belong to $\mathcal{V}^{\prime}_{A}$, and the remaining vertex belongs to $\mathcal{V}^{\prime}_{D}$.
\end{definition}
\begin{figure}[!h]
\centerline{\includegraphics[width=0.4\columnwidth]{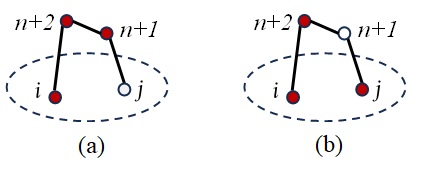}}
\caption{2-vertex addition. (a) The newly added two vertices belong to $\mathcal{V}_{A}$. (b) The newly added two vertices belong to $\mathcal{V}_{D}$ and $\mathcal{V}_{A}$. The dots in red represent vertices in $\mathcal{V}_{A}$.}
\label{lf118v}
\end{figure}

The following lemma demonstrates that 2-vertex addition can also preserve global SA-RoD rigidity. The proof is similar to that in Appendix $\ref{padd}$.
\begin{lemma}\label{2ver}
Given a globally SA-RoD rigid framework $(\mathcal{G}(A,D),p)$ in $\mathbb{R}^{2}$, after performing a 2-vertex addition,  the induced framework $(\mathcal{G}^{\prime}(A,D),p^{\prime})$ is globally SA-RoD rigid as well. 
\end{lemma}

Now we are ready to propose Algorithm $\ref{alg:buildSR}$ for construction of globally SA-RoD rigid frameworks.

\begin{algorithm}
\caption{{\it SA-RoD ordering} for constructing globally SA-RoD rigid frameworks} 
\label{alg:buildSR}
\begin{algorithmic}
\STATE{Start with $\mathcal{V}=\{1,2\},\ p(2)=(p_{1}^{\top},p_{2}^{\top})^{\top}$, and $ \mathcal{E}=\{(1,2)\}$.}
\WHILE{$|\mathcal{V}| < k$}
\STATE{Perform a 1-vertex addition or a 2-vertex addition}
\ENDWHILE
\end{algorithmic}
\end{algorithm}

\begin{figure}[!h]
\centerline{\includegraphics[width=\columnwidth]{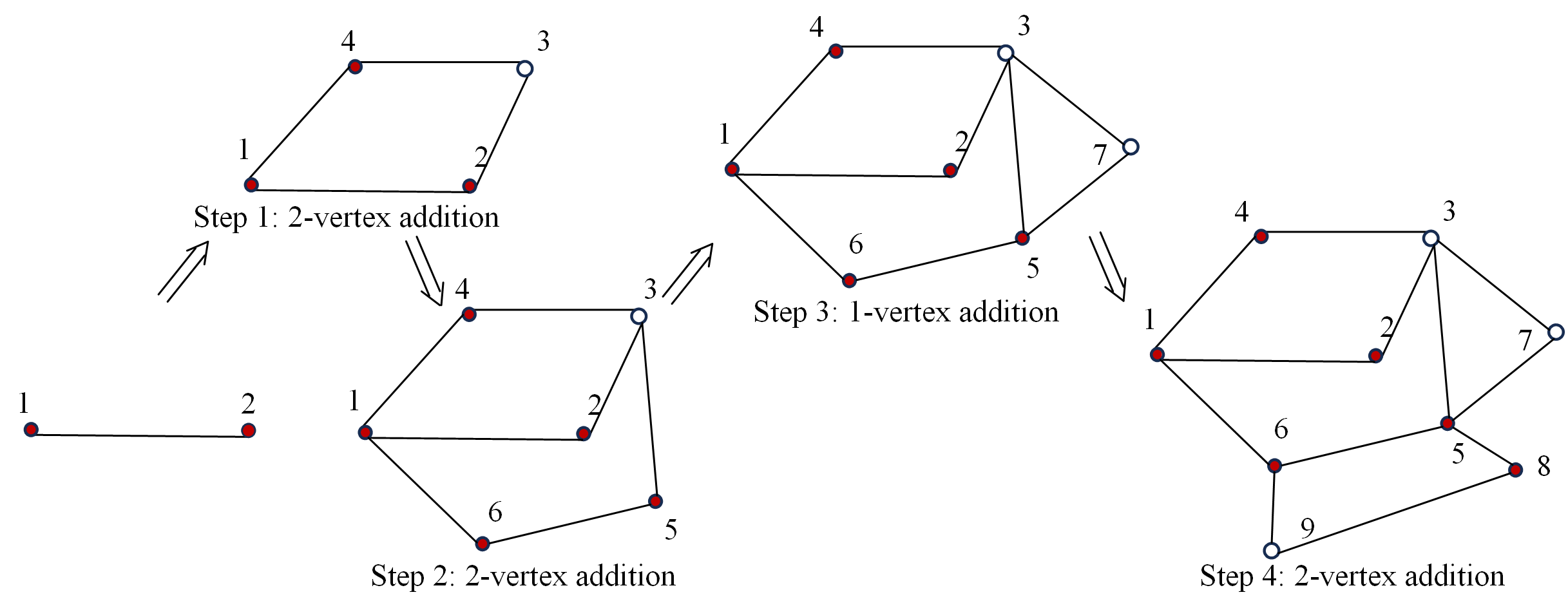}}
\caption{Construction of minimally globally SA-RoD rigid frameworks through an SA-RoD ordering. The dots in red represent vertices in $\mathcal{V}_{A}$.}
\label{fig7}
\end{figure}

Based on Lemma $\ref{1ver}$ and Lemma $\ref{2ver}$, the following theorem concerning global SA-RoD rigidity can be presented.
\begin{theorem}\label{order1}
If a framework $(\mathcal{G}(A,D),p)$ is constructed by an SA-RoD ordering in $\mathbb{R}^{2}$, then it is globally SA-RoD rigid.
\end{theorem}

Now we show two special cases for construction of minimally globally SA-RoD rigid frameworks with generic configurations: 
\begin{itemize}
    \item A framework constructed by $k$ steps of 2-vertex additions, where there exist $2k+2$ vertices and $3k+1$ edges in the underlying graph; 
    \item A framework constructed by $k$ steps of 2-vertex additions and $1$ step of 1-vertex addition, where there exist $2k+1$ vertices and $3k$ edges in the underlying graph. 
\end{itemize}
In these two cases, the lower bound in Lemma $\ref{ll13}$ is tight and the frameworks are minimally globally SA-RoD rigid. Figure $\ref{fig7}$ shows an example of the second case.   

Besides these minimally rigid frameworks, we introduce some special globally SA-RoD rigid frameworks constructed by SA-RoD ordering. If a framework is purely constructed by 1-vertex additions in Algorithm $\ref{alg:buildSR}$, it is said to have a {\it bilateration ordering} \cite{lyjb06}. Furthermore, if a  bilateration ordering is generated by Type $A_{1}$ and Type $D_{1}$ 1-vertex additions in turn, we call it a Type $(A_{1},D_{1})$ bilateration ordering. Similarly, bilateration orderings of Type $(D_{1},A_{1})$, Type $(A_{1},D_{2})$, Type $(A_{2},D_{1})$, Type $(A_{2},D_{2})$, Type $(D_{2},A_{1})$, Type $(D_{1},A_{2})$, and Type $(D_{2},A_{2})$ can be defined. If a framework is purely constructed by 2-vertex additions in Algorithm $\ref{alg:buildSR}$ and every edge lies in a quadrilateral, it is said to be a quadrilateralized framework. 

Note that frameworks constructed by different SA-RoD orderings can exhibit different SA or RoD connectivity properties, which can affect localization algorithm synthesis to be discussed in Subsection $\ref{sub:loccon}$. Specifically, in Algorithm $\ref{alg:buildSR}$, if 2-vertex additions and Type $D_{1}$ 1-vertex additions occur in turn, and there exists at least one edge connecting two vertices in $\mathcal{V}_{D}$, then we can obtain a framework, which is neither SA connected nor RoD connected. Such a framework is said to be generated by a Type $(2,D_{1})$ SA-RoD ordering. Figure $\ref{lf117}$ presents three frameworks generated with a Type $(D_{1},A_{1})$ bilateration ordering, 2-vertex additions, and a Type $(2, D_{1})$ SA-RoD ordering, respectively. 

\begin{figure}[!h]
\centerline{\includegraphics[width=0.85\columnwidth]{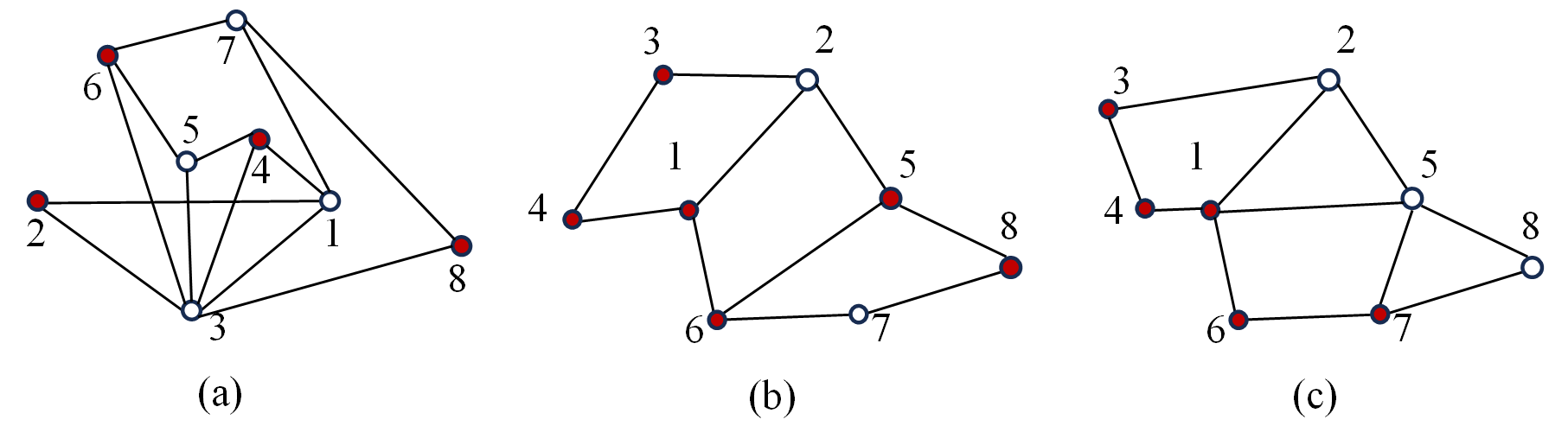}}
\caption{(a) A framework with a Type $(D_{1},A_{1})$ bilateration ordering. It is RoD connected over the underlying graph. (b) A quadrilateralized framework constructed by 2-vertex additions. It is SA connected over the underlying graph.(c) A framework with a Type $(2, D_{1})$ SA-RoD ordering. It is neither RoD connected nor SA connected over the underlying graph. The dots in red represent vertices in $\mathcal{V}_{A}$.}
\label{lf117}
\end{figure}
\subsection{Merging globally SA-RoD rigid frameworks}
Merging is an important operation to increase the network scale while maintaining rigidity in formation control and network localization. It results in changes in the vertex set or interconnection structure of sensing and communication links. In this subsection, we study merging by two types of operations. The first type is connecting two globally SA-RoD rigid frameworks by adding edges. The second type is contracting vertices of two globally SA-RoD rigid frameworks. Existing works on merging distance or bearing rigid graphs can be found in \cite{eren2005merging} and \cite{lyjbmin}.  
 
To ensure global bearing rigidity, three edges are required to be added in merging two existing bearing rigid graphs in $\mathbb{R}^{2}$ \cite{lyjbmin}. We will show that under the setting of merging globally SA-RoD rigid frameworks, adding two edges suitably is sufficient to ensure global rigidity.  The reason lies in the SA-RoD rigidity results of quadrilaterals in Proposition $\ref{pro1}$. We now state the following theorem. There might be a slight notational overlap ($A_{i},D_{i},i=1,2$), but it will not cause any ambiguity. 
\begin{theorem}[Adding edges]\label{mert}
Given two globally SA-RoD rigid frameworks $(\mathcal{G}_{1}(A_{1},D_{1}),p)$ and $(\mathcal{G}_{2}(A_{2},D_{2}),q)$ in $\mathbb{R}^{2}$, choose vertices $i,m\in\mathcal{V}_{1}$ and $j,k\in\mathcal{V}_{2}$ and the merged framework $(\mathcal{G}(A,D),r)$ is denoted by $\mathcal{V}_{A}=\mathcal{V}_{A_{1}}\cup \mathcal{V}_{A_{2}},\mathcal{V}_{D}=\mathcal{V}_{D_{1}}\cup \mathcal{V}_{D_{2}},r=(p^{\top},q^{\top})^{\top}$, and $\mathcal{E}=\mathcal{E}_{1}\cup\mathcal{E}_{2}\cup\{(m,k),(i,j)\}$. If $|\mathcal{V}_{A}\cap\{i,m,j,k\}|=3$ and the corresponding three points are non-collinear, then $(\mathcal{G}(A,D),r)$ is globally SA-RoD rigid. 
\end{theorem}

The proof is presented in Appendix $\ref{pmer}$. However, if the chosen four vertices $i,m,j,k$ belong to $\mathcal{V}_{A}$ and the corresponding four points are non-collinear, the newly merged framework by adding three edges is also globally SA-RoD rigid. This can be proved similarly by combining Lemma $\ref{eqSA1}$ and Theorem $\ref{shp}$. 

Analogously to \cite[Theorem 3.6 ]{lyjbmin}, we present the second type of merging by contracting vertices of globally SA-RoD rigid frameworks in the following theorem. It is important to note that this type of operation is allowed only among vertices of the {\it same attributes} (i.e., they belong to the same part of the bipartition). The proof is similar to that of Theorem $\ref{mert}$ and will be omitted. 
\begin{theorem}[Contracting vertices]\label{mert1}
Given two globally SA-RoD rigid frameworks $(\mathcal{G}_{1}(A_{1},D_{1}),p)$ and $(\mathcal{G}_{2}(A_{2},D_{2}),q)$ in $\mathbb{R}^{2}$, suppose there exist two pairs of vertices $\{i,j\}$ and $\{m,k\}$ with $i,m\in\mathcal{V}_{1}$ and $j,k\in\mathcal{V}_{2}$, such that $p_{i}=q_{j},p_{m}=q_{k}$, and vertices in each pair share the same attributes. Let $\mathcal{V}_{A}=\mathcal{V}_{A_{1}}\cup \mathcal{V}_{A_{2}},\mathcal{V}_{D}=\mathcal{V}_{D_{1}}\cup \mathcal{V}_{D_{2}}$, and $\mathcal{E}=\mathcal{E}_{1}\cup\mathcal{E}_{2}$. If the merged framework $(\mathcal{G}(A,D),r)$ is obtained from $\mathcal{V}_{A}, \mathcal{V}_{D},\mathcal{E}$ and $(p^{\top},q^{\top})^{\top}$ by contracting the vertex pairs $\{i,j\}$ and $\{m,k\}$, then $(\mathcal{G}(A,D),r)$ is globally SA-RoD rigid. 
\end{theorem}

 In Figure $\ref{lf118m}$, two frameworks with vertices $\{1,2,3,4,5,6\}$ and $\{7,8,9,10\}$ are globally SA-RoD rigid. Three different merging approaches are presented by adding two edges, three edges, and contracting vertices, respectively. If the four chosen vertices in Theorem $\ref{mert}$ all belong to $\mathcal{V}_{A}$, adding three edges, such as $(2,7),(2,10)$ and $(3,10)$ in Figure $\ref{lf118m}(b)$, can preserve the global SA-RoD rigidity when $r_{2}, r_{7},r_{3},r_{10}$ are non-collinear. Similarly, contracting two pairs of vertices $\{2,7\}$ and $\{3,10\}$ as in Figure $\ref{lf118m}(c)$ can also ensure global SA-RoD rigidity. 
\begin{figure}[!h]
\centerline{\includegraphics[width=0.9\columnwidth]{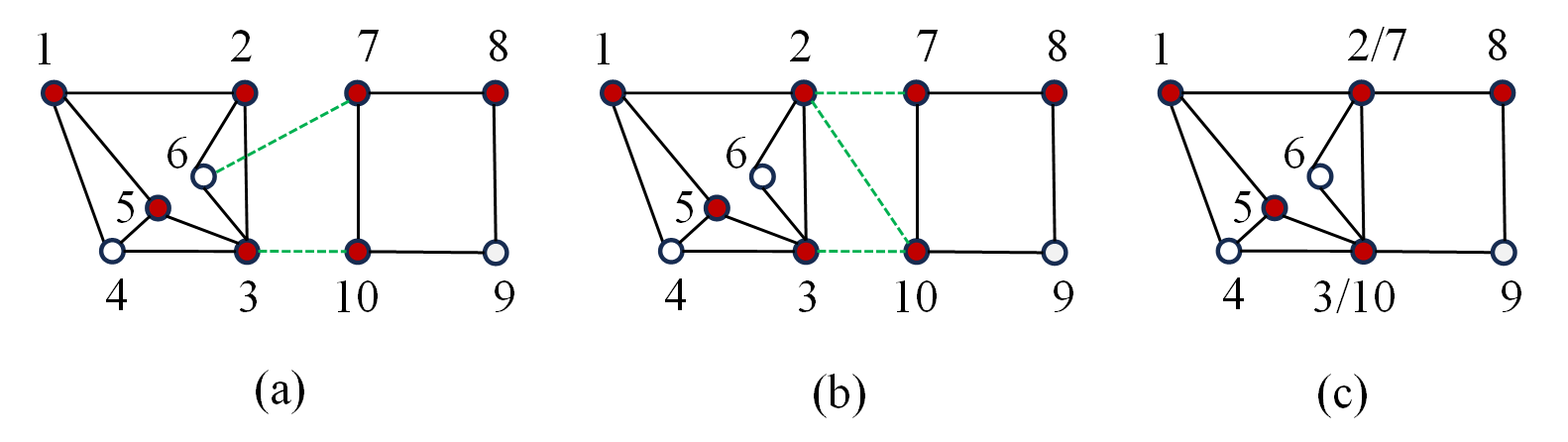}}
\caption{ Different merging approaches for two globally SA-RoD rigid frameworks. (a) Adding two edges $(6,7)$ and $(3,10)$. (b) Adding three edges $(2,7),(2,10)$ and $(3,10)$. (c) Contracting two pairs of vertices $(2,7)$ and $(3,10)$. The dots in red represent vertices in $\mathcal{V}_{A}$.}
\label{lf118m}
\end{figure}

\section{SNL with SA and RoD measurements}
In this section, we will provide a mathematical formulation of the SNL problem with SA and RoD measurements. Based on the developed rigidity theory, we will discuss the localizability of sensor networks, and propose localization algorithms with guaranteed convergence for both centralized and distributed SNL problems.

\subsection{Problem Formulation}
Given a static sensor network in $\mathbb{R}^{2}$ composed of $n_{a}$ {\it anchors} whose locations are known and $n_{f}$ free nodes whose locations are to be determined. The sensor network is indexed by $\mathcal{V}=\mathcal{V}_{a}\cup\mathcal{V}_{f}$, where $\mathcal{V}=\{1,2,\ldots,n\}, \mathcal{V}_{a}=\{1,2,\ldots,n_{a}\}$ is the set of anchors, $\mathcal{V}_{f}=\{n_{a}+1,n_{a}+2,\ldots,n_{a}+n_{f}\}$ is the set of free nodes. Consider a connected undirected graph $\mathcal{G}(A,D)=(\mathcal{V}_{A}\cup\mathcal{V}_{D},\mathcal{E})$ associated with a given bipartition of the vertex set representing the sensing graph among all nodes. Each sensor in $\mathcal{V}_{A}$ has the capability of sensing SAs subtended at itself.  Each sensor in $\mathcal{V}_{D}$ is capable of measuring RoDs with respect to any two neighbor nodes.

SNL aims at finding the locations of free nodes in $\mathcal{V}_{f}$ when anchor locations, SAs measured by sensors in $\mathcal{V}_{A}$, and  RoDs measured by sensors in $\mathcal{V}_{D}$ are provided. 
Denote the location of sensor $i$ as $x_{i}\in\mathbb{R}^{2}$, and $x=(x^{\top}_{1},\ldots,x^{\top}_{n})^{\top}\in\mathbb{R}^{2n}$. The true locations of sensors are represented by $p=(p^{\top}_{1},\ldots,p^{\top}_{n})^{\top}\in\mathbb{R}^{2n}$. An important observation is that the bearing $\frac{x_{i}-x_{j}}{||x_{i}-x_{j}||}$ and distance $||x_{i}-x_{j}||$ can be obtained for any pair of anchors $i,j\in\mathcal{V}_{a}$ even if $(i,j)\notin\mathcal{E}$. Employing the handling in \cite{lyjb09, lyjb07}, the augmented framework is defined as $(\hat{\mathcal{G}}(A,D),x)$ with $\hat{\mathcal{G}}=(\mathcal{V}_{A}\cup\mathcal{V}_{D},\hat{\mathcal{E}})$, $\hat{\mathcal{E}}=\mathcal{E}\cup\{(i,j)\in\mathcal{V}^{2}:i,j\in\mathcal{V}_{a}\}$. Namely, $\hat{\mathcal{E}}$ is obtained by adding edges to $\mathcal{E}$ which connect every pair of anchors to each other. In this paper, $(\hat{\mathcal{G}}(A,D),x,\mathcal{V}_{a})$ represents a sensor network to be localized and $(\hat{\mathcal{G}}(A,D),x)$ is called the underlying framework.
Then the centralized SNL problem can be formulated as follows.

\textit{Problem 1:} 
\begin{align}\label{snl11}
&\text{find}\ x,
\notag
\\\text{s.t.}\ \frac{x_{k}-x_{i}}{||x_{k}-x_{i}||}&=\mathcal{R}(\theta_{ijk})\frac{x_{j}-x_{i}}{||x_{j}-x_{i}||}, (i,j,k)\in\mathcal{T}_{\hat{A}},
\notag
\\ \frac{||x_{t}-x_{r}||}{||x_{s}-x_{r}||}&=\kappa_{rst}, (r,s,t)\in\mathcal{T}_{\hat{D}},
\notag
\\ x_{i}&=p_{i}, i\in\mathcal{V}_{a},
\end{align}
where $\kappa_{rst}$ is the RoD information measured by $r\in\mathcal{V}_{D}$, $\theta_{ijk}$ is the SA information measured by $i\in\mathcal{V}_{A}$, $\mathcal{T}_{\hat{A}}=\{(u,v,w)\in\mathcal{V}^{3}:u\in\mathcal{V}_{A},(u,v),(u,w)\in\hat{\mathcal{E}},v<w\}$ is the SA index set determining all the SAs subtended in the framework, $\mathcal{T}_{\hat{D}}=\{(u,v,w)\in\mathcal{V}^{3}:u\in\mathcal{V}_{D},(u,v),(u,w)\in\hat{\mathcal{E}},v<w\}$ is the RoD index set determining all the RoDs subtended in the framework. 


 \begin{remark}
  Due to the unavoidable noises in relative measurements, SNL in a noisy environment has received considerable attention\cite{anderson2010formal,chenji2012toward,xiao2017noise}. In the presence of noises, we denote the noisy SA and RoD measurements in Problem 1 as $\tilde{\theta}_{ijk}=\theta_{ijk}+u_{ijk}$, $\tilde{\kappa}_{rst}=\kappa_{rst}+v_{rst}$, where $u_{ijk}$ and $v_{rst}$ are noise-induced terms. Then, the SNL problem can be modeled as a likelihood maximization problem as in \cite{lyjb09}, where $x,\theta_{ijk}$ and $\kappa_{rst}$ are unknown variables to be determined. We leave the detailed analysis of noisy SNL with  SA and RoD measurements as a topic of our future work.
\end{remark} 
\subsection{SA-RoD localizability}

In this subsection, we study the uniqueness of the solution to Problem 1. 
First, we introduce the notion of {\it SA-RoD localizability}.
\begin{definition}
A sensor network is called SA-RoD localizable if there exists a unique feasible  solution to $(\ref{snl11})$.
\end{definition}

\begin{figure}[h]
\centerline{\includegraphics[width=\columnwidth]{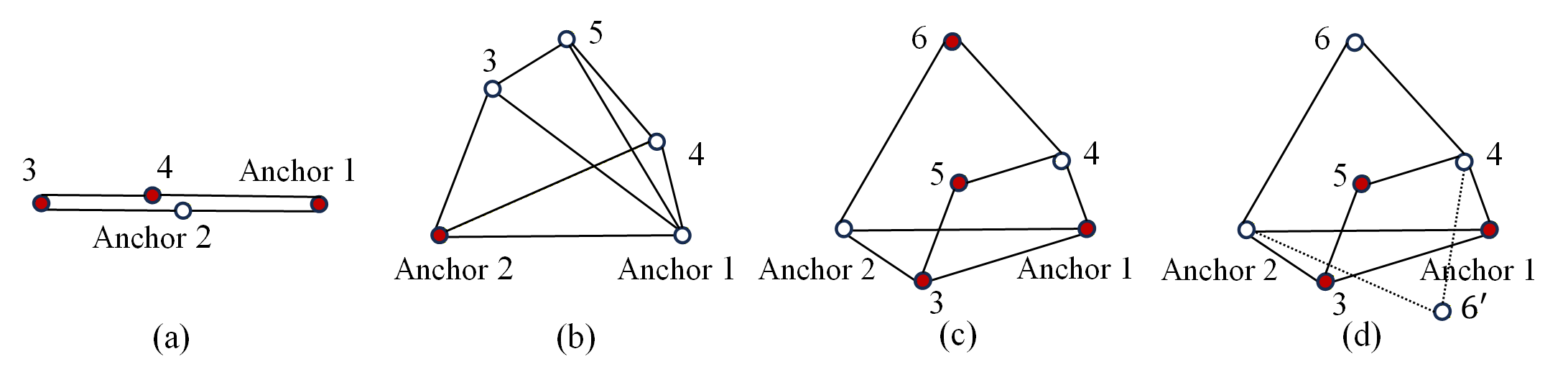}}
\caption{ (a) A sensor network that is not SA-RoD localizable. Node 4 can move freely on the segment connecting nodes 1 and 3. (b) A sensor network that is SA-RoD localizable. The key point is that vertex 5 has degree 3 which avoids possible reflections. (c) A sensor network that is SA-RoD localizable. (d) A sensor network that is not SA-RoD localizable. The location of sensor 6 can not be uniquely determined due to a possible flipping ambiguity. The dots in red represent nodes in $\mathcal{V}_{A}$.}
\label{lf11n}
\end{figure}
Figure $\ref{lf11n}$ presents four sensor networks to illustrate SA-RoD localizability. Note that the only difference between the two sensor networks in Figures $\ref{lf11n}(c)$ and  $\ref{lf11n}(d)$ lies in the type of measurements for sensor 6. This indicates the critical effects of bipartition on SA-RoD localizability. Furthermore, we observe that for localizable  networks in Figures $\ref{lf11n}(b)$ and $\ref{lf11n}(c)$, the underlying frameworks are globally SA-RoD rigid.
To establish the equivalence between SA-RoD localizability and global rigidity, we propose the following assumption for Problem 1.

\textit{Assumption 3:} Problem 1 is feasible, $\mathcal{V}_{a}\cap\mathcal{V}_{A}\neq\emptyset$, and  $\mathcal{V}_{a}\cap\mathcal{V}_{D}\neq\emptyset.$

Now we provide the following theorem. The proof is presented in Appendix \ref{peq}.
\begin{theorem}\label{equv}
Under Assumption 3, a sensor network $(\hat{\mathcal{G}}(A,D),x,\mathcal{V}_{a})$ is SA-RoD localizable if and only if the framework $(\hat{\mathcal{G}}(A,D),x)$ is globally SA-RoD rigid.
\end{theorem}
\subsection{An edge-based analysis}
In what follows, we provide an edge-based approach to solving Problem 1. Before that, we define the cycle bearing matrix for a framework.

\begin{definition}\label{cyc_beariing}
Let a graph cycle basis matrix of an oriented graph $\mathcal{G}$ be expressed as $C=[c_{ij}]\in\mathbb{R}^{(m-n+1)\times m}$, where $c_{ij}\in\{\pm 1,0\}$. Let $B=(\ldots,b_{ij},\ldots)\in\mathbb{R}^{2\times m}$ be the bearing matrix of a framework $(\mathcal{G},p)$, where $b_{ij}=\frac{p_{j}-p_{i}}{||p_{j}-p_{i}||},(i,j)\in\mathcal{E}$.
The cycle bearing matrix $C_{b}$ for $(\mathcal{G},p)$ is defined as $C_{b}=C\odot B\in\mathbb{R}^{2(m-n+1)\times m}.$
\end{definition}

Denote the distances of all edges by $d=(d_{1},\ldots,d_{m})^{\top}\in\mathbb{R}_{+}^{m}$. Note that on each cycle $\mathcal{C}_{k}$, we have the following compatibility condition \cite{lyjbb11}:
\begin{equation}\label{comp00}
    \sum\limits_{(i,j)\in\mathcal{E}}c_{k,ij}b_{ij}d_{ij}=0,\ k=1,2\ldots,m-n+1,
\end{equation}
where $c_{k,ij}\in\{\pm 1,0\}$. 
Using the cycle bearing matrix, (\ref{comp00}) can be written as:
\begin{equation}\label{comp1}
C_{b} d=0.
\end{equation}

To solve the SNL problem, we consider the following edge-based formulation: 

\textit{Problem 2:}\footnote{Here it is assumed that $m=|\hat{\mathcal{E}}|$.}
\begin{align}\label{snl12}
\text{find}\ d_{ij},\  &b_{ij}, (i,j)\in \mathcal{\hat{E}}
\notag
\\\text{s.t.}\ \  b_{ik}-\mathcal{R}(\theta_{ijk})b_{ij}&=0,\ (i,j,k)\in\mathcal{T}_{\hat{A}},
\notag
\\ d_{rt}-\kappa_{rst}d_{rs}&=0,\ (r,s,t)\in\mathcal{T}_{\hat{D}},
\notag
\\ \sum\limits_{(i,j)\in\mathcal{C}_{k}}c_{k,ij}d_{ij}b_{ij}&=0,\ k=1, 2, 3,\ldots, m-n+1,
\notag
\\ d_{ij}=d^{\ast}_{ij},\ b_{ij}&=b^{\ast}_{ij},\ i,j\in\mathcal{V}_{a},
\notag
\\ d_{ij}>0,\ ||b_{ij}||&=1,\ (i,j)\in\mathcal{\hat{E}}.
\end{align}

 Our next result reveals the intrinsic connection between Problem 1 and Problem 2. The proof is provided in Appendix $\ref{pequv}$.

\begin{theorem}\label{equv1}
Under Assumption 3, Problem 1 is equivalent to Problem 2, in the sense that an explicit bijection can be built between the solution sets of them.
\end{theorem}

Without loss of generality, let $l\in\mathcal{V}_{a}$ be the base vertex of a path matrix $P_{l}$ (see Definition $\ref{def:path}$). According to Lemma $\ref{lem:reco}$, the explicit mapping to recover the solution of Problem 1 from that of Problem 2 is 
\begin{equation}\label{eq:rec}
    x=\mathbf{1}_{n}\otimes x_{l}+P_{l}\odot B d.
\end{equation}
Comparing Problem 2 with Problem 1, we can find that the possible nonlinear constraints are caused by the compatibility constraint $(\ref{comp1})$ and the unit norm of bearings $||b_{ij}||=1$ for $(i,j)\in\hat{\mathcal{E}}$. All the other constraints in Problem 2 occur in linear forms, which will be very helpful for solving SNL in next subsection. 
\subsection{Centralized SNL via SA/RoD connectivity}\label{sub:loccon}
In this subsection, we try to solve Problem 2 in a centralized approach. In centralized SNL, a central unit collects anchor locations and measurements from all sensors to solve the localization problem.  
Problem 2 can be solved, for example, by first-order and second-order methods in classical optimization theory \cite{lyjbb15}. Interestingly, we will show that under some connectivity assumptions of $\mathcal{T}_{\hat{A}}$ and $\mathcal{T}_{\hat{D}}$, Problem 2 can be decoupled into two subproblems of finding bearings and distances, respectively. 
\subsubsection{SNL with connected $\mathcal{T}_{\hat{A}}$ over $\mathcal{\hat{G}}$}
As in \cite[Lemma 6]{lyjbb10}, all the SAs corresponding to the total triplet set $\mathcal{T}_{\mathcal{\hat{G}}}$ can be uniquely determined when the SA set $\mathcal{T}_{\hat{A}}$ is connected over $\mathcal{\hat{G}}$. By solving Subproblem 1, all bearings of edges can be derived.

\textit{Subproblem 1:}
\begin{equation}\label{snl13}
\left\{
\begin{aligned}
b_{ik}-\mathcal{R}(\theta_{ijk})b_{ij}&=0,\ (i,j,k)\in\mathcal{T}_{\hat{A}},
\\ b_{ij}&=b^{\ast}_{ij},\ i,j\in\mathcal{V}_{a},
\end{aligned}
\right.
\end{equation}
where $b_{ij}, b_{ik} ((i,j,k)\in\mathcal{T}_{\hat{A}})$  are the variables to be determined.

Then, we propose Subproblem 2 to find distances of all edges in the original framework.

\textit{Subproblem 2:}
\begin{equation}\label{snl14}
\left\{
\begin{aligned}
d_{ik}-\kappa_{ijk}d_{ij}&=0,\ (i,j,k)\in\mathcal{T}_{\hat{D}},
\\ \sum\limits_{(i,j)\in\mathcal{C}_{k}}c_{k,ij}b_{ij}d_{ij}&=0,\ k=1, 2, 3,\ldots, m-n+1,
\\ d_{ij}&=d^{\ast}_{ij},\ i,j\in\mathcal{V}_{a},
\end{aligned}
\right. 
\end{equation}
where $d_{ij}, d_{ik} ((i,j,k)\in\mathcal{T}_{\hat{D}})$  are the variables to be determined.
Denote $s=|\mathcal{T}_{\hat{D}}|+\frac{n_{a}(n_{a}-1)}{2}+2(m-n+1)$. In addition, $D_{0}\in\mathbb{R}^{[|\mathcal{T}_{D}|+\frac{n_{a}(n_{a}-1)}{2}]\times m}$ is defined as
\begin{equation}\label{CM}
    \bordermatrix{%
&\ldots	& \text{edge}\  (i,j)  & \ldots     & \text{edge}\  (i,k)  & \ldots& \text{edge}\  (s,t)  & \ldots \cr
 \ldots &\ldots & \ldots & \ldots & \ldots & \ldots& \ldots & \ldots \cr
	\kappa_{ijk}&\mathbf{0}  & -\kappa_{ijk}& \mathbf{0}&1& \mathbf{0} & 0  & \mathbf{0}\cr
 \ldots&\ldots & \ldots & \ldots & \ldots & \ldots &\ldots  & \ldots\cr
	d_{st} &\mathbf{0} & 0 & \mathbf{0} & 0& \mathbf{0}& 1 & \mathbf{0}\cr
 \ldots &\ldots & \ldots & \ldots & \ldots & \ldots& \ldots  & \ldots
    },
\end{equation}
where the rows are indexed by triplets in $\mathcal{T}_{\hat{D}}$ and distances among anchors. If we denote $C_{D}=(C_{b}^{\top},D_{0}^{\top})^{\top}\in\mathbb{R}^{s\times m}$, $y=(0,\ldots,0,d^{\ast}_{ij},\ldots)^{\top}\in\mathbb{R}^{s}$, where $C_{b}$ is the cycle bearing matrix, then Equation $(\ref{snl14})$ can be compactly written as
\begin{equation}\label{snl15}
C_{D}d=y.
\end{equation}
Now we summarize the localization approach in Algorithm $\ref{alg:loca}$. 
\begin{algorithm}
\caption{SNL with connected $\mathcal{T}_{\hat{A}}$ over $\mathcal{\hat{G}}$}
\label{alg:loca}
\begin{algorithmic}
\STATE{Find bearings of all edges by solving $(\ref{snl13})$}
\STATE{Find distances of all edges by solving $(\ref{snl15})$}
\STATE{Compute locations of unknown sensors based on $(\ref{eq:rec})$}
\end{algorithmic}
\end{algorithm}
For Algorithm $\ref{alg:loca}$,  we can establish the following algebraic criterion.
\begin{theorem}\label{snllp}
Under Assumption 3, if $\mathcal{T}_{\hat{A}}$ is connected over $\mathcal{\hat{G}}$, then \\
1) $(\mathcal{\hat{G}}(A,D),x,\mathcal{V}_{a})$ is localizable if and only if $\rank(C_{D})=m$;\\
2) Algorithm \ref{alg:loca} solves Problem 1 for sensor network $(\mathcal{\hat{G}}(A,D),x,\mathcal{V}_{a})$ under the condition $\rank(C_{D})=m$. 
\end{theorem}
\begin{proof}
1) Since $\mathcal{T}_{\hat{A}}$ is connected over $\mathcal{\hat{G}}$, all bearings of edges can be uniquely determined by $(\ref{snl13})$. According to Theorem \ref{equv1}, $(\mathcal{\hat{G}}(A,D),x,\mathcal{V}_{a})$ is localizable if and only if there exists a unique solution to Problem 2. Equivalently, there exists a unique solution to $(\ref{snl15})$. Based on Assumption 3, there exists a feasible solution to $(\ref{snl15})$. Therefore, the uniqueness of the solution to $(\ref{snl15})$ is equivalent to  $\rank(C_{D})=m$.\\
2) This conclusion can be obtained as a consequence of 1).
\end{proof}

\subsubsection{SNL with connected $\mathcal{T}_{\hat{D}}$ over $\mathcal{\hat{G}}$}

Similarly, if $\mathcal{T}_{\hat{D}}$ is connected over $\mathcal{\hat{G}}$, distances of all edges can be uniquely determined by solving Subproblem 3.

\textit{Subproblem 3:}
\begin{equation}\label{snl16}
\left\{
   \begin{aligned}
d_{ik}-\kappa_{ijk}d_{ij}&=0,\ (i,j,k)\in\mathcal{T}_{\hat{D}},
\\ d_{ij}&=d^{\ast}_{ij},\ i,j\in\mathcal{V}_{a},
\end{aligned} 
\right.
\end{equation}
where $d_{ij}, d_{ik} ((i,j,k)\in\mathcal{T}_{\hat{D}})$  are the variables to be determined.

Then, finding bearings of all edges can be achieved by solving Subproblem 4.

\textit{Subproblem 4:}
\begin{equation}\label{snl17}
\left\{
  \begin{aligned}
\ \ \sum\limits_{(i,j)\in\mathcal{C}_{k}}c_{k,ij}b_{ij}d_{ij}&=0,\ k=1, 2, 3,\ldots, m-n+1,
\\ b_{ik}-\mathcal{R}(\theta_{ijk})b_{ij}&=0,\ (i,j,k)\in\mathcal{T}_{\hat{A}},
\\ b_{ij}&=b^{\ast}_{ij},\ i,j\in\mathcal{V}_{a},
\\ ||b_{ij}||&=1,\ (i,j)\in\mathcal{\hat{E}},
\end{aligned}
\right.
\end{equation}
where $b_{ij}, b_{ik} ((i,j,k)\in\mathcal{T}_{\hat{A}})$  are the variables to be determined.
Denote $C_{B,1}=[C\odot(\ldots,d_{ij},\ldots)]\otimes I_{2}$, $t=2(m-n+1)+2|\mathcal{T}_{\hat{A}}|+n_{a}(n_{a}-1)$, $C_{B}=(C_{B,1}^{\top},C_{B,2}^{\top})^{\top}\in\mathbb{R}^{t\times2m}$, $b=(\ldots,b^{\top}_{ij},\ldots)^{\top}\in\mathbb{R}^{2m}$, $z=(0,\ldots,0,\ldots,b^{\ast \top}_{ij},\ldots)^{\top}\in\mathbb{R}^{t}$, where  $C_{B,2}$ is defined by 
\begin{equation}\label{CMn}
    \bordermatrix{%
&\ldots	& \text{edge}\  (i,j)  & \ldots     & \text{edge}\  (i,k)  & \ldots& \text{edge}\  (s,t)  & \ldots \cr
 \ldots &\ldots & \ldots & \ldots & \ldots & \ldots& \ldots & \ldots \cr
	\theta_{ijk}&\mathbf{0}  & -\mathcal{R}(\theta_{ijk})& \mathbf{0}&I_{2}& \mathbf{0} & 0_{2\times2}  & \mathbf{0}\cr
 \ldots&\ldots & \ldots & \ldots & \ldots & \ldots &\ldots  & \ldots\cr
	b_{st} &\mathbf{0} & 0_{2\times2} & \mathbf{0} & 0_{2\times2} & \mathbf{0}& I_{2} & \mathbf{0}\cr
 \ldots &\ldots & \ldots & \ldots & \ldots & \ldots& \ldots  & \ldots
    },
\end{equation}
where the rows are indexed by triplets in $\mathcal{T}_{\hat{A}}$ and bearings among anchors.
Thus, a compact form of $(\ref{snl17})$ is 
\begin{equation}\label{snl18}
  \left\{  
  \begin{aligned}
C_{B}b&=z,
\\ ||b(2k-1:2k)||&=1,\ k=1,2,\ldots,m.
\end{aligned}
\right.
\end{equation}
From Assumption 3, we know that there exists at least one solution to $(\ref{snl17})$. First, using Penrose-Moore inverse, all solutions of $C_{B}b=z$ can be represented as 
\begin{equation}\label{eq:b1}
    b=C^{\dagger}_{B}z+(\tilde{b}_{1},\ldots,\tilde{b}_{L})w=C^{\dagger}_{B}z+\sum_{i=1}^{L}w_{i}\tilde{b}_{i},
\end{equation}
where $L=\dim(\nulll(C_{B}) )$, $\nulll( C_{B})=\spann\{\tilde{b}_{i}\}^{L}_{i=1}$, and $w=(w_{1},w_{2},\ldots,w_{L})^{\top}\in\mathbb{R}^{L}$ is the variable to be determined. Let $N_{B}=(\tilde{b}_{1},\ldots,\tilde{b}_{L})\in\mathbb{R}^{m\times L}$, then $(\ref{eq:b1})$ can be rewritten as \begin{equation}\label{snl19}
    b=C^{\dagger}_{B}z+N_{B}w.
\end{equation} Denote $E_{k}=\diag(\mathbf{0},1,1,\mathbf{0})\in\mathbb{R}^{2m\times 2m}$, where $1$ occurs at the $(2k-1)$-th and $2k$-th locations. So the norm constraints $||b(2k-1:2k)||=1$ can be expressed as 
\begin{align}\label{snl20}
(C^{\dagger}_{B}z+N_{B}w)^{\top}E_{k}(C^{\dagger}_{B}z+N_{B}w)=1,\ k=1,2,\ldots,m.
\end{align}
Denote $H(w)=\sum\limits_{k=1}^{m}[(C^{\dagger}_{B}z+N_{B}w)^{\top}E_{k}(C^{\dagger}_{B}z+N_{B}w)-1]^{2}$, $(\ref{snl20})$ can be equivalently transformed into:
\begin{equation}\label{snl21}
H(w)=0,\ w\in\mathbb{R}^{L}.
\end{equation} 
 Despite the difficulty caused by the non-convexity of $H(\cdot)$, this problem has been simplified to an equation on the null space. 
For Problem 2, bearings and distances are determined by solving $(\ref{snl16})$ and $(\ref{snl21})$. 
In this case, the localization approach is summarized into Algorithm $\ref{alg:locd}$.

\begin{algorithm}
\caption{SNL with connected $\mathcal{T}_{\hat{D}}$ over $\mathcal{\hat{G}}$}
\label{alg:locd}
\begin{algorithmic}
\STATE{Find distances of all edges by solving $(\ref{snl16})$}
\STATE{Find bearings of all edges by solving $(\ref{snl19})$ and $(\ref{snl21})$}
\STATE{Compute locations of unknown sensors based on $(\ref{eq:rec})$}
\end{algorithmic}
\end{algorithm}

We note that solving $(\ref{snl21})$ for large-scale networks remains a time-consuming task. Fortunately, for some frameworks such as those constructed by Type $(D_{1},A_{1})$ bilaterion orderings (see Figure $\ref{lf117}(a)$), a concise description of $\nulll(C_{B})$ can be derived to remove the quadratic constraints $||b(2k-1:2k)||=1$ in $(\ref{snl18})$. We will show that under this setting, the matrix $C_{B}\in\mathbb{R}^{t\times (4n-6)}$ is of full column rank.
 Now we present the following theorem. The proof can be found in Appendix $\ref{app:ful}$. 
\begin{theorem}\label{pr1}
Under Assumption 3, suppose that $(\mathcal{\hat{G}}(A,D),x)$ is generated by a Type $(D_{1},A_{1})$ bilateration ordering and $\mathcal{V}_{a}=\{1,2\}$, then\\ 
1) $\rank(C_{B})=4n-6$, and bearings of all edges are determined by $b=C^{\dagger}_{B}z$;\\
2) Algorithm \ref{alg:locd} solves Problem 1 for sensor network $(\mathcal{\hat{G}}(A,D),x,\mathcal{V}_{a})$.
\end{theorem}

\subsection{Distributed SNL}
\label{dis_SNL}
In this subsection, we propose a distributed algorithm for the SNL problem. Different from the centralized approach, each sensor only needs to estimate its own location by leveraging local information from its neighbor sensors. Note that in previous discussion, SNL with connected $\mathcal{T}_{\hat{D}}$ (or disconnected $\mathcal{T}_{\hat{A}}$ and $\mathcal{T}_{\hat{D}}$) over $\mathcal{\hat{G}}$ can be transformed into nonconvex optimization problems, where multiple local minimizers or saddle points may occur. Therefore, designing distributed SNL algorithms with global convergence for general cases remains a challenging task. Here we consider the scenario with connected $\mathcal{T}_{\hat{A}}$ over $\mathcal{\hat{G}}$.

\textit{Assumption 4:} The sensor network $(\mathcal{G}(A,D),x,\mathcal{V}_{a})$ satisfies: (i)  the SA index set $\mathcal{T}_A$ is connected over graph $\mathcal{G}$; (ii) $(\mathcal{G}(A,D),x,\mathcal{V}_{a})$ is SA-RoD localizable; (iii) there exist two anchors $\{i,j\}\subseteq\mathcal{V}_a$ satisfying $(i,j)\in\mathcal{E}$.

To achieve distributed SNL, we propose a two-stage SNL strategy. In the first stage, since $\mathcal{T}_{A}$ is connected over $\mathcal{G}$, the bearings of all edges can be estimated as in \cite{lyjbb10}. Specifically, for each node $i$ and $j\in\mathcal{N}_i$, $\hat{b}_{ij}(t)$ can be updated by
\begin{align}
\dot{\hat{b}}_{ij}= -&\Big[\sum_{(i,k,j)\in\mathcal{T}_{\mathcal{A}}}\Big(\hat{b}_{ij}-\mathcal{R}(\theta_{ikj})\hat{b}_{ik}\Big)+\sum_{(i,j,k)\in\mathcal{T}_{\mathcal{A}}}\Big(\hat{b}_{ij}-\mathcal{R}^\top(\theta_{ijk})\hat{b}_{ik}(t)\Big)
\notag
\\&+\sum_{(j,k,i)\in\mathcal{T}_{\mathcal{A}}}\Big(\hat{b}_{ij}+\mathcal{R}(\theta_{jki})\hat{b}_{jk}\Big)+\sum_{(j,i,k)\in\mathcal{T}_{\mathcal{A}}}\Big(\hat{b}_{ij}+\mathcal{R}^\top(\theta_{jik})\hat{b}_{jk}\Big)\Big],
\label{eq:estiSA}
\end{align}
where $\theta_{jki},\ \theta_{jik}$, and $\hat{b}_{jk}$ are obtained by node $i$ through communications with node $j\in\mathcal{N}_i$. In fact, Equation (\ref{eq:estiSA}) can be viewed as a decentralized approach to solving Subproblem 1. Furthermore, it has been shown that $\hat{b}_{ij}(t)$ exponentially converges to $b_{ij}$ under any initial estimate $\hat{b}_{ij}(0)$ \cite{lyjbb10}. 

In the second stage, we concentrate on estimates of distances and locations. Denote 
\begin{equation}
V(\hat{x},\hat{d})=\frac{1}{2}\Big(\sum_{(i,j)\in\mathcal{E}}||\hat{x}_j-\hat{x}_i-\hat{d}_{ij}b_{ij}||^2+\sum_{(i,j,k)\in\mathcal{T}_{D}}(\hat{d}_{ik}-\kappa_{ijk}\hat{d}_{ij})^2\Big),
\label{eq:defcv}
\end{equation}
where $\{b_{ij}\}_{(i,j)\in\mathcal{E}}$ corresponds to the equilibrium of system (\ref{eq:estiSA}), $\hat{d}_{ij}>0$ for each $(i,j)\in\mathcal{E}$, $\hat{x}_i=x_i$ for each $i\in\mathcal{V}_a$, and $\hat{d}_{ij}=d_{ij}$ for $i,j\in\mathcal{V}_a$. Note that $V(\hat{x},\hat{d})$ is a convex function in $\mathbb{R}^{2n}\times\mathbb{R}^{2m}_{+}$. Observe that $V(x,d)=0$, i.e., $(x,d)$ is a global minimizer of $V$. Based on the localizability of $(\mathcal{G}(A,D),x,\mathcal{V}_{a})$ and the convexity of $V$, we can naturally derive the following lemma.
\begin{lemma}\label{lem:conuni}
 Under Assumption 4, $(x,d)$ is the unique global minimizer of $V$ in $\mathbb{R}^{2n}\times\mathbb{R}^{2m}_{+}$.
\end{lemma}

Now we can jointly estimate the distances and locations using the gradient flow of $V$, i.e.,
\begin{align}\label{eq:estiRoDLoc}
\dot{\hat{x}}_i&=-\sum_{j\in\mathcal{N}_i}(\hat{x}_i-\hat{x}_j+\hat{d}_{ij}b_{ij}),
\notag
\\ \dot{\hat{d}}_{ij}&=-\hat{d}_{ij}+b_{ij}^\top(\hat{x}_j-\hat{x}_i)+\sum_{(i,j,k)\in\mathcal{T}_D}\kappa_{ijk}(\hat{d}_{ik}-\kappa_{ijk}\hat{d}_{ij})-\sum_{(i,k,j)\in\mathcal{T}_D}(\hat{d}_{ij}-\kappa_{ikj}\hat{d}_{ik})
\notag
\\ &\ \ \ +\sum_{(j,i,k)\in\mathcal{T}_D}\kappa_{jik}(\hat{d}_{jk}-\kappa_{ijk}\hat{d}_{ij})-\sum_{(j,k,i)\in\mathcal{T}_D}(\hat{d}_{ij}-\kappa_{jki}\hat{d}_{jk}), (i,j)\in\mathcal{E},
\end{align}
where $b_{ij}$ is the true bearing of edge $(i,j)$, $\kappa_{jik},\kappa_{jik}$, and $\hat{d}_{jk}$ can be obtained through communications with node $j\in\mathcal{N}_i$. Note that if $i,j\in\mathcal{V}_A$, $\hat{d}_{ij}$ is updated according to
$\dot{\hat{d}}_{ij}=-\hat{d}_{ij}+b^\top_{ij}(\hat{x}_j-\hat{x}_i).$
 
 Thanks to the convexity of $V$, all equilibria of system (\ref{eq:estiRoDLoc}) are global minimizers of $V$. From Lemma $\ref{lem:conuni}$, we know system (\ref{eq:estiRoDLoc}) has a unique equilibrium. Therefore, we can establish the following global convergence theorem.
\begin{theorem}[Global convergence property]
Under Assumption 4,  by implementing protocol (\ref{eq:estiSA}) in the first stage and (\ref{eq:estiRoDLoc}) in the second stage, $\hat{b}_{ij}, \hat{d}_{ij}, \hat{x}_{i}$ exponentially converge to $b_{ij}, d_{ij},x_i$, respectively, for all the free nodes $i\in\mathcal{V}_{f}$ under any initial estimates $\hat{x}_i(0), \hat{b}_{ij}(0), \hat{d}_{ij}(0)>0$. 
\label{con_dis_SNL}
\end{theorem}
\begin{proof}
In the first stage, the global exponential stability (GES) of (\ref{eq:estiSA}) at $\{b_{ij}\}_{(i,j)\in\mathcal{E}}$ has been proved in \cite{lyjbb10}. In the second stage, for (\ref{eq:estiRoDLoc}), Lemma $\ref{lem:conuni}$ implies that $V(\hat{x},\hat{d})>0$ in $\mathbb{R}^{2n}\times\mathbb{R}_{+}^{2m}\setminus\{(x,d)\}$ and $V(x,d)=0$. Furthermore, $\dot{V}(\hat{x},\hat{d})=-||\nabla V(\hat{x},\hat{d})||^2\leq 0$. From the convexity of $V$, $||\nabla V(\hat{x},\hat{d})||=0$ if and only if $(\hat{x},\hat{d})$ is a global minimizer, i.e., $(\hat{x},\hat{d})=(x,d)$. That is, $\dot{V}(\hat{x},\hat{d})<0$ in $\mathbb{R}^{2n}\times\mathbb{R}_{+}^{2m}\setminus\{(x,d)\}$ and $\dot{V}(x,d)=0$. By Lyapunov's stability theorem \cite[Theorem 4.1]{2002nonlinear}, $(x,d)$ is an asymptotically stable equilibrium for (\ref{eq:estiRoDLoc}). Since system (\ref{eq:estiRoDLoc}) is linear and has a unique equilibrium $(x,d)$,  system (\ref{eq:estiRoDLoc}) is GES at $(x,d)$.
\end{proof}
%
\section{Numerical simulations}
In this section, we provide some numerical examples to validate the efficacy of the proposed SNL method. 
\subsection{Simulations for centralized SNL}
In this subsection, the nodes of sensor networks in all examples are supposed to lie in the unit box $[0,1]^{2}$ randomly.
 Furthermore, the number of nodes is 70 and there exist 2 anchors in each network. 


\textit{Example 1:} Figure $\ref{lfn1602}(a)$ shows a quadrilateralized sensor network generated via 2-vertex additions, which has 103 edges and is SA connected over the underlying graph. Computational results indicate that $\rank(C_{D})=103$. By implementing Algorithm $\ref{alg:loca}$, the estimated positions of all free nodes match their true locations closely, as illustrated in Figure $\ref{lfn1602}(b)$. In this context, all governing equations are linear, and each linear equation  $\mathcal{M}x=y$ ($x$ and $y$ are vectors, and $\mathcal{M}$ is a matrix of suitable size) is solved with the MATLAB command "$\mathcal{M}\backslash y$". Here the mean squared error between the estimated locations and true locations is calculated as $8.1398e{-15}$. 

\textit{Example 2:} The sensor network in Figure $\ref{lfn1602}(c)$ is generated by a Type $(D_{1},A_{1})$ bilateration ordering, which is RoD connected over the underlying graph. Since the underlying graph is a Laman graph, the number of edges is $m=2n-3=137$. First, after computation, we find that $\rank(C_{B})=4n-6=274$ and derive a unique solution to $(\ref{snl16})$, which is consistent with Theorem $\ref{pr1}$. Then we estimate the locations of unknown sensors using Algorithm $\ref{alg:locd}$. Similarly, this process only involves linear equation solvers. As shown in Figure $\ref{lfn1602}(d)$, the estimated locations also match the true locations of nodes closely. In this case, the mean squared error is $7.2709e-14$. 
\begin{figure}[!h]
\centerline{\includegraphics[width=\columnwidth]{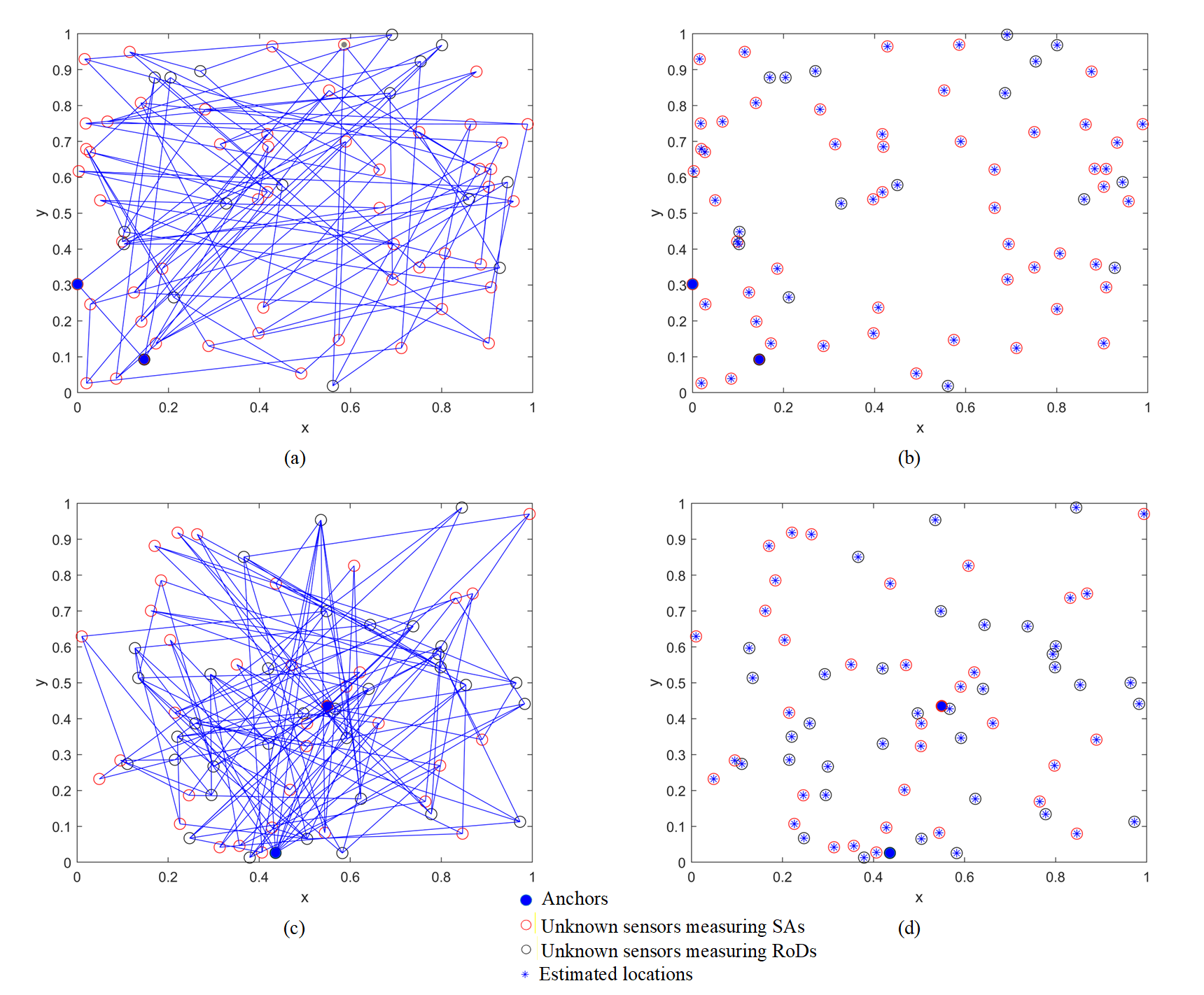}}
\caption{(a) A quadrilateralized network that is SA connected and constructed by 2-vertex additions and satisfies $\rank(C_{D}) =103$. (b) Localization results. (c) A RoD connected sensor network constructed by a Type $(D_{1}, A_{1})$ bilateration ordering. (d)  Localization results. }
\label{lfn1602}
 \end{figure}



\subsection{Simulations for distributed SNL} 
Consider a sensor network of seven nodes, which satisfies Assumption 4 as shown in Figures \ref{lfdis}(a). Note that in Figure \ref{lfdis}(a), the anchor set $\mathcal{V}_a=\{1,2\}$, and the bipartition is determined by $\mathcal{V}_A=\{1,3,4,6,7\},\ \mathcal{V}_D=\{2,5\}$. By implementing (\ref{eq:estiSA}) in the first stage and (\ref{eq:estiRoDLoc}) in the second stage, the evolutions of bearing, distance, and location estimation errors are illustrated, respectively, in Figure \ref{lfdis}(b). We can observe that the estimated bearings, distances, and positions asymptotically converge to the corresponding true values, respectively. This is consistent with Theorem \ref{con_dis_SNL}.
 \begin{figure}[!htbp]
 \centerline{\includegraphics[width=\columnwidth]{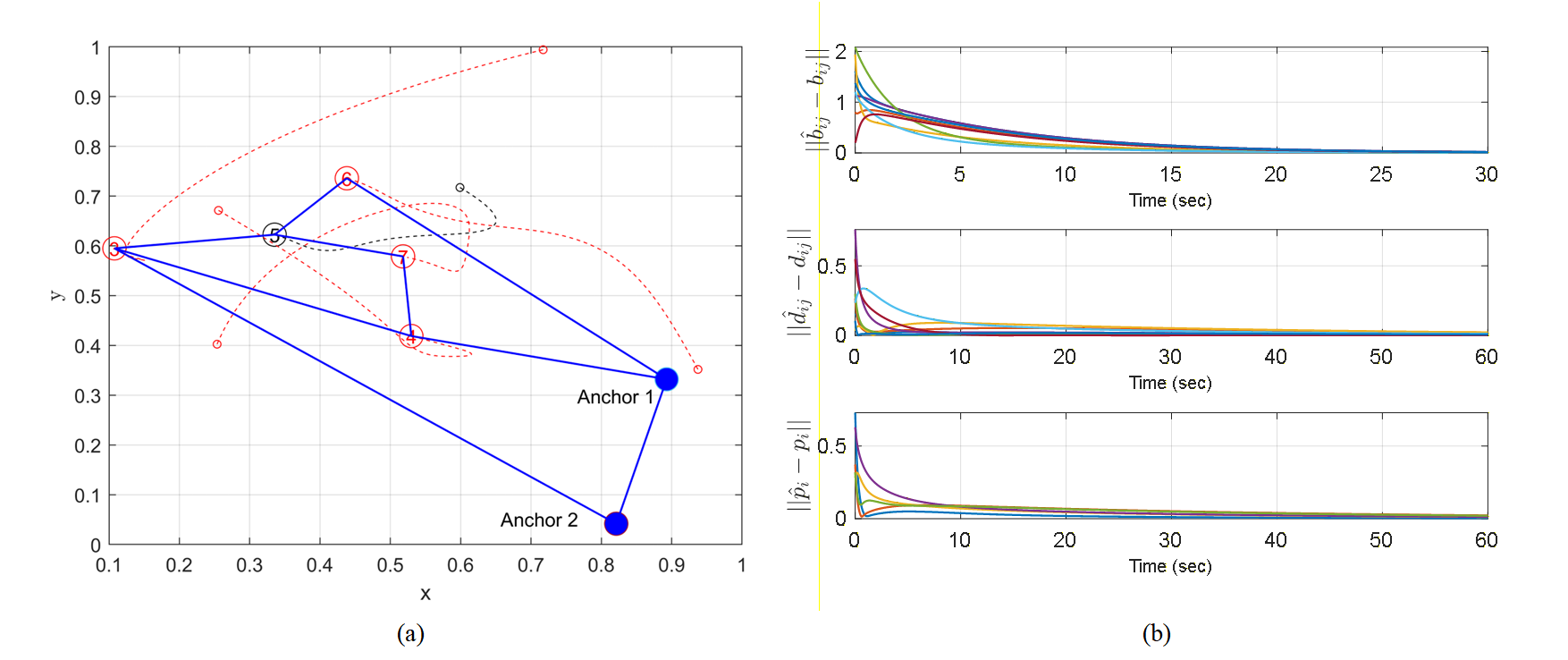}}
 \caption{ (a) Trajectories of estimated positions for five free nodes, $\mathcal{V}_a=\{1,2\}$, $\mathcal{V}_A=\{1,3,4,6,7\}$. (b) Evolution of bearings, distances, and locations estimation errors. The blue solid dots represent anchors. Free nodes sensing SA and RoD are denoted by hollow red dots and hollow black dots, respectively. The smaller hollow dots represent initial estimates.}
 \label{lfdis}
 \end{figure}
\section{Conclusions and future directions}
This work established the SA-RoD rigidity theory for frameworks with bipartition on the vertex set. It has been shown that a duality theorem holds for infinitesimally SA-RoD rigid frameworks with respect to bipartition. Furthermore, if the bipartition-induced triplet set $\mathcal{T}_{A}$ or $\mathcal{T}_{D}$ is connected over the underlying graph $\mathcal{G}$, the framework $(\mathcal{G}(A,D),p)$ can inherit the SA or RoD rigidity property of ($\mathcal{G},p)$. We have also proposed an approach called SA-RoD ordering for construction of globally SA-RoD rigid frameworks with specific bipartions. The SA-RoD rigidity theory was applied to the problem of SNL with heterogeneous nodes. By using an edge-based approach, we found that the connectivity of $\mathcal{T}_{\hat{A}}$ or $\mathcal{T}_{\hat{D}}$ over $\mathcal{\hat{G}}$ can help to decouple the centralized SNL problem into two simple subproblems. A distributed SNL protocol was also proposed for sensor networks with connected SA index sets. 

Meaningful future work may include: 1) extending the proposed rigidity theory to frameworks in $\mathbb{R}^{3}$, 2) investigating SNL under noisy measurements, 3) finding an appropriate bipartition strategy to achieve faster convergence of the SNL algorithms.

\appendix
\begingroup
\fontsize{15.8pt}{20pt}\selectfont
\noindent\textbf{6\ \ \ \ Appendix}
 \par\endgroup

\section{Preliminaries}
\label{pre}
In this section, we will briefly review some mathematical tools in algebraic graph theory and results in graph rigidity theory.
\subsection {Algebraic graph theory}
Given a graph $\mathcal{G}=(\mathcal{V},\mathcal{E})$, a {\it path} of length $l$ is a sequence of $l + 1$ distinct vertices such that any two consecutive vertices are adjacent.  A {\it cycle} $\mathcal{C}$ is a connected graph where each vertex has exactly two neighbors. A graph with no cycles is called {\it acyclic}. A {\it tree} refers to a connected acyclic graph. A {\it spanning tree} is an acyclic spanning subgraph.
An {\it orientation} of a graph $\mathcal{G}$ is the assignment of a direction to each edge. Equivalently, it is a function $\sigma$ of all arcs such that if $(i,j)$ is an arc, then $\sigma(i,j)=-\sigma(j,i)$. A graph together with a specific orientation is called an {\it oriented graph}. For an oriented graph $\mathcal{G}$, $H=[h_{ij}]\in \mathbb{R}^{m\times n}$ represents the {\it incidence matrix} with rows and columns indexed by edges and vertices of $\mathcal{G}$. $h_{ij}=1$ if the $i$-th edge sinks at vertex $j$, $h_{ij}=-1$ if the $i$-th edge leaves vertex $j$, and $h_{ij}=0$ otherwise. 

\begin{definition}[Cut space and flow space\cite{lyjb1}] For an oriented graph $\mathcal{G}$, the cut space is a subspace of $\mathbb{R}^{m}$ corresponding to the column space of the incidence matrix $H$, and the flow/cycle space is the orthogonal complement of the cut space.
\end{definition}

For a cycle $\mathcal{C}_{l}$ of length $l$ in the oriented graph $\mathcal{G}$, one can define the {\it oriented cycle} with a signed characteristic vector $z=(z_{1},z_{2},\ldots,z_{m})^{\top}\in \mathbb{R}^{m}$, where $z_{i}=1$ if $i$ is consistent with the orientation, $z_{i}=-1$ if $i$ is opposite to the orientation, and $z_{i}=0$ otherwise.
\begin{lemma}[Dimension and base of flow space\cite{lyjb1}]  Given a connected graph $\mathcal{G}$ with $n$ vertices and $m$ edges, its flow space has dimension $m - n + 1$. Furthermore, this space is spanned by the signed characteristic vectors of its cycles. 
\label{lem1}
\end{lemma}

It should be noted that these independent cycles can be obtained by adding chords to a maximal spanning tree of $\mathcal{G}$. A matrix with rows corresponding to $m-n+1$ independent signed characteristic vectors of cycles is called a {\it cycle basis matrix}.
\subsection{ Graph rigidity theory}\label{deg:grt}

For a graph $\mathcal{G}$, if every vertex $i$ is assigned with a coordinate $p_{i}\in \mathbb{R}^{d}$, then $p={{(p_{1}^{\top},p_{2}^{\top},\ldots, p_{n}^{\top})}^{\top}}\in\mathbb{R}^{nd}$ is called a {\it configuration} of graph $\mathcal{G}$, and $(\mathcal{G},p)$ is called a {\it framework}. A configuration $p\in\mathbb{R}^{nd}$ is said to be {\it non-degenerate} if $p_{1},p_{2},\ldots, p_{n}$ do not lie in a common hyperplane of $\mathbb{R}^{d}$, and $(\mathcal{G},p)$ is called a {\it non-degenerate framework}. 

In different types of rigidity theory, the {\it rigidity function} $f_{\mathcal{G}}:\mathbb{R}^{nd}\to\mathbb{R}^{q}$ is commonly defined. For example, the distance rigidity function is 
\begin{equation}
{{f}_{\mathcal{G}}}(p)\triangleq{{(\ldots ,d_{ij}^{2},\ldots )}^{\top}},\text{ }(i,j)\in \mathcal{E},
\end{equation}
 where $d_{ij}=||{{p}_{i}}-{{p}_{j}}||$. We say a framework $(\mathcal{G},p)$ is {\it rigid} if $f_{\mathcal{G}}^{-1}(f_{\mathcal{G}}(p))\cap{{{U}_{p}}}=f_{\mathcal{K}}^{-1}({{f}_{\mathcal{K}}}(p))\cap{{{U}_{p}}}$ for some neighborhood ${{U}_{p}}\subseteq\mathbb{R}^{nd}$ of $p$. If $f_{\mathcal{G}}^{-1}(f_{\mathcal{G}}(p))=f_{\mathcal{K}}^{-1}({f_{\mathcal{K}}}(p))$, $(\mathcal{G},p)$ is said to be {\it globally rigid}. 
$(\mathcal{G},p)$ is {\it minimally rigid} in $\mathbb{R}^{d}$ if there does not exist any subgraph $\mathcal{H}$ with $n$ vertices and less than $m$ edges such that $(\mathcal{H},p)$ is rigid in $\mathbb{R}^{d}$.
 The Jacobi matrix $R_{\mathcal{G}}(p)=\frac{\partial f_{\mathcal{G}}}{\partial p}$ is called the {\it rigidity matrix}. Vectors in the kernel of $R_{\mathcal{G}}(p)$ are called {\it infinitesimal motions}, which ensure the invariance of the rigidity function $f_{\mathcal{G}}$. If the rigidity function is defined by RoDs (bearings, SAs), similar notions can also be proposed in RoD (bearing, SA) rigidity theory. 
Here we briefly review some results useful for later discussion.

\subsubsection{Distance rigidity theory versus RoD rigidity theory}
In RoD rigidity theory, the rigidity function is defined as 
\begin{equation}\label{def:RoD0}
{{f}_{\mathcal{G}}}(p)\triangleq{{(\ldots,\kappa_{ijk},\ldots)}^{\top}},\text{ }(i,j,k)\in\mathcal{T}_{\mathcal{G}},
\end{equation}
where $\mathcal{T}_{\mathcal{G}}=\{(i,j,k):(i,j),(i,k)\in\mathcal{E},j<k\}$, each RoD (Figure \ref{def:SA_RoD}(a)) is defined as  
\begin{equation}    \label{def:RoD}
    \kappa_{ijk}=\frac{d_{ik}^{l}}{d_{ij}^{l}},l\in\mathbb{Z}^{+}.
\end{equation}

\begin{figure}[!ht]
\centerline{\includegraphics[width=0.55\columnwidth]{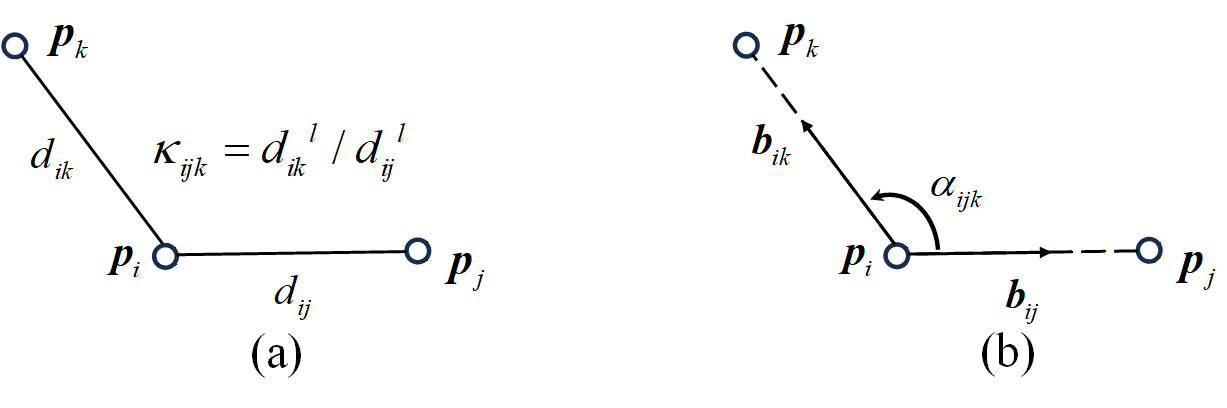}}
\caption{Illustration of RoD $\kappa_{ijk}$ and SA $\alpha_{ijk}$.}
\label{def:SA_RoD}
\end{figure}
The following lemma connects distance rigidity with RoD rigidity.
\begin{lemma}[Equivalence between RoD rigidity and distance rigidity theory\cite{lyjbb8.1}]\label{eqRoD}
   (Infinitesimal/Global) RoD rigidity of a framework is equivalent to its (infinitesimal/global) distance rigidity.
\end{lemma}

\subsubsection{Bearing rigidity theory versus SA rigidity theory}

In bearing rigidity theory, the rigidity function is defined as 
\begin{equation}
{{f}_{\mathcal{G}}}(p)\triangleq(\ldots ,b_{ij}^{\top},\ldots)^{\top},\text{ }(i,j)\in \mathcal{E},
\end{equation}
where $b_{ij}(p)=\frac{p_{j}-p_{i}}{||p_{j}-p_{i}||}$. The infinitesimal bearing rigidity of a framework implies its bearing rigidity, which is equivalent to its global bearing rigidity\cite{lyjb4}. 

In SA rigidity theory \cite{lyjb9,lyjbb10}, the rigidity function is defined as
\begin{equation}
{{f}_{\mathcal{G}}}(p)\triangleq(\ldots ,\alpha_{ijk},\ldots)^{\top},\text{ }(i,j,k)\in\mathcal{T}_{\mathcal{G}},
\end{equation}
where SA $\alpha_{ijk}\in[0,2\pi)$ (see Figure \ref{def:SA_RoD} (b)) is given by \footnote{Different from the labeling scheme used in \cite{lyjb9,lyjbb10}, we adopt $\alpha_{ijk}$ to represent the signed angle from $b_{ij}$ to $b_{ik}$ in the counter-clockwise direction.}
\begin{equation}\label{def:SA}
{{\alpha }_{ijk}}=
\begin{cases}
\text{acos}(b_{ij}^{\top}{{b}_{ik}}),& b_{ik}^{\top}\mathcal{R}(\pi/2){{b}_{ij}}\ge 0, \\ 
\text{2}\pi \text{-acos}(b_{ij}^{\top}{{b}_{ik}}), & b_{ik}^{\top}\mathcal{R}(\pi/2){{b}_{ij}}<0. 
\end{cases}
\end{equation}

The following lemmas connect SA rigidity theory with bearing rigidity theory and provide a criterion for shape determination by SAs.
\begin{lemma}[Equivalence between SA rigidity and bearing rigidity theory\cite{lyjbb10}]\label{eqSA}
    (Infinitesimal/Global) SA rigidity of a framework in $\mathbb{R}^{2}$ is equivalent to its (infinitesimal/global) bearing rigidity.
\end{lemma}
\begin{lemma}[Shape determination by SAs\cite{lyjbb10}]\label{eqSA1}
    Given a framework $(\mathcal{G}, p)$ in $\mathbb{R}^{2}$, the following
statements are equivalent:\\
(i) $(\mathcal{G}, p)$  is infinitesimally SA rigid;\\
(ii) $(\mathcal{G}, p)$ is non-degenerate and globally SA rigid;\\
(iii) The shape of $(\mathcal{G}, p)$  can be uniquely determined by SAs up to uniform rotations, translations, and scalings.
\end{lemma}
\subsubsection{Generic property}
A configuration $p$ is said to be {\it generic} if its coordinates do not satisfy any non-zero polynomial equations with integer coefficients. $(\mathcal{G},p)$ is called a {\it generic framework} if $p$ is generic. Generic configurations form a dense set, i.e., the set of all non-generic configurations is of measure zero\cite{lyjbb9.14}.

In distance rigidity theory, Laman graphs are perhaps most well-known due to their elegant properties and simple construction process\cite{laman1970graphs}. 
\begin{definition}[Laman graph]
Graph $\mathcal{G}=(\mathcal{V},\mathcal{E})$ is a Laman graph if $m=2n-3$ and there exist less than $2k-3$ edges for any subgraph of $\mathcal{G}$ with $k\geq 2$ vertices.
\end{definition}


For distance/bearing/SA constrained frameworks, underlying graphs play a key role in determining (global) rigidity. Specifically, infinitesimal distance/bearing/SA rigidity is a generic property of the underlying graph. That is, given a graph
 $\mathcal{G}$, suppose that there exists
  one generic configuration $p\in\mathbb{R}^{nd}$ such that $(\mathcal{G},p)$ is infinitesimally distance/bearing/SA rigid. Then for any generic configuration $q\in\mathbb{R}^{nd}$, $(\mathcal{G},q)$ is infinitesimally distance/bearing/SA rigid \cite{lyjbb9.14,lyjb4,lyjbb10}. 

\section{Proof of Lemma $\ref{lemm01}$}
\label{pol1}
\begin{proof}
Denote $D_{ij}=\frac{b_{ij}}{||e_{ij}||}, D_{ik}=\frac{b_{ik}}{||e_{ik}||}, D_{ijk}=D_{ij}-D_{ik}$. Let 
$A_{ij}=\mathcal{R}(\pi/2)D_{ij}$, $A_{ik}=\mathcal{R}(\pi/2)D_{ik}$,$ A_{ijk}=A_{ij}-A_{ik}$. The SA-RoD rigidity matrix $R_{\mathcal{G}_{A,D}}(p)$ can be expressed in the following form:
\begin{equation}\label{IRM}
    \bordermatrix{%
&\ldots	& \text{vertex}\  i  & \text{vertex}\  r      & \text{vertex}\  j  & \text{vertex}\  s &\text{vertex}\  k & \text{vertex}\  t&\ldots\cr
 \ldots &\ldots & \ldots & \ldots & \ldots & \ldots & \ldots& \ldots&\ldots\cr
		\alpha_{rst} &\mathbf{0} & 0_{1\times2} & A_{rst}^{\top} & 0_{1\times2}& -A_{rs}^{\top}& 0_{1\times2} & A_{rt}^{\top}&\mathbf{0}\cr
 \ldots&\ldots & \ldots & \ldots & \ldots & \ldots & \ldots& \ldots&\ldots\cr
    \kappa_{ijk}&\mathbf{0}  & \kappa_{ijk}D_{ijk}^{\top} & 0_{1\times2}& -\kappa_{ijk}D_{ij}^{\top} & 0_{1\times2} & \kappa_{ijk}D_{ik}^{\top} & 0_{1\times2}&\mathbf{0}\cr
 \ldots &\ldots & \ldots & \ldots & \ldots & \ldots & \ldots& \ldots &\ldots
    },
\end{equation}
where the rows are indexed by the triplet sets $\mathcal{T}_{D}$ and $\mathcal{T}_{A}$, the columns are indexed by the vertices. 

We consider the row vectors in $R_{\mathcal{G}(A,D)}$. If $i\in\mathcal{V}_{D}$, let $j_{0}=\argmin_{j\in\mathcal{N}_{i}}j$. After fixing the edge $(i,j_{0})\in\mathcal{E}$, consider the triplets at node $i$ formed by this edge with all other $\deg(i)-1$ edges. It holds that $\kappa_{ij_{0}j_{1}}\kappa_{ij_{1}j_{2}}=\kappa_{ij_{0}j_{2}}$ whenever $j_{0}<j_{1}<j_{2}\in\mathcal{N}_{i}$. The partial derivative with respect to $p$ can be obtained as
$\frac{\partial\kappa_{ij_{1}j_{2}}}{\partial p}=\kappa_{ij_{0}j_{1}}^{-1}[\frac{\partial\kappa_{ij_{0}j_{2}}}{\partial p}-\kappa_{ij_{1}j_{2}}\frac{\partial\kappa_{ij_{0}j_{1}}}{\partial p}].$
Thus, all the vectors in the set $\{\frac{\partial\kappa_{ijk}}{\partial p}:j<k\in\mathcal{N}_{i}\}$ can be written as linear combinations of vectors in the set $\{\frac{\partial\kappa_{ij_{0}k}}{\partial p}:k\in\mathcal{N}_{i},k\neq j_{0}\}$.
Then, we have $$\dim (\spann\{\frac{\partial\kappa_{ijk}}{\partial p}:j<k\in\mathcal{N}_{i}\})\leq |\{\frac{\partial\kappa_{ij_{0}k}}{\partial p}:k\in\mathcal{N}_{i},k\neq j_{0}\}|=\deg(i)-1.$$ Similarly, this conclusion is also valid for $i\in\mathcal{V}_{A}$. That is, the rank of the submatrix corresponding to constraints at vertex $i$ in $(\ref{IRM})$ is less than $\deg(i)-1$.  Using the rank inequality $\rank[X^{\top},Y^{\top}]^{\top}\leq \rank(X)+\rank(Y)$ inductively, we can derive that  $\rank(R_{\mathcal{G}_{A,D}}(p))\leq\sum\limits_{i\in\mathcal{V}}[\deg(i)-1]=2|\mathcal{E}|-|\mathcal{V}|=2m-n.$ Thus, the proof is finished.
\end{proof}

 \section{Proof of Lemma $\ref{lem1o}$}
 \label{pK}
 Hereafter, we use the abbreviated notation $\mathcal{G}$ or $\mathcal{K}$  instead of $\mathcal{G}(A,D)$ or $\mathcal{K}(A,D)$ whenever the bipartition is clear from the context.
 \begin{proof}
 Suppose that $f_{\mathcal{K}}(q)=f_{\mathcal{K}}(p)$ for some $q\in\mathbb{R}^{2n}$. Assume $i\in\mathcal{V}_{A},j\in\mathcal{V}_{D}$, let $c=\frac{||q_{i}-q_{j}||}{||p_{i}-p_{j}||}$ and $\theta\in[0,2\pi)$ be the signed angle rotating from the bearing $b_{ij}(p)$ to $b_{ij}(q)$. Thus, 
 \begin{equation}
 q_{i}-q_{j}=c\mathcal{R}(\theta)(p_{i}-p_{j}).
 \label{l6}
 \end{equation} Let $\xi=q_{i}-c\mathcal{R}(\theta)p_{i}$, it can be verified that \begin{equation}q_{k}=\xi+c\mathcal{R}(\theta)p_{k},\ k\in\{i,j\}.
 \label{l7}
 \end{equation} Then, our goal is to show that
 \begin{equation}q_{k}=\xi+c\mathcal{R}(\theta)p_{k},\ k\in\mathcal{V}\setminus\{ i,j\}.
 \label{l8}
 \end{equation}

 \textit{Case 1:} $k\in\mathcal{V}_{A}$. Consider the triangle $\Delta ijk$ (may be degenerate), two signed angles $\alpha_{ijk},\alpha_{kij}$ are known, then the third signed angle can also be determined by $\alpha_{jik}=\pi+\alpha_{ijk}+\alpha_{kij}\mod2\pi$. If the triangle $\Delta ijk$ is non-degenerate, then RoDs can be uniquely determined by the standard sine law. If the triangle $\Delta ijk$ is degenerate, without loss of generality, assume $\alpha_{ijk}=\alpha_{kij}=0, \alpha_{jik}=\pi$. Utilizing RoD $\kappa_{jik}=||p_{k}-p_{j}||/||p_{i}-p_{j}||$, we can obtain that $\kappa_{ijk}=1+\kappa_{jik},\kappa_{kij}=\kappa_{jik}/(1+\kappa_{jik}).$  Therefore, all the RoDs and SAs in the triangle  $\Delta ijk$ are determined. Thus, based on $(\ref{l6})$, we have
 \begin{align}
 q_{i}-q_{k}&=\frac{||q_{i}-q_{k}||}{||q_{i}-q_{j}||}\mathcal{R}(\alpha_{ijk})(q_{i}-q_{j})
 =c\frac{||q_{i}-q_{k}||}{||q_{i}-q_{j}||}\mathcal{R}(\alpha_{ijk})\mathcal{R}(\theta)(p_{i}-p_{j})
 \notag
 \\&=c\mathcal{R}(\theta)\frac{||p_{i}-p_{k}||}{||p_{i}-p_{j}||}\mathcal{R}(\alpha_{ijk})(p_{i}-p_{j})
 =c\mathcal{R}(\theta)(p_{i}-p_{k}).
 \end{align}
 Using $(\ref{l7})$, one can directly derive $(\ref{l8})$.

 \textit{Case 2:} $k\in\mathcal{V}_{D}$. There are two vertices among $i, j$, and $k$ lying in $\mathcal{V}_{D}$. Similarly, all the RoDs and SA $\alpha_{ijk}$ are determined. One can show that all the SAs can be determined no matter whether the triangle is degenerate or not. Following a similar routine in Case 1, one can finish the proof.
 \end{proof}

\section{Proof of Lemma $\ref{lem2nn}$} 
\label{diffe}
\begin{proof}
According to the definition of $ S_{p}$, we construct a mapping $F:[0,2\pi)\times\mathbb{R}^{2}\times \mathbb{R}^{+}\to S_{p}\subseteq\mathbb{R}^{2n}$ as follows:
$$F(\theta,\xi,c)\triangleq\mathbf{1}_{n}\otimes\xi+c(I_{n}\otimes\mathcal{R}(\theta))p.$$

\textit{Step 1:} We will show that $F$ is a bijective mapping. The definition of $S_{p}$ indicates that $F$ is surjective.  To show its injectivity, suppose $F(\theta,\xi,c)=F(\theta^{\prime},\xi^{\prime},c^{\prime})$, that is, $$\mathbf{1}_{n}\otimes\xi+c(I_{n}\otimes\mathcal{R}(\theta))p=\mathbf{1}_{n}\otimes\xi^{\prime}+c^{\prime}(I_{n}\otimes\mathcal{R}(\theta^{\prime}))p.$$
In component form, this is equivalent to
\begin{equation}
\xi+c\mathcal{R}(\theta)p_{i}=\xi^{\prime}+c^{\prime}\mathcal{R}(\theta^{\prime})p_{i},\ \forall i\in \{1,2,\ldots,n\}.
\label{l0}
\end{equation}
Then we obtain 
\begin{equation}
c\mathcal{R}(\theta)(p_{i}-p_{j})=c^{\prime}\mathcal{R}(\theta^{\prime})(p_{i}-p_{j}),\ \forall i,j\in \{1,2,\ldots,n\}, i\neq j.
\label{l1}
\end{equation}
Since $p_{i}\neq p_{j}$ when $i\neq j$, analyzing the norms on both hand sides of $(\ref{l1})$ provides that $c=c^{\prime}>0$ and $\theta=\theta^{\prime}\in [0,2\pi)$. Substituting these into $(\ref{l0})$ yields that $\xi=\xi^{\prime}$.

\textit{Step 2:} We will show that $S_{p}$ is a 4-dimensional smooth manifold.
First, we show that for each $q\in S_{p}$, there exists a neighborhood $U_{q}\subseteq S_{p}$ that is homeomorphic to an open set $U\subseteq\mathbb{R}^{4}$. If $q\in S_{p}$, there exist prameters $\theta_{0}\in[0,2\pi)$, $\xi_{0}\in\mathbb{R}^{2}$, and $c_{0}>0$ such that $$q=F(\theta_{0},\xi_{0},c_{0})=\mathbf{1}_{n}\otimes\xi_{0}+c_{0}(I_{n}\otimes\mathcal{R}(\theta_{0}))p.$$

 If $\theta_{0}\neq0$, let $\delta=\min\{|\theta_{0}|,|\theta_{0}-2\pi|\}$, $U_{q}=\{q^{\prime}=\mathbf{1}_{n}\otimes\xi+c(I_{n}\otimes\mathcal{R}(\theta))p:\theta\in(\theta_{0}-\delta,\theta_{0}+\delta),\xi\in\mathbb{R}^{2},c\in(0,+\infty)\}$. One can observe that $(\theta_{0}-\delta,\theta_{0}+\delta)\subseteq [0,2\pi)$ and $U_{q}\subseteq S_{p}$ is a neighborhood near $q$. Based on the conclusion in Step 1, $F:(\theta_{0}-\delta,\theta_{0}+\delta)\times\mathbb{R}^{2}\times \mathbb{R}^{+}\to U_{q}$ is bijective. One can also verify its inverse mapping is also continuous. So $U_{q}$ is homeomorphic to the open set $(\theta_{0}-\delta,\theta_{0}+\delta)\times\mathbb{R}^{2}\times \mathbb{R}^{+}\subseteq\mathbb{R}^{4}$. 

If $\theta_{0}=0$,  we take a small constant $\delta\in (0,\pi)$ and $U_{q}=\{q^{\prime}=\mathbf{1}_{n}\otimes\xi+c(I_{n}\otimes\mathcal{R}(\theta))p:\theta\in(-\delta,\delta),\xi\in\mathbb{R}^{2},c\in(0,+\infty)\}$. Although $(-\delta,\delta)$ is not included in $[0,2\pi)$, the periodicity of the rotation matrix $\mathcal{R}(\theta)$ and arguments in Step 1 also yield that  $F:(-\delta,\delta)\times\mathbb{R}^{2}\times \mathbb{R}^{+}\to U_{q}$ is a homeomorphism.

Thus, $S_{p}$ is a 4-dimensional topological manifold. The transition map in different charts can be verified to be smooth. Therefore, $S_{p}$ is a 4-dimensional smooth manifold.

\textit{Step 3:} We will show that $F$ is a smooth immersion into $\mathbb{R}^{2n}$.
The smoothness of mapping $F$ is obvious. To verify the immersion property, we calculate the rank of its Jacobi map. The $i$-th block of the Jacobi matrix $DF\in\mathbb{R}^{2n\times 4}$ is
$$[I_{2},c\frac{d}{d\theta}\mathcal{R}(\theta)p_{i},\mathcal{R}(\theta)p_{i}]=[I_{2},c\mathcal{R}(\theta+\frac{\pi}{2})p_{i},\mathcal{R}(\theta)p_{i}],$$ where we have applied the identity $\frac{d}{d\theta}\mathcal{R}(\theta)=\mathcal{R}(\theta+\frac{\pi}{2})$. Denote $e_{1}=(1,0)^{\top},e_{2}=(0,1)^{\top}$. Then the four columns of $DF$ correspond to 
$v_{1}=\mathbf{1}_{n}\otimes e_{1},v_{2}=\mathbf{1}_{n}\otimes e_{2},v_{3}=c[\mathbf{1}_{n}\otimes\mathcal{R}(\theta+\pi/2)]p$, and $v_{4}=[\mathbf{1}_{n}\otimes\mathcal{R}(\theta)]p.$
Assume there exist some coefficients $c_{i}\in\mathbb{R}\ (i=1,2,3,4)$ such that $c_{1}v_{1}+c_{2}v_{2}+c_{3}v_{3}+c_{4}v_{4}=0$.  In component form, we have 
\begin{equation}
c_{1}e_{1}+c_{2}e_{2}+c_{3}c\mathcal{R}(\theta+\pi/2)p_{i}+c_{4}\mathcal{R}(\theta)p_{i}=0,\ \forall i\in \{1,2,\ldots,n\}.
\label{l4}
\end{equation}
Similarly, it holds that 
\begin{equation}
c_{3}c\mathcal{R}(\theta+\pi/2)(p_{i}-p_{j})+c_{4}\mathcal{R}(\theta)(p_{i}-p_{j})=0, \ \forall i,j\in \{1,2,\ldots,n\}, i\neq j.
\label{l5}
\end{equation}
Since $\mathcal{R}(\theta+\pi/2)(p_{i}-p_{j})$ is orthogonal to $\mathcal{R}(\theta)(p_{i}-p_{j})$,  (\ref{l5}) yields that $c_{3}=c_{4}=0$. From $(\ref{l4})$, it holds that $c_{1}=c_{2}=0$. Thus, $\rank(DF_{p})=4$. Therefore, $F$ is a smooth immersion.

Consider the well-known parameterization of the unit circle $\mathbb{S}^{1}$ by an exponential mapping $E: [0,2\pi)\to\mathbb{S}^{1}$ ($E(t)=\exp(\sqrt{-1}t)$). Define $\tilde{F}:\mathbb{S}^{1}\times\mathbb{R}^{2}\times \mathbb{R}^{+}\to S_{p}\subseteq\mathbb{R}^{2n},$
$$\tilde{F}(z,\xi,c)={F}(E^{-1}(z),\xi,c)=\mathbf{1}_{n}\otimes\xi+c(I_{n}\otimes\mathcal{R}(E^{-1}(z)))p.$$ From Step 1 and Step 2, $\tilde{F}$ is bijective and its Jacobi map is of rank 4. Based on the global constant rank theorem \cite[Theorem 4.14]{lyjbb9}, $\tilde{F}$ is a diffeomorphism. 
\end{proof}

\section{Proof of Proposition $\ref{pro1}$}
\label{p2}
\begin{proof}
 Given a quadrilateral framework $(\mathcal{G}(A,D),p)$, we have
 \begin{equation}\label{eqli0}
 d_{12}b_{12}+d_{23}b_{23}+d_{34}b_{34}+d_{41}b_{41}=0.
 \end{equation}
  Then, $(\ref{eqli0})$ can be rewritten as:
 \begin{equation}\label{eqli}
 d_{12}\left[\begin{array}{c}\cos\theta_{12}\\ \sin\theta_{12}\end{array}\right]+d_{23}\left[\begin{array}{c}\cos\theta_{23}\\ \sin\theta_{23}\end{array}\right]+d_{34}\left[\begin{array}{c}\cos\theta_{34}\\ \sin\theta_{34}\end{array}\right]+d_{41}\left[\begin{array}{c}\cos\theta_{41}\\ \sin\theta_{41}\end{array}\right]=0,
 \end{equation}

 \textit{Case 1:} $\mathcal{V}_{A}=\{1,2,3\}$, $\mathcal{V}_{D}=\{4\}$. All SAs and $\frac{d_{13}}{d_{14}}$ can be obtained. 

 Denote
 $$M_{2}=\left[
 \begin{array}{cc}
 \cos\theta_{12}& \cos\theta_{23}\\
 \sin\theta_{12}& \sin\theta_{23}
 \end{array}
 \right],\   z=-\left[
 \begin{array}{cc}
 \cos\theta_{34}& \cos\theta_{41}\\
 \sin\theta_{34}& \sin\theta_{41}
 \end{array}
 \right]
 \left[\begin{array}{c} \frac{d_{34}}{d_{14}}\\ 1\end{array}\right].$$
 We investigate solutions of the following equation: 
 \begin{equation}\label{3a}
     M_{2}x=z.
 \end{equation}
 Similarly, $(\frac{d_{12}}{d_{14}},\frac{d_{23}}{d_{14}})^{\top}$ is always a feasible solution to $(\ref{3a})$.
 Note that $\det(M_{2})=\sin(\theta_{23}-\theta_{12})$. So there exists a unique solution to $(\ref{3a})$ if and only if $\sin(\theta_{23}-\theta_{12})\neq 0$, i.e., vertices 1, 2, and 3 are non-collinear. From Theorem $\ref{shp}$, $(\mathcal{G}(A,D),p)$ is globally SA-RoD rigid if and only if vertices 1, 2, and 3 are non-collinear.

 \textit{Case 2:} $\mathcal{V}_{A}=\{1\}$, $\mathcal{V}_{D}=\{2,3,4\}$. The shape of triangle $\Delta 124$ can be uniquely determined by $\alpha_{124}$ and $\kappa_{213}\kappa_{324}\kappa_{413}^{-1}=\frac{d_{14}}{d_{12}}$. Furthermore, the shape of triangle $\Delta 234$ can also be uniquely determined by  RoDs in this triangle. To ensure the shape uniqueness of $\Diamond 1234$, the possible reflection of vertex 3 with respect to the line connecting vertices 2 and 4 should be avoided. This is equivalent to that vertices $2,3,4$ are collinear or the reflection of vertex 3 with respect to the line connecting vertices 2 and 4 leads to the same location as vertex 1, i.e., $d_{14}=d_{34}, d_{12}=d_{32}$.
 
 \textit{Case 3:} $\mathcal{V}_{A}=\{1,2\}$, $\mathcal{V}_{D}=\{3,4\}$. 
 Denote 
 $$M_{1}=\left[
 \begin{array}{ccc}
 \cos\theta_{12}& \frac{d_{34}}{d_{14}}&0\\
 \sin\theta_{12}& 0 & \frac{d_{34}}{d_{14}}
 \end{array}
 \right],\ y=-\left[
 \begin{array}{cc}
 \cos\theta_{23}& \cos\theta_{41}\\
 \sin\theta_{23}& \sin\theta_{41}
 \end{array}
 \right]
 \left[\begin{array}{c} \frac{d_{23}}{d_{14}}\\ 1\end{array}\right].$$
 We investigate solutions of the following equation: 
 \begin{equation}\label{qua2}
     M_{1}x=y
 \end{equation}
 Since $p$ always corresponds to a feasible solution, $(\frac{d_{12}}{d_{14}},\cos\theta_{34},\sin\theta_{34})^{\top}$ is a solution of $(\ref{qua2})$. Furthermore,
 we have $\nulll(M)=\spann\{(1,-\frac{d_{14}}{d_{34}}\cos\theta_{12},-\frac{d_{14}}{d_{34}}\sin\theta_{12})^{\top}\}$. Thus, the solution set of 
 $(\ref{qua2})$ is $$\{(\frac{d_{12}}{d_{14}},\cos\theta_{34},\sin\theta_{34})^{\top}+t(1,-\frac{d_{14}}{d_{34}}\cos\theta_{12},-\frac{d_{14}}{d_{34}}\sin\theta_{12})^{\top}: \frac{d_{12}}{d_{14}}+t>0,t\in\mathbb{R}\}.$$
 Considering the unit norm constraints of bearing vectors, we have 
 $$||(\cos\theta_{34}-t\frac{d_{14}}{d_{34}}\cos\theta_{12},\sin\theta_{34}-t\frac{d_{14}}{d_{34}}\sin\theta_{12})^\top||=1.$$
 Solving this equation gives $t=0$ or $t=2\frac{d_{34}}{d_{14}}\cos(\theta_{34}-\theta_{12})$. If $\frac{d_{12}}{d_{14}}+2\frac{d_{34}}{d_{14}}\cos(\theta_{34}-\theta_{12})>0$, there exist two feasible RoDs subtended at vertex 1. From Theorem $\ref{shp}$, ($\mathcal{G}(A,D),p)$ is not globally SA-RoD rigid. If $\frac{d_{12}}{d_{14}}+2\frac{d_{34}}{d_{14}}\cos(\theta_{34}-\theta_{12})\leq0$, there exists a unique feasible RoD subtended at vertex 1. By utilizing geometric constraints in triangles $\Delta124$ and $\Delta234$, all RoDs and SAs can be uniquely determined. Therefore, it is globally SA-RoD rigid if and only if 
 \begin{equation}\label{qua3}
 d_{12}+2d_{34}\cos(\theta_{34}-\theta_{12})\leq 0.
 \end{equation}

 \textit{Case 4:} $\mathcal{V}_{A}=\{1,3\}$, $\mathcal{V}_{D}=\{2,4\}$. By invoking the cosine law in triangles $\Delta 234$ and $\Delta 124$, one has
 $$d_{23}^{2}+d_{34}^{2}-2d_{23}d_{34}\cos\theta_{324}=d_{12}^{2}+d_{14}^{2}-2d_{12}d_{14}\cos\theta_{124}.$$
 Combining with RoD constraints, we can conclude that the variable $x=\frac{d_{14}}{d_{12}}>0$ satisfies 
\begin{equation}
(1-\kappa_{413}^{2})x^{2}-2(\cos\theta_{124}-\kappa_{413}\kappa_{213}\cos\theta_{324})x+1-\kappa_{213}^{2}=0.
\label{eq:qu3}
\end{equation}
 Then a direct geometric observation provides that the shape of a quadrilateral is uniquely determined if and only if there exists a unique feasible solution to (\ref{eq:qu3}). Note that this equation always has a feasible positive solution. If $\kappa_{413}=1$, i.e., $d_{14}=d_{34}$, this equation degenerates into a linear equation. Consequently, the uniqueness is guaranteed. Otherwise, $d_{14}\neq d_{34}$. Under this setting, to ensure a unique feasible solution, there exist two cases:\\ (a) There exist two equal positive solutions to (\ref{eq:qu3}). Consequently, the discriminant of this quadratic equation is 0, i.e.,
 $$\kappa_{413}^{2}+\kappa_{213}^{2}-\sin^{2}\theta_{124}-\kappa_{413}^{2}\kappa_{213}^{2}\sin^{2}\theta_{324}-2\kappa_{413}\kappa_{213}\cos\theta_{324}\cos\theta_{124}=0.$$
 Thus, we have 
 \begin{equation}\label{qua4_o}
 d_{12}^{2}d_{34}^{2}+d_{14}^{2}d_{23}^{2}-d_{14}^{2}d_{21}^{2}\sin^{2}\theta_{124}-d_{34}^{2}d_{23}^{2}\sin^{2}\theta_{324}-2d_{12}d_{23}d_{34}d_{14}\cos\theta_{324}\cos\theta_{124}=0.
 \end{equation}
 (b) There exist two distinct solutions to (\ref{eq:qu3}). Furthermore, one root is positive while the other is non-positive. We have 
 $$x_{1}x_{2}=\frac{1-\kappa_{213}^{2}}{1-\kappa_{413}^{2}}\leq 0.$$
 Equivalently, it can be derived that
 \begin{equation}\label{qua5_o}
 (d_{23}-d_{12})(d_{34}-d_{14})\leq 0.
 \end{equation}
 Note that the possible case of $d_{34}=d_{14}$ is also included in this inequality.

 \end{proof}

\section{Proof of Theorem $\ref{thm0}$}
\label{pdual}
\begin{proof}
We only need to show that $$\dim(\nulll (R_{\mathcal{G}(A,D)}(p)))=\dim(\nulll(R_{\mathcal{G}^{\prime}(A,D)}(p))).$$ 
Based on the expression $(\ref{IRM})$, suppose $x=[x_{1}^{\top},\cdots,x_{n}^{\top}]^{\top}\in \nulll (R_{\mathcal{G}(A,D)}(p))$. Then, $$D_{ijk}^{\top}x_{i}-D_{ij}^{\top}x_{j}+D_{ik}^{\top}x_{k}=0,\ (i,j,k)\in\mathcal{T}_{D},$$ 
$$A_{rst}^{\top}x_{r}-A_{rs}^{\top}x_{s}+A_{rt}^{\top}x_{t}=0,\ (r,s,t)\in\mathcal{T}_{A}.$$
Using the orthogonal relations between $D_{\cdot}$ and $A_{\cdot}$, we have
 $$A_{ijk}^{\top}\mathcal{R}(\pi/2)x_{i}-A_{ij}^{\top}\mathcal{R}(\pi/2)x_{j}+A_{ik}^{\top}\mathcal{R}(\pi/2)x_{k}=0,\ (i,j,k)\in\mathcal{T}_{D}=\mathcal{T}^{\prime}_{A},$$ 
$$D_{rst}^{\top}\mathcal{R}(\pi/2)x_{r}-D_{rs}^{\top}\mathcal{R}(\pi/2)x_{s}+D_{rt}^{\top}\mathcal{R}(\pi/2)x_{t}=0,\ (r,s,t)\in\mathcal{T}_{A}=\mathcal{T}^{\prime}_{D}.$$ 
That is $(I_{n}\otimes\mathcal{R}(\pi/2))x\in \nulll(R_{\mathcal{G}^{\prime}(A,D)}(p))$. Therefore,  we can construct a linear mapping 
$F: \nulll(R_{\mathcal{G}(A,D)}(p))\to \nulll(R_{\mathcal{G}^{\prime}(A,D)}(p))$ as follows:
$F(x)=(I_{n}\otimes\mathcal{R}(\pi/2))x.$
Furthermore, it is easy to verify that $F$ is injective. Thus, $$\dim(\nulll (R_{\mathcal{G}(A,D)}(p)))\leq\dim(\nulll( R_{\mathcal{G}^{\prime}(A,D)}(p))).$$ Similarly, $\dim(\nulll( R_{\mathcal{G}^{\prime}(A,D)}(p)))\leq\dim(\nulll (R_{\mathcal{G}(A,D)}(p)))$. Thus, we have finished the proof. 
\end{proof}

\section{Proof of Theorem $\ref{gen1}$}
\label{pg1}
Before presenting the proof, we recall that a point $q$ is said to be a {\it regular point} of a smooth mapping $h: \mathbb{R}^{k}\to\mathbb{R}^{l}$ if $\rank(\frac{\partial h}{\partial p}(q))=\max_{p\in\mathbb{R}^{k}}\rank(\frac{\partial h}{\partial p}(p))$. Then, we provide the following Lemma.
\begin{lemma}\label{lem:reg}
Given a graph $\mathcal{G}(A,D)$ with an arbitrary bipartition, if $p\in\mathbb{R}^{2n}$ is generic, then $p$ is a regular point of $f_{\mathcal{G}(A,D)}$. 
\end{lemma}
\begin{proof}
    Assume $k=\max_{p\in\mathbb{R}^{2n}}\rank(\frac{\partial f_{\mathcal{G}(A,D)}}{\partial p}(p))=\max_{p\in\mathbb{R}^{2n}}\rank(R_{\mathcal{G}(A,D)}(p))$, then there exists a k-by-k submatrix of $R_{\mathcal{G}(A,D)}(p)$ whose rank is $k$. Furthermore, the determinant  $h(p)$ of this submatrix satisfies $h(p)\neq 0$. Suppose that $q$ is a generic configuration, but not a regular point of $f_{\mathcal{G}(A,D)}$, i.e., $\rank(R_{\mathcal{G}(A,D)}(q))<k$. Thus, we have $h(q)=0$. We denote the set of edges involved in the submatrix as $\mathcal{E}_{1}$. 
    
Note that the elements in $R_{\mathcal{G}(A,D)}$ (see $(\ref{IRM})$) are in the form of
$\frac{b_{ij}}{||e_{ij}||}=\frac{p_{j}-p_{i}}{||p_{j}-p_{i}||^{2}}$ (a rational function of the coordinates),
 and $\kappa_{ijk}\frac{b_{ij}}{||e_{ij}||}=||p_{k}-p_{i}||\frac{p_{j}-p_{i}}{||p_{j}-p_{i}||^{3}}=||p_{k}-p_{i}||||p_{j}-p_{i}||\frac{p_{j}-p_{i}}{||p_{j}-p_{i}||^{4}}$(a product of a rational function and square root functions).
Following the routine in \cite[section IV]{lyjb002}, the space $\mathbb{L}(\mathcal{E}_{1})$ is defined as
\begin{equation}\label{ll3}
\mathbb{L}(\mathcal{E}_{1})=\left\{\sum\limits_{F\in 2^{\mathcal{E}_{1}}}P_{F}(x_{1},x_{2},\ldots,x_{n})\prod_{(i,j)\in F}d_{x_{i},x_{j}}: P_{F}\in \mathbb{Q}(x_{1},x_{2},\ldots,x_{n})\right\},
\end{equation}
where $2^{\mathcal{E}_{1}}$ is the power set of $\mathcal{E}_{1}$, $\mathbb{Q}(x_{1},x_{2},\ldots,x_{n})$ is the field of rational functions with rational coefficients, $d_{x_{i},x_{j}}=||x_{i}-x_{j}||$. Let $|\mathcal{E}_{1}|=m_{1}$. According to \cite[Lemma 3]{lyjb002}, $\mathbb{L}(\mathcal{E}_{1})/\mathbb{Q}(x_{1},x_{2},\ldots,x_{n})$ is a field extension of degree $2^{m_{1}}$. Furthermore, $h\in\mathbb{L}(\mathcal{E}_{1})$. Since $q\in\mathbb{R}^{2n}$ is generic satisfying $h(q)=0$, \cite[Lemma 4]{lyjb002} implies that $h(\cdot)=0$ for any configuration. Thus,
we have obtained a contradiction to the fact that $h(p)\neq 0$. This completes the proof.
\end{proof}

Then, we show that for generic frameworks, SA-RoD rigidity and infinitesimal SA-RoD rigidity are equivalent.
\begin{lemma}\label{lem:eqge}
Given a graph $\mathcal{G}(A,D)$ with an arbitrary bipartition and a generic configuration $p\in\mathbb{R}^{2n}$, then framework $(\mathcal{G}(A,D),p)$ is SA-RoD rigid if and only if it is infinitesimally SA-RoD rigid. 
\end{lemma}
\begin{proof}
Necessity: Since $(\mathcal{G}(A,D),p)$ is SA-RoD rigid with a generic configuration $p$. Based on Lemma $\ref{lem:reg}$, $p$ is a regular point of $f_{\mathcal{G}(A,D)}$. Let $\rank(R_{\mathcal{G}(A,D)}(p))=k$. According to \cite[Proposition 2]{lyjb8}
, there exists a neighborhood $U$ of $p$ such that $f_{\mathcal{G}(A,D)}^{-1} (f_{\mathcal{G}(A,D)}(p))\cap U$
is a $2n-k$ dimensional manifold. Based on Lemma $\ref{lem1o}$ and Lemma $
\ref{lem2nn}$, SA-RoD rigidity of $(\mathcal{G}(A,D),p)$ implies that $2n-k=4$. Thus, $k=2n-4$, i.e., $(\mathcal{G}(A,D),p)$ is infinitesimally SA-RoD rigid. 

Sufficiency: It is a consequence of Theorem $\ref{the1}$.
\end{proof}
\begin{proof} [Proof of Theorem $\ref{gen1}$] Based on Lemma $\ref{lem:eqge}$, we only need to show the result for infinitesimal SA-RoD rigidity. Suppose that $p^{\ast}\in\mathbb{R}^{2n}$ is generic, $(\mathcal{G}(A,D),p^{\ast})$ is infinitesimally SA-RoD rigid. So $\rank(R_{\mathcal{G}(A,D)}(p^{\ast}))=\max_{p\in\mathbb{R}^{2n}}\rank(R_{\mathcal{G}(A,D)}(p))=2n-4$. Let $q\in\mathbb{R}^{2n}$ be a generic point distinct to $p^{\ast}$. Based on Lemma $\ref{lem:reg}$, we have $\rank(R_{\mathcal{G}(A,D)}(q))=\max_{p\in\mathbb{R}^{2n}}\rank(R_{\mathcal{G}(A,D)}(p))=2n-4$. Therefore, $(\mathcal{G}(A,D),q)$ is infinitesimally SA-RoD rigid.  
\end{proof}

\section{Proof of Lemma $\ref{1ver}$}
\label{padd}
\begin{proof}
 For a Type $A_{1}$ 1-vertex addition, without loss of generality, assume $i,n+1\in\mathcal{V}^{\prime}_{A}$, and $j\in\mathcal{V}^{\prime}_{D}$.
From the global SA-RoD rigidity of the framework $(\mathcal{G}(A,D),p)$, there exist at least two neighbors for any vertex in graph $\mathcal{G}$. Assume $(m,i)\in\mathcal{E}$, $(l,j)\in\mathcal{E}$, $m\neq j$, and $l\neq i$. Based on Theorem $\ref{shp}$, SA $\alpha_{imj}$ and RoD $\kappa_{jli}$ can be uniquely determined by SA and RoD constraints no matter whether $(i,j)\in\mathcal{E}$. The algebraic properties of SAs and RoDs indicates that $$\kappa_{ji,n+1}= \kappa_{jl,n+1}/\kappa_{jli},\ \alpha_{ij,n+1}=\alpha_{im,n+1}-\alpha_{imj}.$$
In addition, since $n+1\in\mathcal{V}_A$, the triangle $\Delta ij,n+1$ (may be degenerate) is constrained by $\kappa_{ji,n+1}$, $\alpha_{ij,n+1}$, and $\alpha_{n+1,ij}$. Thus, the shape of the triangle $\Delta ij,n+1$ is uniquely determined and no reflection of vertex $n+1$ can occur. So $(\mathcal{G}^{\prime}(A,D),p^{\prime})$ is globally SA-RoD rigid.

Then, we briefly sketch the proof for other types. For a Type $A_{2}$ 1-vertex addition, we have $i$, $j$, $n+1\in\mathcal{V}^{\prime}_{A}$. If these three vertices are non-collinear, the shape uniqueness of $\Delta ij,n+1$ under SA constraints can be guaranteed. A similar argument can provide global SA-RoD rigidity of the induced framework. For a Type $D_{1}$ 1-vertex addition, the proof can be conducted similarly. For a Type $D_{2}$ 1-vertex addition, we have $i$, $j$, $k$, $n+1\in\mathcal{V}^{\prime}_{D}$. Similarly, we know $\kappa_{ik,n+1}$, $\kappa_{ij,n+1}$, and $\kappa_{kj,n+1}$ can be uniquely determined by SA and RoD constraints. Combining with RoDs $\kappa_{n+1,i,k}$ and $\kappa_{n+1,i,j}$, the shapes of triangle $\Delta{ik,n+1}$ and triangle $\Delta{jk,n+1}$ can be uniquely determined. Since vertices $i$, $j$, and $k$ are non-collinear, the flipping ambiguity of vertice $n+1$ is circumvented. So the induced framework is globally SA-RoD rigid as well.
\end{proof}

\section{Proof of Theorem $\ref{mert}$}
\label{pmer}
\begin{proof}
Without loss of generality, assume $i\in\mathcal{V}_{D},m,j,k\in\mathcal{V}_{A}$. For simplicity, we use $f_{\mathcal{G}}$ to represent the rigidity function $f_{\mathcal{G}(A,D)}$. To show $f_{\mathcal{G}}^{-1}(f_{\mathcal{G}}(r))=S_{r}$, it is sufficient to prove $f_{\mathcal{G}}^{-1}(f_{\mathcal{G}}(r))\subseteq S_{r}$. Let $r^{\prime}=[p^{\prime \top},q^{\prime \top}]^{\top}\in f_{\mathcal{G}}^{-1}(f_{\mathcal{G}}(r))$, then we have $f_{\mathcal{G}_{1}}(p^{\prime})=f_{\mathcal{G}_{1}}(p)$ and $f_{\mathcal{G}_{2}}(q^{\prime})=f_{\mathcal{G}_{2}}(q)$. By utilizing the global SA-RoD rigidity of $(\mathcal{G}_{1}(A_{1},D_{1}),p)$ and $(\mathcal{G}_{2}(A_{2},D_{2}),q)$, there exist $c_{i}>0,\theta_{i}\in [0,2\pi),\xi_{i}\in\mathbb{R}^{2}$ for $i=1,2$ such that 
\begin{equation}\label{qu0}
p^{\prime}=c_{1}I_{n_{1}}\otimes\mathcal{R}(\theta_{1})p+\mathbf{1}_{n_{1}}\otimes\xi_{1}, q^{\prime}=c_{2}I_{n_{2}}\otimes\mathcal{R}(\theta_{1})q+\mathbf{1}_{n_{2}}\otimes\xi_{2},
\end{equation} where $n_{i}=|\mathcal{V}_{i}|,i=1,2$. In particular, it holds that \begin{equation}\label{qu1}
p_{s}^{\prime}=c_{1}\mathcal{R}(\theta_{1})p_{s}+\xi_{1}, q_{t}^{\prime}=c_{2}\mathcal{R}(\theta_{1})q_{t}+\xi_{2},
\end{equation}
where $s=m,i$ and $t=j,k$. As in the proof in Lemma $\ref{1ver}$, we know that in the quadrilateral $\Diamond{mijk}$, the SAs subtended at vertices $m,k,j$ and the RoD subtended at vertice $i$ can be completely determined from the rigidity function $f_{\mathcal{G}}(r)$. Since $m,j,k$ are non-collinear, the shape of $\Diamond{mijk}$ can be uniquely determined. It indicates that $c_{1}=c_{2}$, $\theta_{1}=\theta_{2}$, and $\xi_{1}=\xi_{2}$ in $(\ref{qu1})$. Thus, $r^{\prime}=c_{1}I_{n}\otimes\mathcal{R}(\theta_{1})r+\mathbf{1}_{n}\otimes\xi_{1}$, i.e., $r\in S_{r}$. This completes the proof.
\end{proof}

\section{Proof of Theorem $\ref{equv}$}
\label{peq}
\begin{proof}
\textit{Necessity:} Suppose a configuration $y=(y^{\top}_{1},\ldots,y^{\top}_{n})^{\top}\in\mathbb{R}^{2n}$ satisfies $f_{\hat{\mathcal{G}}}(x)=f_{\hat{\mathcal{G}}}(y)$, where the subscript $A,D$ is omitted in the expression ${\hat{\mathcal{G}}}(A,D)$. Based on Assumption 3 and Lemma $\ref{lem1o}$, the subgraph of ${\hat{\mathcal{G}}}$ spanned by $\mathcal{V}_{a}$ is complete and the induced framework is globally SA-RoD rigid. Without loss of generality, assume $\mathcal{V}_a=\{1,2,\ldots,n_a\}$, and denote $x_{a}=(x^{\top}_{1},\ldots,x^{\top}_{n_{a}})^{\top}$, $y_{a}=(y^{\top}_{1},\ldots,y^{\top}_{n_{a}})^{\top}$. Then $y_{a}\in S_{x_{a}}$. So there exist $c>0, \theta\in[0,2\pi)$, and $\xi\in\mathbb{R}^{2}$ such that $y_{i}=c\mathcal{R}(\theta)x_{i}+\xi$ for $i=1,2,\ldots,n_{a}$. It implies that $cI_{n}\otimes\mathcal{R}(\theta)x+\mathbf{1}_{n}\otimes\xi$ is a solution to problem $(\ref{snl11})$. Based on the localizability of $(\hat{\mathcal{G}}(A,D),x,\mathcal{V}_{a})$, we have $y=cI_{n}\otimes\mathcal{R}(\theta)x+\mathbf{1}_{n}\otimes\xi$.  So $(\hat{\mathcal{G}}(A,D),x)$ is globally SA-RoD rigid.

\textit{Sufficiency:} Suppose $y=(y^{\top}_{1},\ldots,y^{\top}_{n})^{\top}\in\mathbb{R}^{2n}$ is a solution to $(\ref{snl11})$. Then $y\in  f_{\hat{\mathcal{G}}}^{-1}(f_{\hat{\mathcal{G}}}(x))$. Due to the global SA-RoD rigidity of $(\hat{\mathcal{G}}(A,D),x)$, $f_{\hat{\mathcal{G}}}^{-1}(f_{\hat{\mathcal{G}}}(x))=S_{x}$. Therefore, there exist $c>0, \theta\in[0,2\pi)$, and $\xi\in\mathbb{R}^{2}$ such that $y_{i}=c\mathcal{R}(\theta)x_{i}+\xi$ for $i=1,2,\ldots,n$. In particular, it holds that $y_{i}-y_{j}=c\mathcal{R}(\theta)(x_{i}-x_{j}), i\in\mathcal{V}_{a}$. From Assumption 3, we have $n_{a}\geq 2$. Since $x_{i}=y_{i}$ for $i\in\mathcal{V}_{a}$, we have $c=\frac{||y_{i}-y_{j}||}{||x_{i}-x_{j}||}=1,\ \theta=0$, and $\xi=0$. Thus, $y=x$ and  $(\mathcal{\hat{G}}(A,D),x,\mathcal{V}_{a})$ is SA-RoD localizable.
\end{proof}

\section{Proof of Theorem $\ref{equv1}$}
\label{pequv}
To facilitate the recovery of locations from bearings and distances,  we define a path matrix with respect to a base vertex.
\begin{definition}\label{def:path}
Given an oriented graph $\mathcal{G}=(\mathcal{V},\mathcal{E})$ and any base vertex $l\in \mathcal{V}$, after fixing a spanning tree ${Tree}_{\mathcal{G}}$, a path matrix $P_{l}=[{P}_{l,ij}]\in\mathbb{R}^{n\times m}$  associated with base vertex $l$ is defined as
\begin{equation}
{{P}_{l,ij}}\triangleq \begin{cases}
   1 &  \text{the orientation of $e_{j}$ is consistent with that of the path $\pi_{l\to i}$},  \\
   -1 & \text{the orientation of $e_{j}$ is opposite to that of the path $\pi_{l\to i}$}, \\
   0 & \text{$e_{j}$ does not occur in the path $\pi_{l\to i}$},
\end{cases}
\end{equation}
where $e_j$ is the $j$-th edge, $\pi_{l\to i}$ encodes the path from vertex $l$ to vertex $i$ in ${Tree}_{\mathcal{G}}$.
\end{definition}

Note that the elements in the $l$-th row of $P_{l}$ are always zero, and $P_{l}$ depends on the choice of the spanning tree. Despite this dependence, we can establish the following lemma. 

\begin{lemma}Given a graph $\mathcal{G}$, for any two path matrices $P_{l},P^{\prime}_{l}$ with the same base vertex $l\in\mathcal{V}$, there exists a matrix $\Phi_{l}\in\mathbb{R}^{n\times (m-n+1)}$ such that
\begin{equation}\label {ec1}
P_{l}-P^{\prime}_{l}=\Phi_{l}C,
\end{equation}
where $C$ is the graph cycle basis matrix of $\mathcal{G}$.
\end{lemma}
\begin{proof}
Note that every nonzero row of $P_{l}-P^{\prime}_{l}$ represents a cycle. Thus, each row in matrix $P_{l}-P^{\prime}_{l}$ can be obtained as a linear combination of rows in $C$. That is, there exists a matrix $\Phi_{l}$ of suitable size such that $P_{l}-P^{\prime}_{l}=\Phi_{l}C$.
\end{proof}
\begin{lemma}\label{lem:reco}
    Given a framework $(\mathcal{G},x)$, suppose that $P_{l}$ is a path matrix with base vertex l, $B=(\ldots,b_{ij},\ldots)\in\mathbb{R}^{2\times m}$ is the bearing matrix, and $d=(\ldots,d_{ij},\ldots)^{\top}\in\mathbb{R}^{m}$ is the distance vector of all edges, then it holds that $$x=\mathbf{1}_{n}\otimes x_{l}+P_{l}\odot B d.$$
\end{lemma}
\begin{proof}
    Denote $y=\mathbf{1}_{n}\otimes x_{l}+P_{l}\odot B d$. First, we will show that $y$ is well defined, i.e., different choices of path matrices $P_{l}, P^{\prime}_{l}$ will not affect the definition of $y$. Let $y^{\prime}=\mathbf{1}_{n}\otimes x_{l}+P^{\prime}_{l}\odot B d$.
From equation (\ref{ec1}), there exists a matrix $\Phi_{l}\in\mathbb{R}^{n\times (m-n+1)}$ such that 
\begin{equation}\label{com1}
y-y^{\prime}=(P_{l}-P^{\prime}_{l})\odot Bd=[\Phi_{l}C]\odot Bd=\Phi_{l}C_{b}d=0.
\end{equation}
Then, each component of $y$ can be calculated to conclude that $y=x$.
\end{proof}
\begin{proof}[Proof of Theorem $\ref{equv1}$]
Denote the solution sets of $(\ref{snl11})$ and $(\ref{snl12})$ as $\mathcal{S}_{1}$ and $\mathcal{S}_{2}$, respectively. We construct a mapping $T:\ \mathcal{S}_{1}\to\mathcal{S}_{2}$ as $$T(x)=(\ldots, ||x_{j}-x_{i}||, \frac{x^{\top}_{j}-x^{\top}_{i}}{||x_{j}-x_{i}||},\ldots)^{\top},\ (i,j)\in\mathcal{\hat{E}}.$$
It is direct to show that $T$ is well defined. Then we will verify its bijectivity.

\textit{Step 1:} We will prove that $T$ is injective. Suppose $T(x)=T(y)$ for $x, y\in\mathcal{S}_{1}$. By the definition, $||x_{j}-x_{i}||=||y_{j}-y_{i}||,\  \frac{x_{j}-x_{i}}{||x_{j}-x_{i}||}=\frac{y_{j}-y_{i}}{||y_{j}-y_{i}||}$ for $(i,j)\in\hat{\mathcal{E}}$. So $x_{j}-x_{i}=y_{j}-y_{i}$ for $(i,j)\in\hat{\mathcal{E}}$. 
The connectivity of $\mathcal{G}$ implies that $x_{j}-x_{1}=y_{j}-y_{1}$ for any $j\in\mathcal{V}$. Then it follows that $x=y$.

\textit{Step 2:} To show $T$ is surjective, suppose $(\ldots,d_{ij},\  b^{\top}_{ij},\ldots)^{\top}\in \mathcal{S}_{2}$, we need to find $x\in\mathcal{S}_{1}$ such that $T(x)=(\ldots,d_{ij},\  b^{\top}_{ij},\ldots)^{\top}$. First, assume node 1 is an anchor, let $x_{1}=p_{1}$. Based on Lemma $\ref{lem:reco}$, $x=\mathbf{1}_{n}\otimes x_{1}+P_{1}\odot B d$.
If $(i,j)\in\hat{\mathcal{E}}$, without loss of generality, assume that the shortest path $t_{1},t_{2},\ldots,t_{i}$ connecting vertex $i$ with vertex $1$ does not pass vertex $j$, otherwise we can swap $i$ and $j$. The path consisting of $t_{1},t_{2},\ldots,t_{i}$ and edge $(i,j)$ ends with node $j$. By previous arguments, we can modify path matrix $P_{1}$ by changing the $i$-th and $j$-th rows using these constructed paths without affecting $x$. Then the $i$-th and $j$-th rows only differ in the edge $(i,j)\in\hat{\mathcal{E}}$. It follows that
$x_{j}-x_{i}=d_{ij}b_{ij}.$
Since $d_{ij}>0$ and $||b_{ij}||=1$ hold for $(i,j)\in\mathcal{\hat{E}}$, we obtain that
$||x_{i}-x_{j}||=d_{ij},\frac{x_{i}-x_{j}}{||x_{i}-x_{j}||}=b_{ij}.$ On the other hand, the first two equations in constraints of $(\ref{snl12})$ can be transformed into those of $(\ref{snl11})$. This completes the proof.

\end{proof}
\section{Proof of Theorem $\ref{pr1}$}
\label{app:ful}
\begin{proof}
Note that 2) can be derived as a consequence of 1). Hence, we will prove the statement in 1) by induction on $n$. In particular, $C_{B}$ can be chosen as a square matrix. If $n=3$, the framework is a triangle with anchors $1\in\mathcal{V}_{D}$, $2 \in\mathcal{V}_{A}$ and the free node $3\in\mathcal{V}_{D}$. One can verify that $\rank(C_{B})=6$. Then we suppose that for a framework $(\mathcal{\hat{G}}(A,D),x)$ constructed by a Type $(D_{1},A_{1})$ bilateration ordering, $\rank(C_{B,n})=4n-6$, where the subscript $n$ indicates the number of nodes. The goal is to show after adding the $(n+1)$-th node, $\rank(C_{B,n+1})=\rank(C_{B,n})+4=4n-2$. \\

\textit{Case 1:} If $n+1$ is even, node $n+1\in\mathcal{V}_{A}$. Assume two edges $(n+1,i)$ and $(n+1,j)\in\mathcal{E}$ are added satisfying $i,j\in\mathcal{V}_{D},i<j<n+1$. One can find a path connecting nodes $i$ and $j$ in the original graph. Note that the path $i\to n+1\to j$ also connects nodes $i$ and $j$. These two paths are distinct due to the existence of node $n+1$. Thus, a new cycle can be obtained. This corresponds to adding a submatrix $[\ast,\pm d_{2n-2}I_{2},\pm d_{2n-1}I_{2}]\in\mathbb{R}^{2\times(4n-2)}$. On the other hand, the SA $\theta_{n+1,i,j}$ between edges $(n+1,i)$ and $(n+1,j)$ corresponds to a submatrix $[\textbf{0},\mathcal{R}(\theta_{n+1,i,j}),-I_{2}]\in\mathbb{R}^{2\times(4n-2)}$. So $C_{B,n+1}\in\mathbb{R}^{(4n-2)\times(4n-2)}$ has the following block structure:
\begin{equation}
C_{B,n+1}
=\left[
\begin{array}{cc}
   C_{B,n}&\mathbf{0}  \\
   \ast&\Delta_{n+1}
\end{array}
\right]
:=
\left[
\begin{array}{ccc}
     C_{B,n} &\textbf{0}&\textbf{0}\\
   \ast&\pm d_{2n-2}I_{2}&\pm d_{2n-1}I_{2}  \\
    \textbf{0}&\mathcal{R}(\theta_{n+1,i,j})&-I_{2}
\end{array}
\right].
\nonumber
\end{equation}
Since $\rank(C_{B,n})=4n-6$, $C_{B,n}\in\mathbb{R}^{(4n-6)\times(4n-6)}$ is nonsingular. By Gaussian elimination method, it holds that
\begin{equation}
\rank (C_{B,n+1})
=\rank( C_{B,n})+\rank(\Delta_{n+1}).
\nonumber
\end{equation}
If $\theta_{n+1,i,j}\neq 0\mod\pi$, we can verify that $\rank(\Delta_{n+1})=4$. Otherwise, nodes $n+1$,$i$,and $j$ are collinear. We consider the case that vertice $n+1$ lies on the segment connecting node $i$ and node $j$ (other cases can be handled similarly, here we omit the details). Then,  $\Delta_{n+1}=\left[
\begin{array}{cc}
   d_{2n-2}I_{2}&d_{2n-1}I_{2}  \\
   I_{2}&-I_{2}
\end{array}
\right]$, whose rank is also 4. So it holds that $\rank(C_{B,n+1})=\rank(C_{B,n})+4=4n-2$.

\textit{Case 2:} If $n+1$ is odd, we have $n+1\in\mathcal{V}_{D}$. Without loss of generality, assume that two edges $(n+1,i)$ and $(n+1,j)\in\mathcal{E}$ are added satisfying $i\in\mathcal{V}_{A},j\in\mathcal{V}_{D},i<j<n+1$. Similarly, the newly generated cycle corresponds to a sub-matrix $[\ast,\pm d_{2n-2}I_{2},\pm d_{2n-1}I_{2}]\in\mathbb{R}^{2\times(4n-2)}$. On the other hand, the SA $\theta_{i,i_{0},n+1}$ between edge $(i,n+1)$ and edge $(i,i_{0})$ corresponds to a submatrix $[\textbf{0},\mathcal{R}(\theta_{i,i_{0}n+1}),\textbf{0},-I_{2},\textbf{0}]\in\mathbb{R}^{2\times(4n-2)}$. For convenience, this submatrix is expressed as $[\ast,-I_{2},\textbf{0}]$, and $C_{B,n+1}\in\mathbb{R}^{(4n-2)\times(4n-2)}$ can be written as:
\begin{equation}
C_{B,n+1}
=
\left[
\begin{array}{ccc}
     C_{B,n} &\textbf{0}&\textbf{0}\\
   \ast&\pm d_{2n-2}I_{2}&\pm d_{2n-1}I_{2}  \\
    \ast&-I_{2}&\textbf{0}
\end{array}
\right].
\nonumber
\end{equation}
Similarly, we can verify that $\rank(C_{B,n+1})=4n-2$. This completes the proof.
\end{proof}

\bibliography{references}   

@Article{lyjb05,
  author = {J. Fang and M. Cao and A. S. Morse and B. D. O. Anderson},
  title = {Sequential localization of sensor networks},
  journal = {SIAM J. Control Optim.},
  volume = {48},
  number = {1},
  pages = {321--350},
  year = {2009},
}

@Article{lyjb06,
  author = {C. Wan and G. Jing and S. You and R. Dai},
  title = {Sensor network localization via alternating rank minimization algorithms},
  journal = {IEEE Trans. Control Netw. Syst.},
  volume={7},
  number={2},
  pages={1040-1051},
  year = {2019},
}

@Article{lyjb07,
  author = {S. Zhao and D. Zelazo},
  title = {Localizability and distributed protocols for bearing-based network localization in arbitrary dimensions},
  journal = {Automatica},
  volume = {69},
  pages = {334--341},
  year = {2016},
}

@Article{lyjb002,
  author = {C. Cros and P. O. Amblard and C. Prieur and J. F. D. Rocha},
  title = {Pseudorange rigidity and solvability of cooperative {GNSS} positioning},
  journal = {IEEE Trans. Control Netw. Syst.},
  volume = {12},
  number = {1},
  pages = {1176--1187},
  month = mar,
  year = {2025},
}

@Article{lyjb09,
  author = {G. Jing and C. Wan and R. Dai},
  title = {Angle-based sensor network localization},
  journal = {IEEE Trans. Autom. Control},
  volume = {67},
  number = {2},
  pages = {840--855},
  month = feb,
  year = {2022},
}

@Article{lyjb0012,
  author = {T. Eren},
  title = {Cooperative localization in wireless ad hoc and sensor networks using hybrid distance and bearing (angle of arrival) measurements},
  journal = {J. Wireless Com. Netw.},
  pages = {72},
  month = aug,
  year = {2011},
}

@Article{lyjb012,
  author = {Z. Lin and T. Han and R. Zheng and C. Yu},
  title = {Distributed localization with mixed measurements under switching topologies},
  journal = {Automatica},
  volume = {76},
  pages = {251--257},
  year = {2017},
}

@Article{lyjb013,
  author = {X. Fang and X. Li and L. Xie},
  title = {Angle-displacement rigidity theory with application to distributed network localization},
  journal = {IEEE Trans. Autom. Control},
  volume = {66},
  number = {6},
  pages = {2574--2587},
  month = jun,
  year = {2021},
}

@Article{lyjb014,
  author = {X. Fang and X. Li and L. Xie},
  title = {3-{D} distributed localization with mixed local relative measurements},
  journal = {IEEE Trans. Signal Process.},
  volume = {68},
  pages = {5869--5881},
  year = {2020},
}

@Book{lyjb1,
  author = {C. Godsil and G. F. Royle},
  title = {Algebraic Graph Theory},
  series = {Grad. Texts in Math.},
  volume = {207},
  publisher = {Springer},
  address = {New York},
  year = {2001},
}

@Book{lyjbb9.14,
  author = {R. Connelly and S. D. Guest},
  title = {Frameworks, Tensegrities, and Symmetry},
  publisher = {Cambridge University Press},
  year = {2022},
}

@Article{lyjbb8.1,
  author = {K. Cao and Z. Han and X. Li and L. Xie},
  title = {Ratio-of-distance rigidity theory with application to similar formation control},
  journal = {IEEE Trans. Autom. Control},
  volume = {65},
  number = {6},
  pages = {2598--2611},
  month = jun,
  year = {2020},
}

@Article{lyjb4,
  author = {S. Zhao and D. Zelazo},
  title = {Bearing rigidity and almost global bearing-only formation stabilization},
  journal = {IEEE Trans. Autom. Control},
  volume = {61},
  number = {5},
  pages = {1255--1268},
  year = {2016},
}

@Article{lyjb9,
  author = {L. Chen and M. Cao and C. Li},
  title = {Angle rigidity and its usage to stabilize multiagent formations in {2-D}},
  journal = {IEEE Trans. Autom. Control},
  volume = {66},
  number = {8},
  pages = {3667--3681},
  month = aug,
  year = {2021},
 
}

@misc{lyjbb10,
      author={J. Huang and G. Jing},
      title={Signed Angle Rigid Graphs for Network Localization and Formation Control},
      year = 2025,
      howpublished = {preprint, \url{https://arxiv.org/abs/2505.19945}},
}

@Article{lyjb8,
  author = {L. Asimow and B. Roth},
  title = {The rigidity of graphs},
  journal = {Trans. Amer. Math. Soc.},
  volume = {245},
  pages = {279--289},
  year = {1978},
}

@Article{lyjbb8,
  author = {G. Jing and G. Zhang and H. W. J. Lee and L. Wang},
  title = {Angle-based shape determination theory of planar graphs with application to formation stabilization},
  journal = {Automatica},
  volume = {105},
  pages = {117--129},
  year = {2019},
}

@Book{lyjbb9,
  author = {J. Lee},
  title = {Introduction to Smooth Manifolds},
  series = {Grad. Texts in Math.},
  volume = {218},
  publisher = {Springer},
  address = {New York},
  year = {2000},
}

@Article{lyjbb11,
  author = {F. Arrigoni and A. Fusiello},
  title = {Bearing-based network localizability: A unifying view},
  journal = {IEEE Trans. Pattern Anal. Mach. Intell.},
  volume = {41},
  number = {9},
  pages = {2049--2069},
  day = {1},
  month = sep,
  year = {2019},
}

@Article{lyjbmin,
  author = {M. H. Trinh and Q. Van Tran and H.--S. Ahn},
  title = {Minimal and redundant bearing rigidity: Conditions and applications},
  journal = {IEEE Trans. Autom. Control},
  volume = {65},
  number = {10},
  pages = {4186--4200},
  year = {2019},
}

@Book{lyjbb15,
  author = {A. Ruszczy\'nski},
  title = {Nonlinear Optimization},
  publisher = {Princeton University Press},
  year = {2011},
}

@INPROCEEDINGS{lyjbb17,
  author={S. H. Kwon and Z. Sun and B. D. O. Anderson and H. S. Ahn},
  booktitle={Proc. 59th IEEE Conf. Decis. Control(CDC)}, 
  title={Hybrid rigidity theory with signed constraints and its application to formation shape control in 2-D space}, 
  pages={518-523},
  year={2020}, 
  address={Jeju, South Korea},
 }

@Article{TE,
  author = {T. Eren},
  journal = {Turk. J. Electr. Eng. Comput. Sci.},
  title = {Using Angle of Arrival (Bearing) Information for Localization in Robot Networks},
  volume = {15},
  number={2},
  pages={169-186},
  month = jan,
  year = {2007},
}

@Article{CAOr,
  author = {K. Cao and D. Li and L. Xie},
  journal = {Automatica},
  title = {Bearing-ratio-of-distance rigidity theory with application to directly similar formation control},
  volume = {109},
  pages = {108540},
  month = aug,
  year = {2019},
}

@article{jing2018weak,
  author={G. Jing and G. Zhang and H. W. Lee and L. Wang},
  journal={SIAM J. Control Optim.},
  title={Weak rigidity theory and its application to formation stabilization},
  volume={56},
  number={3},
  pages={2248--2273},
  month={apr},
  year={2018},
}

@inproceedings{eren2005merging,
  title={Merging globally rigid formations of mobile autonomous agents}, 
  author={T. Eren and B. D. O. Anderson and W. Whiteley and A.S. Morse and P.N. Belhumeur},
 booktitle = {Proc. of the 3rd IEEE international joint conference on
autonomous agents and multiagent systems},
  pages={1260-1261},
  year={2004},
}

@article{laman1970graphs,
  title={On graphs and rigidity of plane skeletal structures},
  author={G. Laman},
  journal={J. Eng. Math.},
  volume={4},
  number={4},
  pages={331--340},
  year={1970},
}

@article{servatius1999constraining,
  title={Constraining plane configurations in computer-aided design: Combinatorics of directions and lengths},
  author={B. Servatius and W. Whiteley},
  journal={SIAM J. Discrete Math.},
  volume={12},
  number={1},
  pages={136--153},
  year={1999},
}

@article{nixon2014characterization,
  title={A characterization of generically rigid frameworks on surfaces of revolution},
  author={A. Nixon and J. C. Owen and S. C. Power},
  journal={SIAM J. Discrete Math.},
  volume={28},
  number={4},
  pages={2008--2028},
  year={2014},
}

@article{anderson2010formal,
  title={Formal theory of noisy sensor network localization},
  author={B. D. O. Anderson and I. Shames and G. Mao and B. Fidan},
  journal={SIAM J. Discrete Math.},
  volume={24},
  number={2},
  pages={684--698},
  year={2010},
}

@article{garamvolgyi2024partial,
  title={Partial reflections and globally linked pairs in rigid graphs},
  author={D. Garamv{\"o}lgyi and T. Jord{\'a}n},
  journal={SIAM J. Discrete Math.},
  volume={38},
  number={3},
  pages={2005--2040},
  year={2024},
}

@article{jackson2005connected,
  title={Connected rigidity matroids and unique realizations of graphs},
  author={B. Jackson and T. Jord{\'a}n},
  journal={J. Comb. Theory, Ser. B},
  volume={94},
  number={1},
  pages={1--29},
  year={2005},
}

@article{jordan2021note,
  title={A note on generic rigidity of graphs in higher dimension},
  author={T. Jord{\'a}n},
  journal={Discrete Appl. Math.},
  volume={297},
  pages={97--101},
  year={2021},
}

@article{clinch2023global,
  title={Global rigidity of 2-dimensional direction-length frameworks with connected rigidity matroids},
  author={K. Clinch},
  journal={Discrete Appl. Math.},
  volume={325},
  pages={241--261},
  year={2023},
}

@article{chenji2012toward,
  title={Toward accurate mobile sensor network localization in noisy environments},
  author={H. Chenji and R. Stoleru},
  journal={IEEE Trans. Mobile Comput. },
  volume={12},
  number={6},
  pages={1094--1106},
  year={2012},
}

@article{xiao2017noise,
  title={Noise-tolerant wireless sensor networks localization via multinorms regularized matrix completion},
  author={F. Xiao and W. Liu and Z. Li and L. Chen and R. Wang},
  journal={IEEE Trans. Veh. Technol.},
  volume={67},
  number={3},
  pages={2409--2419},
  year={2017},
}

@Book{2002nonlinear,
  author = {H. K. Khalil and J. W. Grizzle},
  title = {Nonlinear Systems},
  volume = {3},
  publisher = {NJ},
  address = {Prentice hall Upper Saddle River},
  year = {2002},
}

@article{TIE2025,
    title = {Distributed Shape Determination With Heterogeneous Sensing: Theory and Application},
    author = {X. Li  and N. Li and X. Fang and S. Zhu and X. Luo and X. Guan},
    journal = {IEEE Trans. Ind. Electron.},
    volume={106},
    number={3},
    pages={117--128},
    year = {2025},
}

@article{chen2019controlling,
  title={Controlling and stabilizing a rigid formation using a few agents},
  author={X. Chen and M. A. Belabbas and T. Basar},
  journal={SIAM J. Control Optim.},
  volume={57},
  number={1},
  pages={104--128},
  year={2019},
}

@article{suttner2018formation,
  title={Formation shape control based on distance measurements using lie bracket approximations},
  author={R. Suttner and Z. Sun},
  journal={SIAM J. Control Optim.},
  volume={56},
  number={6},
  pages={4405--4433},
  year={2018},
}

@article{peters2015sensor,
  title={Sensor network localization on the group of three-dimensional displacements},
  author={J. R. Peters and D. Borra and B. E. Paden and F. Bullo},
  journal={SIAM J. Control Optim.},
  volume={53},
  number={6},
  pages={3534--3561},
  year={2015},
}

@article{ KNN,
Author = {G. Kalai and E. Nevo  and I. Novik},
Title = {Bipartite rigidity},
Journal = {Trans. Amer. Math. Soc.},
Volume = {368},
Number = {8},
Pages = {5515-5545},
Month = {AUG},
Year = {2016},
}

@Book{Alg,
  author = {T. H. Cormen and C. E. Leiserson and R. L. Rivest and C. Stein.},
  title = {Introduction
to Algorithms.},
  address={Cambridge, Massachusetts London, England},
  publisher = {USA: MIT Press},
  year = {2009},
}
%

\end{document}